\pdfoutput=1

\newif\ifpersonal
\personalfalse  

\documentclass[10pt,reqno,a4paper,dvipsnames,x11names]{amsart}

\usepackage[utf8]{inputenc}
\usepackage[T1]{fontenc}
\usepackage[british]{babel}
\usepackage[a4paper,textwidth=170mm,textheight=239mm,left=20mm,top=33.1mm,includehead=false,includefoot=false]{geometry}

\usepackage{microtype}
\usepackage{mathpazo}
\usepackage{euler}

\usepackage{amsmath,amssymb,amsthm,mathtools,mathrsfs,bm,eucal}

\usepackage[all,cmtip]{xy}
\usepackage{tikz}
\usetikzlibrary{
 arrows.meta,
 calc,
 positioning,
 fit,
 backgrounds,
 shapes.geometric,
 decorations.markings
}
\usepackage{tikz-cd}

\tikzset{
 >=Stealth,
 box/.style={draw, rounded corners, inner sep=6pt},
 v/.style={draw, circle, inner sep=1.4pt},
 arr/.style={->, thick},
 darr/.style={->, thick, dashed},
 midarrow/.style={
 postaction={decorate},
 decoration={markings,mark=at position 0.58 with {\arrow{Stealth}}}
 }
}

\usepackage{enumitem}
\usepackage{needspace}
\usepackage{comment,braket,xspace,csquotes}

\usepackage{xcolor}
\usepackage{hyperref}
\hypersetup{
 colorlinks=true,
 pdftitle={Transversal unfoldings of derived foliations},
 pdfauthor={Mauricio Correa},
 linkcolor=blue!60!black,
 citecolor=blue!60!black,
 urlcolor=blue!60!black
}
\usepackage[capitalize,nameinlink]{cleveref}

\AtBeginDocument{%
 \SetMathAlphabet{\mathsf}{normal}{OT1}{cmss}{m}{n}%
 \SetMathAlphabet{\mathsf}{bold}{OT1}{cmss}{bx}{n}%
}

\numberwithin{equation}{section}

\theoremstyle{plain}
\newtheorem{theorem}{Theorem}[section]
\newtheorem{proposition}[theorem]{Proposition}
\newtheorem{corollary}[theorem]{Corollary}
\newtheorem{lemma}[theorem]{Lemma}
\newtheorem*{theoremA}{Theorem A}
\newtheorem*{theoremB}{Theorem B}
\newtheorem*{theoremC}{Theorem C}
\newtheorem*{theoremD}{Theorem D}
\newtheorem*{theoremE}{Theorem E}

\theoremstyle{definition}
\newtheorem{definition}[theorem]{Definition}
\newtheorem{example}[theorem]{Example}
\newtheorem{construction}[theorem]{Construction}
\newtheorem{hypothesis}[theorem]{Hypothesis}

\theoremstyle{remark}
\newtheorem{remark}[theorem]{Remark}

\newcommand{\cO}{\mathcal{O}}

\newcommand{\End}{\mathrm{End}}
\newcommand{\Aut}{\mathrm{Aut}}
\newcommand{\Hom}{\mathrm{Hom}}

\DeclareMathOperator{\Spec}{Spec}
\DeclareMathOperator{\Sym}{Sym}
\DeclareMathOperator{\Der}{Der}
\DeclareMathOperator{\RDer}{RDer}
\DeclareMathOperator{\Def}{Def}
\DeclareMathOperator{\RHom}{RHom}

\DeclareMathOperator{\fib}{fib}
\DeclareMathOperator{\Tot}{Tot}
\DeclareMathOperator{\MC}{MC}
\DeclareMathOperator{\CE}{CE}
\DeclareMathOperator{\DR}{DR}

\DeclareMathOperator{\Fol}{Fol}
\DeclareMathOperator{\Unf}{Unf}
\DeclareMathOperator{\Flat}{Flat}
\DeclareMathOperator{\Map}{Map}

\DeclareMathOperator{\id}{id}

\newcommand{\kk}{k}

\newcommand{\cF}{\mathcal F}

\newcommand{\cR}{\mathcal R}
\newcommand{\g}{\mathfrak g}
\newcommand{\h}{\mathfrak h}
\newcommand{\e}{\mathfrak e}
\newcommand{\kfr}{\mathfrak k}
\newcommand{\z}{\mathfrak z}
\newcommand{\uuder}{\mathfrak u_{\mathrm{der}}}
\newcommand{\Uder}{\mathbb U_{\mathrm{der}}}
\newcommand{\bbT}{\mathbb T}
\newcommand{\bbL}{\mathbb L}

\newcommand{\xto}[1]{\xrightarrow{#1}}
\newcommand{\dg}{\mathrm{dg}}
\newcommand{\gm}{\mathrm{gm}}
\newcommand{\bas}{\mathrm{bas}}

\newcommand{\tr}{\mathrm{tr}}
\newcommand{\str}{\mathrm{str}}
\newcommand{\eff}{\mathrm{eff}}

\newcommand{\from}{\colon}
\newcommand{\holim}{\operatorname*{holim}}
\newcommand{\hocolim}{\operatorname*{hocolim}}

\title[Transversal unfoldings of derived foliations]{Transversal unfoldings of derived foliations}
\date{}
\author[M. Corr\^ea]{Maur\'icio Corr\^ea}
\address[M. Corr\^ea]{Dipartimento di Matematica, Universit\`a degli Studi di Bari, Via E. Orabona 4, I-70125, Bari, Italy}
\email{mauricio.correa.mat@gmail.com, mauricio.barros@uniba.it}

\begin{document}

\begin{abstract}
We study when a derived foliation along the fibres of a smooth
morphism extends transversally to the total space. We classify
strict unfoldings by basic Lie derivations, retaining the
automorphisms arising from central isotropy. We define an
intrinsic space of unfoldings and prove its agreement with the
strict classification when the relative cotangent complex is a
vector bundle in degree zero. A counterexample shows that the
strict construction is not invariant under equivalent
presentations. For two-term cotangent complexes, we describe
cotangent extensions and their integrability conditions, and
obtain cohomological existence criteria for strict unfoldings
in explicit families. On smooth representable leaf spaces,
descended unfoldings correspond to integrable Ehresmann
connections.
\end{abstract}

\maketitle

\section{Introduction}

Derived foliations, in the sense of To\"en and Vezzosi, describe
possibly singular foliations by graded mixed de Rham algebras
\cite{TVRH,TV}. Their cohomology and formal leaves retain
information that is absent from the corresponding classical
foliations \cite[Chapters~2 and~3]{TV}. In the rigid
holomorphic setting, singular Frobenius theory also relates
such foliations to analytic leaf spaces under suitable
hypotheses \cite{TVRH,TVInt,TV}. We consider the problem of
transversal unfoldings of a given relative derived foliation
in a smooth family.

Let $\pi:X\to S$ be smooth and let $\cF/S$ be a relative
derived foliation. A deformation allows the leafwise structure to vary with the
parameter. An unfolding extends the foliation in the directions
of $S$: we seek an absolute foliation on $X$ restricting to
$\cF/S$, with transverse quotient given by the tangent bundle
of the base. In a Lie algebroid presentation,
this means lifting vector fields on $S$ to infinitesimal
symmetries of the relative foliation in a bracket-compatible
way. If
$\widetilde\cF$ denotes the absolute foliation of an
unfolding, transversality is expressed by the cotangent
triangle
\[
 \pi^*\Omega^1_{S/\kk}\longrightarrow
 \bbL_{\widetilde\cF}\longrightarrow\bbL_{\cF/S}
 \longrightarrow\pi^*\Omega^1_{S/\kk}[1].
\]
Dually, its tangent complex has transverse quotient
$\pi^*T_{S/\kk}$ \cite{IllusieI,TV}.

The distinction between a deformation and an unfolding is
already present in the classical theory. If a
codimension-one foliation is locally defined by an
integrable form $\omega$, an infinitesimal relative
deformation has the form $\omega_\varepsilon=\omega+\varepsilon\eta$,
whereas an unfolding contains an additional transverse term
$\widetilde\omega_\varepsilon=\omega+\varepsilon\eta+h\,d\varepsilon$.
In Suwa's theory, the term $h\,d\varepsilon$ records a lift
of the parameter direction to the total space
\cite{Suwa81,Suwa83}. The integrability of the unfolded form imposes an additional
condition on $h$, absent from the relative deformation.
In Lie-algebroid terms this is the flatness condition for
transverse Lie derivations \cite{Quallbrunn,CMQ}. A derived
foliation presents a further difficulty: a strict Lie-algebroid
model need not retain the relevant higher homotopies.

Graded mixed de Rham algebras allow the unfolding problem
to be posed without choosing such a presentation. On
$X=\Spec A$ over $S=\Spec B$, relative restriction is
\[
 C\longmapsto r_\pi C
       =C\otimes^{\mathbb L}_{\DR(B/\kk)}B.
\]
The homotopy fibre $\mathcal U_\pi(\cF)$ over
$\DR(\cF/S)$ parametrises absolute foliations together
with an identification of their relative restrictions with
$\cF/S$. It also retains automorphisms and higher homotopies.

For a specified relative dg-Lie algebroid
$\rho:\g\to T_{A/B}$, a strictly perfect model has
Chevalley--Eilenberg algebra
\[
 \CE^*(\g)=\Sym_A(\g^\vee[1]),
\]
which presents a relative derived foliation
\cite{Rinehart63,Nuiten,TV}. Basic first-order derivations of
$\g$ are infinitesimal Lie symmetries whose symbols project
to vector fields on the base. They form the dg-Lie algebra
$D^1_{\bas}(\g/B)$, obtained from the first-order operator
calculus of Atiyah and Rinehart by imposing this projectability
condition \cite{Atiyah57,Rinehart63,CMQ}. Its inner derivation
map is
\[
 \iota:\g\longrightarrow D^1_{\bas}(\g/B),\qquad
 x\longmapsto(\rho(x),\operatorname{ad}_x,0).
\]
When $\iota$ is injective, the quotient
$\uuder(\g/B)=D^1_{\bas}(\g/B)/\iota(\g)$ inherits a
dg-Lie algebroid structure over $B$. Its flat splittings
classify strict transversal unfoldings.

\begin{theoremA}
Let $B\to A$ be smooth between smooth $\kk$-algebras,
and let $\g\to T_{A/B}$ be effective, meaning that
$\iota$ is injective. Put
$\uuder(\g/B)=D^1_{\bas}(\g/B)/\iota(\g)$.
The effective strict unfolding groupoid has no
non-trivial stabilisers, and the adjoint-symbol construction
induces a bijection
\[
 |\Unf^{\tr,\eff}(\g/B)|
       \simeq\Flat(T_{B/\kk},\uuder(\g/B)).
\]
The absolute dg-Lie algebroid is reconstructed from its
basic extension by scalar extension, with the relations
identifying the original $A$-module structure on $\g$.
\end{theoremA}

The relative algebroid is an ideal in the extension of
projectable sections over $B$. In the absolute algebroid it
is a Lie subalgebroid; even $T_{A/B}$ need not be an ideal
in $T_{A/\kk}$.

If $\ker\iota\ne0$, changing a lift by a central isotropy
element leaves its derivation unchanged. The quotient
therefore loses automorphisms of the unfolding.
The crossed module
$\Uder(\g/B)=[\g\xrightarrow{\iota}D^1_{\bas}(\g/B)]$
retains them.

\begin{theoremB}
Let $B\to A$ be smooth between smooth $\kk$-algebras,
and let $\g\to T_{A/B}$ be a relative dg-Lie algebroid.
There is an equivalence of groupoids
\[
 \Unf^{\tr}_{\str}(\g/B)\simeq
 \Flat_{\str}(T_{B/\kk},\Uder(\g/B)).
\]
The right-hand side consists of basic extensions
$0\to\g\to\e\to T_{B/\kk}\to0$ equipped with maps
to $D^1_{\bas}(\g/B)$ whose derivation components are
the adjoint actions on $\g$.
\end{theoremB}

A dg-Lie algebroid realisation of the crossed module is
$\mathbb U_{\dg}=D^1_{\bas}\oplus\g[1]$.
Using cofibrant resolutions as in \cite{Nuiten}, we describe
derived sections by reduced Chevalley--Eilenberg equations.
These record both the failure of a graded lift to preserve the
differential and its failure to preserve brackets. In the
ordinary case the resulting space agrees with the strict
unfolding groupoid.

\begin{theoremC}
The space $\mathcal U_\pi(\cF)$ is invariant under
equivalence of the relative graded mixed de Rham algebra.
A strictly perfect CE presentation gives a natural map
\[
 N\Unf^{\tr}_{\str}(\g/B)\longrightarrow\mathcal U_\pi(\cF).
\]
If $\bbL_{\cF/S}$ is a vector bundle in degree zero,
the foliation determines the ordinary Lie algebroid
$\h=(\bbL_{\cF/S})^\vee$, and the comparison applied
to this ordinary reconstruction is an equivalence.
In a strictly perfect CE chart,
$\Der^{\gm}_{\bas}(\CE^*(\g))\cong D^1_{\bas}(\g/B)$,
with the inner map and its action.
\end{theoremC}

The strict crossed module can depend on the CE presentation:
Proposition~\ref{prop:strict-CE-noninvariance} exhibits
equivalent presentations whose strict crossed modules have
different cohomology. Beyond the ordinary case, lifting a
cotangent extension to a derived foliation requires a compatible
mixed structure. Theorems~\ref{thm:formal-derived-unfolding}
and~\ref{thm:homotopy-fibre-mixed-lifts} describe, respectively,
the formal deformations of a fixed intrinsic unfolding and,
for two-term cotangent complexes, the space of such mixed lifts.

When the relative cotangent complex is concentrated in degree
zero, the comparison also admits \'etale descent.

\begin{theoremD}
Let $\pi:X\to S$ be a smooth morphism of smooth schemes
and let $\cF/S$ be a relative derived foliation whose
cotangent complex is a vector bundle in degree zero.
Set $\g=(\bbL_{\cF/S})^\vee$ with its induced ordinary
relative Lie algebroid structure.
For $V=\Spec B\subset S$ and an affine étale chart
$U=\Spec A\to X_V$, put
\[
 \mathscr S_{\cF,V}(U)
  =\Map^{\mathbb R}_{/T_{B/\kk}}
           (T_{B/\kk},\mathbb U_{\dg}(\g(U)/B)).
\]
This is a one-truncated étale hypersheaf, and
\[
 \mathcal U_{\pi_V}(\cF|_{X_V})
 \simeq N\Unf^{\tr}_{\str}(\g|_{X_V}/V)
 \simeq\mathrm R\Gamma(X_V,\mathscr S_{\cF,V}).
\]
The comparisons commute with restriction of $V$ and
glue over $S$.
\end{theoremD}

Already in the two-term case, flat connections alone do not
describe every strict unfolding. Let $P$ be a vector bundle
with an integrable connection along the fibres of $X\to S$,
and consider $T_{X/S}\ltimes P[-1]$, with $P[-1]$ in
cohomological degree one.

\begin{theoremE}
Strict transversal unfoldings of $T_{X/S}\ltimes P[-1]$
are classified by integrable absolute connections on $P$
extending the relative one, together with a covariantly closed
$P$-valued one-form annihilating the relative tangent directions. The one-form represents the connecting morphism
of the derived tangent triangle. In particular, it is preserved
by equivalences of pointed intrinsic unfoldings.
\end{theoremE}

Theorem~\ref{thm:two-term-curve-obstruction} gives explicit
existence criteria for the two-term examples. For a specified
extension with abelian vertical kernel, the obstruction to a
flat splitting is likewise a cohomology class; this is distinct
from the existence of an unfolding. On a representable leaf
space, the classification has an interpretation in terms of
Ehresmann connections.

\section{Derived foliations and transversal unfoldings}\label{sec:foliation-unfoldings}

Throughout, $\kk$ is a field of characteristic zero and
$V[n]^i=V^{i+n}$. A strictly perfect complex is a bounded
complex of finite projective modules. Strictly perfect models
allow CE duality, whereas derived mapping spaces are computed
with cofibrant replacements, which need not be strictly perfect
\cite{TV,IllusieI,Nuiten}.

\subsection{Graded mixed de Rham algebras}

Unless specified otherwise, derivations of graded mixed
cdgas have weight zero and commute with the mixed
differential. These are the transverse symmetries considered
below; derivations of other weights occur in the deformation
complex of the graded mixed algebra.

A graded mixed complex is a graded complex \(E=\bigoplus_{p\in\mathbb Z}E^{(p)}\), where \(p\) denotes the weight, with an internal differential \(d\) of bidegree \((1,0)\) and a mixed differential
\[
 \epsilon\from E^{(p)}\longrightarrow E^{(p+1)}[-1]
\]
satisfying \(d^2=\epsilon^2=d\epsilon+\epsilon d=0\). A graded mixed cdga is a commutative algebra object in this category \cite[Section~1.2]{TV}.

For a cdga \(A\), the underlying graded algebra of its derived
de Rham algebra is, with the conventions of
\cite[Section~1.4]{TV} and \cite{IllusieI},
\[
 \DR(A/\kk)^{\mathrm{gr}}\simeq\Sym_A(\bbL_{A/\kk}[1])
\]
with mixed structure induced by the universal derivation
\(A\to\bbL_{A/\kk}\). If \(B\to A\) is a morphism, the
relative de Rham algebra is denoted \(\DR(A/B)\). In the
smooth underived case, its underlying graded algebra is
\(\Omega^\bullet_{A/B}[\bullet]\).

\begin{definition}
Let \(X=\Spec A\). A derived foliation on \(X\), relative to \(\kk\), is a graded mixed cdga \(\DR(\cF)\) in the sense of To\"en--Vezzosi \cite[Definition~2.1.1.1]{TV}, together with a morphism of graded mixed cdgas
\[
 \DR(A/\kk)\longrightarrow \DR(\cF),
\]
such that the canonical morphism of underlying graded cdgas
\[
 \Sym_A\bigl(\DR(\cF)^{(1)}\bigr)\longrightarrow\DR(\cF)
\]
is an equivalence. In particular the weight-zero term is equivalent
to $A$. The cotangent complex is
$\bbL_\cF=\DR(\cF)^{(1)}[-1]$; perfectness means that
this $A$-module is perfect. The canonical map is required to be an equivalence, as in
\cite[Definition~2.1.1.1]{TV}; the existence of an arbitrary
equivalence of graded algebras would not suffice.
\end{definition}

As in \cite{TV}, morphisms of foliations are opposite to
morphisms of their de Rham algebras. We work with perfect
foliations; any further connectivity or Tor-amplitude
condition is stated explicitly. Homotopy constructions use
cofibrant models, which need not be strictly perfect. The
finite CE calculations use the specified strictly perfect
models.

The structural map \(\DR(A/\kk)\to\DR(\cF)\) induces
the cotangent anchor
\[
 \bbL_{A/\kk}\longrightarrow \bbL_\cF.
\]
The higher coherences are part of the graded mixed structure;
they are not additional choices attached to a strict model
\cite[Corollary~1.6.0.4]{TV}.

\begin{definition}
Let \(B\to A\) be a morphism. A perfect relative derived
foliation over \(A/B\) is a graded mixed cdga under
\(\DR(A/B)\), with the same quasi-freeness and perfectness
conditions as for absolute foliations \cite[Section~2.1]{TV}.
Thus it is a graded mixed cdga
\[
 \DR(A/B)\longrightarrow \DR(\cF_{A/B})
\]
whose canonical morphism from the symmetric algebra on its
weight-one part
\[
 \Sym_A\bigl(\DR(\cF_{A/B})^{(1)}\bigr)
   \longrightarrow\DR(\cF_{A/B})
\]
is an equivalence and whose cotangent complex
$\bbL_{\cF_{A/B}}=\DR(\cF_{A/B})^{(1)}[-1]$ is perfect.
Morphisms and equivalences are taken in the localised category
of graded mixed algebras under \(\DR(A/B)\).
\end{definition}

\begin{example}
Assume that $\bbL_{A/B}$ is perfect. The identity map
$
 \DR(A/B)\longrightarrow \DR(A/B)$
defines the relative de Rham foliation. Its cotangent complex is \(\bbL_{A/B}\). Geometrically it is the foliation by the fibres of \(\Spec A\to\Spec B\).
\end{example}

\subsection{Transversal unfoldings of a relative foliation}

Let $B\to A$ be smooth between smooth $\kk$-algebras.
The relative restriction of an absolute foliation with de Rham
algebra $C$ is obtained by killing the positive-weight forms
pulled back from the base:
\begin{equation}\label{eq:relative-restriction-foliation}
 r_\pi C
 :=C\otimes^{\mathbb L}_{\DR(B/\kk)}B
 \simeq C\otimes^{\mathbb L}_{\DR(A/\kk)}\DR(A/B).
\end{equation}
The map $\DR(B/\kk)\to C$ factors through $\DR(A/\kk)$,
while $\DR(B/\kk)\to B$ is the augmentation. The derived
tensor product is taken among graded mixed algebras, so that
its homotopical information is retained.
Write $\Fol(A/\kk)$ and $\Fol(A/B)$ for the
$\infty$-categories of perfect absolute and relative foliations.
For a groupoid $\mathcal G$, we write $N\mathcal G$ for its
nerve and $|\mathcal G|$ for its set of isomorphism classes.

\begin{definition}\label{def:intrinsic-foliation-unfolding}
Let $R=\DR(\cF/S)$ be a perfect relative derived foliation.
The space of transversal unfoldings of $\cF/S$ is
\[
 \mathcal U_\pi(\cF)
 :=\fib_R\bigl(r_\pi:
       \Fol(A/\kk)^\simeq\longrightarrow\Fol(A/B)^\simeq\bigr),
\]
where $(-)^\simeq$ denotes the maximal $\infty$-groupoid.
Thus an object is an absolute foliation $C$, together with
an equivalence $r_\pi C\simeq R$ under $\DR(A/B)$.
Morphisms and their higher homotopies preserve this identification.
For a smooth morphism $X\to S$ the same definition uses sheaves
of graded mixed algebras on $X$.
\end{definition}

The weight-one part of \eqref{eq:relative-restriction-foliation}
gives the cotangent triangle
\begin{equation}\label{eq:intrinsic-unfolding-cotangent}
 A\otimes_B\Omega^1_{B/\kk}\longrightarrow
 \bbL_C\longrightarrow\bbL_{\cF/S}
 \longrightarrow(A\otimes_B\Omega^1_{B/\kk})[1].
\end{equation}
The triangle is established in
Theorem~\ref{prop:intrinsic-comparison}; its dual identifies
the transverse tangent quotient with $T_{S/\kk}$.
For a strict Lie-algebroid presentation, the resulting adjoint
action provides the classification discussed below.

\section{Lie algebroid presentations}

Put \(T_{A/\kk}=\Der_\kk(A,A)\). We use the usual
Lie--Rinehart description of dg-Lie algebroids
\cite{Rinehart63,Nuiten,TVRH,TV}. Thus a dg-Lie algebroid
over \(A/\kk\) consists of an \(A\)-dg-module \(\g\), a
\(\kk\)-linear dg-Lie bracket and an anchor
\[
 \rho\from \g\longrightarrow T_{A/\kk}
\]
which is \(A\)-linear, is a morphism of dg-Lie algebras, and satisfies the Lie--Rinehart identity
\[
 [x,ay]=\rho(x)(a)y+(-1)^{|x||a|}a[x,y]
\]
for homogeneous \(a\in A\) and \(x,y\in\g\). A relative dg-Lie algebroid over \(A/B\) is defined similarly with anchor landing in \(T_{A/B}\).

\begin{construction}
Let \(\g\) be a strictly perfect relative dg-Lie algebroid over \(A/B\). Its Chevalley--Eilenberg algebra gives the associated
graded mixed de Rham algebra \cite[Section~2.2]{TV}:
\[
 \CE^*(\g):=\Sym_A(\g^\vee[1]).
\]
The Chevalley--Eilenberg mixed differential is determined by the anchor and the bracket, while the internal differential is induced by the differential of \(\g\). For \(\alpha\in(\g^\vee)^0\) and \(x,y\in\g^0\), the anchor and bracket terms of the mixed differential give
\[
 \epsilon_{CE}\alpha(x,y)=
 \rho(x)(\alpha(y))
 -(-1)^{|x||y|}\rho(y)(\alpha(x))
 -\alpha([x,y]).
\]
The full graded formula is obtained by dualising the suspended bracket and applying the graded Leibniz rule. The dual of the anchor induces
\[
 \DR(A/B)\longrightarrow \CE^*(\g),
\]
so \(\CE^*(\g)\) is a relative derived foliation.
\end{construction}

\begin{definition}
A relative derived foliation \(\cF_{A/B}\) is called strictly CE-presented if there is a strictly perfect relative dg-Lie algebroid \(\g\) and an equivalence of graded mixed cdgas under \(\DR(A/B)\)
\[
 \DR(\cF_{A/B})\simeq\CE^*(\g).
\]
The strict presentation is part of the data. No claim is made that an arbitrary perfect derived foliation admits a strictly perfect cofibrant Lie-algebroid presentation.
\end{definition}

\subsection{Strict extensions}

Let \(B\to A\) be a smooth morphism between smooth commutative \(\kk\)-algebras. Write
\[
 q\from T_{A/\kk}\longrightarrow A\otimes_B T_{B/\kk}
\]
for the quotient map in
\[
 0\longrightarrow T_{A/B}\longrightarrow T_{A/\kk}
 \xto{q} A\otimes_B T_{B/\kk}\longrightarrow 0.
\]
A basic transverse section is a pair over the map \(T_{B/\kk}\to A\otimes_BT_{B/\kk}\), \(\xi\mapsto1\otimes\xi\). This map need not be injective without faithful flatness; fibre products retain the prescribed base vector field.

\begin{definition}\label{def:absolute-unfolding}
Let \(\rho\from\g\to T_{A/B}\) be a relative dg-Lie algebroid. A strict transversal unfolding of \(\g\) consists of an
absolute dg-Lie algebroid
\[
 \widetilde\rho\from\widetilde\g\longrightarrow T_{A/\kk}
\]
and an inclusion of dg-Lie algebroids
\(\g\hookrightarrow\widetilde\g\), satisfying the following
conditions, which extend the classical definition
\cite{Suwa81,Suwa83,CMQ}:
\begin{enumerate}[label=(\roman*)]
\item \(\g\) is an \(A\)-dg-Lie subalgebroid of \(\widetilde\g\), not necessarily an ideal of its underlying \(\kk\)-Lie algebra;
\item \(\widetilde\rho|_\g=\rho\), viewed as a map to \(T_{A/B}\subset T_{A/\kk}\);
\item there is an isomorphism of \(A\)-dg-modules
$
 \widetilde\g/\g\simeq A\otimes_B T_{B/\kk};$
\item under this isomorphism the induced anchor on the quotient is the identity of \(A\otimes_B T_{B/\kk}\).
\end{enumerate}
\end{definition}

\begin{lemma}\label{lem:basic-ideal}
Let \(\widetilde\g\) be a strict transversal unfolding in the sense of the preceding definition. Its basic part
\[
 \widetilde\g^{\bas}:=\widetilde\g
 \times_{A\otimes_B T_{B/\kk}}T_{B/\kk}
\]
is a \(B\)-dg-Lie algebroid and sits in an exact sequence of \(B\)-dg-Lie algebroids
\[
 0\longrightarrow\g\longrightarrow\widetilde\g^{\bas}
 \longrightarrow T_{B/\kk}\longrightarrow0.
\]
In particular, \(\g\) is a dg-Lie ideal of \(\widetilde\g^{\bas}\), although it need not be an ideal of \(\widetilde\g\).
\end{lemma}

\begin{proof}
Write \(\overline\rho=q\circ\widetilde\rho\) for the \(A\)-linear map to \(A\otimes_B T_{B/\kk}\). Let \(u\in\widetilde\g^{\bas}\) lie over \(\xi\in T_{B/\kk}\), and let \(x\in\g\). On a function \(b\in B\), the Lie-algebroid anchor identity gives
\[
 \widetilde\rho([u,x])(b)
 = [\widetilde\rho(u),\rho(x)](b)
 =-\rho(x)(\xi(b))=0,
\]
because \(\rho(x)\) is \(B\)-linear. By exactness of the sequence in Definition~\ref{def:absolute-unfolding}, \([u,x]\in\g\). Similarly, if \(v\in\widetilde\g^{\bas}\) lies over \(\eta\), then \(\widetilde\rho([u,v])\) restricts to \([\xi,\eta]\) on \(B\), so the bracket of basic elements remains basic. The internal differential preserves the basic part because the quotient has zero differential. Since elements of \(\g\) annihilate \(B\), the Lie--Rinehart identity restricts to the required one over \(B\). The surjectivity of the basic projection follows by lifting \(1\otimes\xi\) through the prescribed surjection \(\widetilde\g\to A\otimes_B T_{B/\kk}\). This proves the assertions.
\end{proof}

\label{rem:relative-not-ideal}
The quotient \(\widetilde\g/\g\) is taken among \(A\)-dg-modules;
the Lie algebroid quotient is instead
\(\widetilde\g^{\bas}/\g=T_{B/\kk}\).

An isomorphism of strict unfoldings is required to be the
identity on both the given relative algebroid \(\g\) and the
quotient \(A\otimes_B T_{B/\kk}\). We do not divide out by
automorphisms of the original foliation: in the presence of
central isotropy, an unfolding can itself have non-trivial
automorphisms.
The defining diagram is
\[
\begin{tikzcd}[column sep=large,row sep=large]
0 \arrow[r] & \g \arrow[r] \arrow[d,"\rho"'] & \widetilde\g \arrow[r] \arrow[d,"\widetilde\rho"] & A\otimes_B T_{B/\kk} \arrow[r] \arrow[d,equal] & 0 \\
0 \arrow[r] & T_{A/B} \arrow[r] & T_{A/\kk} \arrow[r,"q"] & A\otimes_B T_{B/\kk} \arrow[r] & 0 .
\end{tikzcd}
\]

\begin{definition}
Let \(\cF_{A/B}\) be a strictly CE-presented relative derived foliation represented by \(\g\). A strict CE-presented transversal unfolding of \((\cF_{A/B},\g)\) is a strict transversal unfolding \(\widetilde\g\) of \(\g\). Its image under \eqref{eq:strict-to-intrinsic-unfoldings} is the absolute derived foliation
\[
 \DR(\widetilde\cF):=\CE^*(\widetilde\g).
\]
Its cotangent complex fits into the triangle
\[
 A\otimes_B\bbL_{B/\kk}
 \longrightarrow
 \bbL_{\widetilde\cF}
 \longrightarrow
 \bbL_{\cF_{A/B}}
 \longrightarrow
 (A\otimes_B\bbL_{B/\kk})[1].
\]
\end{definition}

This is the triangle \eqref{eq:intrinsic-unfolding-cotangent};
the CE algebra restricts to \(\CE^*(\g)\) by the homotopy
pushout in \eqref{eq:relative-restriction-foliation}.

\subsection{Pointwise and de Rham foliations}\label{sec:examples}

Assume that $\bbL_{A/\kk}$ is perfect.
On $X=\Spec A$, the augmentation \(\DR(A/\kk)\to A\)
and the identity of \(\DR(A/\kk)\) give the pointwise and
de Rham foliations, respectively. They are the initial and
final derived foliations: the former has only transverse
directions, while the latter has only leafwise directions
\cite[Section~2.1]{TV}.

For a morphism $f:X=\Spec A\to Y=\Spec B$ of derived affine
schemes, the relative de Rham algebra is
\[
 \DR(A/B)=\DR(A/\kk)\otimes^{\mathbb L}_{\DR(B/\kk)}B.
\]
Its cotangent complex is \(\bbL_{A/B}\), so it defines a perfect
derived foliation precisely when this complex is perfect.
The cotangent triangle
\[
 A\otimes_B\bbL_{B/\kk}\longrightarrow \bbL_{A/\kk}\longrightarrow \bbL_{A/B}
\]
separates the cotangent directions of the base from those
along the fibres.

For a smooth morphism, the unique strict unfolding of the
relative de Rham foliation is $T_{A/\kk}$: the anchor is an
isomorphism on both the relative tangent part and the
transverse quotient, hence on the whole extension. A flat
Ehresmann connection is an additional splitting of
$T_{A/\kk}\to A\otimes_BT_{B/\kk}$, not part of the unfolding.
Here the leaf space is $\Spec B$ itself.

For the pointwise relative foliation $\g=0$, an unfolding
is a horizontal complement to $T_{A/B}$; bracket closure
makes it a flat Ehresmann connection. For a non-smooth
morphism with perfect relative cotangent complex, the de Rham
construction still defines a perfect derived foliation.
The higher cohomology of $\bbL_{A/B}$ records the additional
derived directions.

\section{Basic derivations and transverse symmetries}\label{sec:basic-derivations}

The appropriate infinitesimal symmetries are Lie derivations
whose symbols are projectable to $S$. They are described by
the first-order operator calculus of Atiyah and Rinehart
\cite{Atiyah57,Rinehart63}.

\subsection{The derivation dg-Lie algebra}

\begin{definition}
Let \(\rho\from\g\to T_{A/B}\) be a relative dg-Lie algebroid. The dg-Lie algebra \(D^1_{A/\kk}(\g)\), of first-order Lie derivations, as in \cite{Atiyah57,Rinehart63,CMQ}, consists, in degree \(n\), of pairs
\[
 (\theta,\delta),
 \qquad
 \theta\in T_{A/\kk}^n,
 \quad
 \delta\from\g\to\g[n],
\]
satisfying, for homogeneous \(a\in A\) and \(x,y\in\g\),
\[
 \delta(ax)=\theta(a)x+(-1)^{n|a|}a\delta(x),
\]
\[
 \delta([x,y])=[\delta x,y]+(-1)^{n|x|}[x,\delta y],
\]
\[
 \rho(\delta x)=[\theta,\rho(x)].
\]
The differential is the commutator with the internal differentials,
\[
 d(\theta,\delta)=([d_A,\theta],[d_\g,\delta]),
\]
which in the underived algebra case reduces to the internal differential on the endomorphism complex of \(\g\). The bracket is
\[
 [(\theta,\delta),(\theta',\delta')]
 =([\theta,\theta'],[\delta,\delta']).
\]
The symbol is the dg-Lie morphism
\[
 \sigma\from D^1_{A/\kk}(\g)\longrightarrow T_{A/\kk},
 \qquad
 (\theta,\delta)\longmapsto\theta.
\]
\end{definition}

\begin{lemma}
The complex \(D^1_{A/\kk}(\g)\) is a dg-Lie algebra and \(\sigma\) is a morphism of dg-Lie algebras.
\end{lemma}

\begin{proof}
The commutator of two first-order operators is first-order,
and its symbol is the commutator of their symbols
\cite{Atiyah57,Rinehart63}. The Leibniz identity follows by expanding \([\delta,\delta'](ax)\) and observing that the second-order terms cancel with the Koszul signs. The bracket-derivation identity follows because the commutator of two derivations of a dg-Lie algebra is again a derivation. Finally,
\[
 \rho([\delta,\delta']x)
 =[\theta,[\theta',\rho(x)]]
 -(-1)^{|\theta||\theta'|}[\theta',[\theta,\rho(x)]],
\]
which is \([[\theta,\theta'],\rho(x)]\) by the graded Jacobi identity. Hence anchor-compatibility is preserved. Compatibility with the internal differential is the same calculation with one of the two derivations equal to the differential.
\end{proof}

\begin{definition}
The inner derivation map is
$
 \iota\from \g\longrightarrow D^1_{A/\kk}(\g),
 $ given by $
 x\longmapsto (\rho(x),\operatorname{ad}_x).
$
The relative dg-Lie algebroid \(\g\) is called effective in this paper if
\(\iota\) is a monomorphism of complexes.  This means that no non-zero
relative section has both zero anchor and trivial adjoint action;
it does not require the anchor itself to be injective.
\end{definition}

Effectivity here is a property of the specified strict presentation.

\begin{lemma}
The map \(\iota\) is a morphism of dg-Lie algebras. Its kernel is the complex
\[
 \z(\g):=\{x\in\g: \rho(x)=0,\ \operatorname{ad}_x=0\}.
\]
Moreover, the image of \(\iota\) is a dg-Lie ideal of \(D^1_{A/\kk}(\g)\), and of every dg-Lie subalgebra containing it.
\end{lemma}

\begin{proof}
The Lie--Rinehart identity says that \(\operatorname{ad}_x\) is a first-order operator with symbol \(\rho(x)\). The equality
\[
 [\operatorname{ad}_x,\operatorname{ad}_y]=\operatorname{ad}_{[x,y]}
\]
is the graded Jacobi identity, and \(\rho([x,y])=[\rho(x),\rho(y)]\). The stated description of the kernel follows directly from the definition. If \((\theta,\delta)\in D^1_{A/\kk}(\g)\), then
\[
 [(\theta,\delta),(\rho(x),\operatorname{ad}_x)]
 =(\rho(\delta x),\operatorname{ad}_{\delta x}),
\]
again by the derivation property and the anchor-compatibility condition. Thus the image is an ideal.
\end{proof}

\subsection{Basic derivations}

\begin{definition}\label{def:basic-derivations}
The dg-Lie algebra of basic first-order derivations of \(\g\) over \(B\) is the sub-dg-Lie algebra
\[
 D^1_{\bas}(\g/B)\subset D^1_{A/\kk}(\g)\times T_{B/\kk}
\]
whose homogeneous sections are triples \((\theta,\delta,\xi)\), with \((\theta,\delta)\in D^1_{A/\kk}(\g)\) and \(\xi\in T_{B/\kk}\), such that
\[
 q(\theta)=1\otimes\xi.
\]
The bracket is induced by
$
 [(\theta,\delta,\xi),(\theta',\delta',\eta)]
 =([\theta,\theta'],[\delta,\delta'],[\xi,\eta]).
$
\end{definition}

Since \(\rho(\g)\subset T_{A/B}\), the inner derivation map lands in the kernel of the projection to \(T_{B/\kk}\). Explicitly, in the basic derivation algebra it is
\[
 \iota(x)=\bigl(\rho(x),\operatorname{ad}_x,0\bigr).
\]
Thus
$
 \iota\from\g\longrightarrow D^1_{\bas}(\g/B)
$
defines a morphism of dg-Lie algebras.

\begin{definition}
Suppose that \(\iota\) is injective.  We write the quotient of basic derivations by inner derivations as
\[
 \uuder(\g/B):=D^1_{\bas}(\g/B)/\iota(\g).
\]
It has an anchor
$
 a_{\uuder}\from \uuder(\g/B)\longrightarrow T_{B/\kk}
$
induced by \((\theta,\delta,\xi)\mapsto\xi\).
\end{definition}

\begin{lemma}\label{lem:base-LR-basic-derivations}
Componentwise multiplication by elements of \(B\) makes
\[
 D^1_{\bas}(\g/B)\longrightarrow T_{B/\kk},
 \qquad (\theta,\delta,\xi)\longmapsto \xi,
\]
a dg-Lie algebroid over \(B\). The image of
\(\iota\from\g\to D^1_{\bas}(\g/B)\) is a vertical dg-Lie ideal. Consequently, in the effective case, \(\uuder(\g/B)\) inherits a dg-Lie algebroid structure over \(B\), with anchor \(a_{\uuder}\).
\end{lemma}

\begin{proof}
For \(b\in B\) and \(u=(\theta,\delta,\xi)\), multiplication by \(b\) means multiplication by its image in \(A\) on the first two components and ordinary multiplication on the vector field \(\xi\). If \(v=(\theta',\delta',\eta)\), the commutator calculation in \(D^1_{A/\kk}(\g)\) gives
\[
 [u,bv]=\xi(b)v+b[u,v],
\]
with the usual Koszul sign in the dg case. This is the Lie--Rinehart identity over \(B\), and the anchor is compatible with brackets. Since inner derivations have base component zero, the image of \(\iota\) lies in the kernel of the base anchor. The ideal property was proved above, and quotienting by this vertical ideal gives the asserted structure on \(\uuder(\g/B)\).
\end{proof}

Let \(T^{\bas}_{A/\kk}\subset T_{A/\kk}\times T_{B/\kk}\) denote the Lie algebra of pairs \((Y,\xi)\) such that \(q(Y)=1\otimes\xi\). In the effective case the construction is summarised by
\[
\begin{tikzcd}[column sep=large,row sep=large]
0 \arrow[r] & \g \arrow[r,"\iota"] \arrow[d,"\rho"'] & D^1_{\bas}(\g/B) \arrow[r] \arrow[d,"\sigma"] & \uuder(\g/B) \arrow[r] \arrow[d,"a_{\uuder}"] & 0 \\
0 \arrow[r] & T_{A/B} \arrow[r] & T^{\bas}_{A/\kk} \arrow[r] & T_{B/\kk} \arrow[r] & 0 .
\end{tikzcd}
\]

\section{Classification of strict unfoldings}\label{sec:affine}\label{sec:higher}

\subsection{Effective unfoldings}

We prove the affine classification by passing from a basic
lift to its adjoint action and then reconstructing the
unfolding from the resulting flat splitting. The argument
has its classical counterpart in
\cite{Suwa81,Suwa83,Quallbrunn,CMQ}.

\begin{definition}
A flat splitting of the anchor
$
 a_{\uuder}\from \uuder(\g/B)\longrightarrow T_{B/\kk}
$
is a \(B\)-linear morphism of dg-Lie algebroids over \(B\)
\[
 s\from T_{B/\kk}\longrightarrow \uuder(\g/B)
\]
whose composite with the anchor is the identity. Equivalently,
it is a \(B\)-linear chain section preserving the Lie bracket.
The set of flat splittings is denoted
\[
 \Flat(T_{B/\kk},\uuder(\g/B)).
\]
\end{definition}

\begin{definition}\label{def:basic-part}
For an unfolding $\widetilde\g$, its basic part is the
$B$-dg-Lie algebroid of Lemma~\ref{lem:basic-ideal},
$\widetilde\g^{\bas}=\widetilde\g
\times_{A\otimes_BT_{B/\kk}}T_{B/\kk}$.
In pair notation its bracket is
$[(\widetilde\xi,\xi),(\widetilde\eta,\eta)]
=([\widetilde\xi,\widetilde\eta],[\xi,\eta])$.
Its anchor over $B$ is the projection to $\xi$; its
projectable symbol is $(\widetilde\rho(\widetilde\xi),\xi)$.
We denote the strict unfolding groupoid by
$\Unf^{\tr}_{\str}(\g/B)$, and by $\Unf^{\tr,\eff}(\g/B)$
when $\g$ is effective.
\end{definition}

\begin{lemma}\label{lem:effective-discrete}
Assume that \(\g\) is effective. Then every automorphism of a strict transversal unfolding which is the identity on \(\g\) and on the quotient \(A\otimes_B T_{B/\kk}\) is the identity. Consequently the effective unfolding groupoid is equivalent to a set.
\end{lemma}

\begin{proof}
Let \(\varphi\) be such an automorphism of \(\widetilde\g\). If \((\widetilde\xi,\xi)\in\widetilde\g^{\bas}\), then
\[
 c_\xi:=\varphi(\widetilde\xi)-\widetilde\xi\in\g .
\]
Since \(\varphi\) preserves the anchor and induces the identity on the quotient, \(\rho(c_\xi)=0\). Since \(\varphi\) is the identity on \(\g\), for every \(x\in\g\) one has
\[
 [\widetilde\xi+c_\xi,x]
 =[\varphi(\widetilde\xi),\varphi(x)]
 =\varphi([\widetilde\xi,x])
 =[\widetilde\xi,x].
\]
Thus \(\operatorname{ad}_{c_\xi}=0\). Hence \(c_\xi\in\ker\iota\), and effectivity gives \(c_\xi=0\). Therefore \(\varphi\) is the identity on the basic part. Lemma~\ref{lem:automatic-basic-reconstruction} then implies that \(\varphi\) is the identity on all of \(\widetilde\g\), because the basic part generates \(\widetilde\g\) after extension of scalars from \(B\) to \(A\).
\end{proof}

\begin{lemma}\label{lem:automatic-basic-reconstruction}
Let \(B\to A\) be smooth, and let \(\widetilde\g\) be a strict transversal unfolding of \(\g\). Write \(p_1\from\widetilde\g^{\bas}\to\widetilde\g\) for the projection from the fibre product. Then the map
\[
 \Phi_{\widetilde\g}\from
 (A\otimes_B\widetilde\g^{\bas})/\widetilde{\cR}
 \longrightarrow \widetilde\g,
 \qquad a\otimes u\longmapsto a p_1(u),
\]
where \(\widetilde{\cR}\) is generated by
\[
 a\otimes x-1\otimes ax,
 \qquad a\in A,\quad x\in\g,
\]
is an isomorphism of \(A\)-dg-Lie algebroids. Thus the basic part generates the unfolding after extension
of scalars, with exactly the displayed relations.
\end{lemma}

\begin{proof}
Surjectivity follows at once. If \(y\in\widetilde\g\), its image in \(A\otimes_BT_{B/\kk}\) may be written as a finite sum \(\sum_i a_i(1\otimes\xi_i)\). Choose pairs \(\widehat\xi_i=(\widetilde\xi_i,\xi_i)\in\widetilde\g^{\bas}\). Then
\[
 y-\sum_i a_i\widetilde\xi_i\in\g,
\]
so \(y\) lies in the image of \(\Phi_{\widetilde\g}\).
For injectivity, let \(\sum_i a_i\otimes u_i\in A\otimes_B\widetilde\g^{\bas}\) map to zero in \(\widetilde\g\). Its image in \(A\otimes_BT_{B/\kk}\) is zero. Since \(B\to A\) is smooth, hence flat, tensoring the exact sequence
\[
 0\longrightarrow\g\longrightarrow\widetilde\g^{\bas}
 \longrightarrow T_{B/\kk}\longrightarrow0
\]
with \(A\) remains exact. Therefore \(\sum_i a_i\otimes u_i\) is congruent, in \(A\otimes_B\widetilde\g^{\bas}\), to an element of \(A\otimes_B\g\). Its image in \(\widetilde\g\) is then the multiplication image in the already \(A\)-linear module \(\g\). The relations defining \(\widetilde{\cR}\) identify \(A\otimes_B\g\) with \(\g\); equivalently they quotient by the kernel of the multiplication map
\[
 A\otimes_B\g\longrightarrow\g,
 \qquad a\otimes x\longmapsto ax .
\]
Hence the original element is zero modulo \(\widetilde{\cR}\). This proves that \(\Phi_{\widetilde\g}\) is an isomorphism of complexes and of \(A\)-modules.

It remains to verify the Lie-algebroid structure. The bracket and anchor on the source are the ones obtained from the fibre-product basic extension by Lemma~\ref{lem:reconstruction-basic-extension}. Under \(\Phi_{\widetilde\g}\), the element \(a\otimes(\widetilde u,u_B)\) is sent to \(a\widetilde u\); the Lie--Rinehart formula in Lemma~\ref{lem:reconstruction-basic-extension} is the restriction of the bracket and anchor of \(\widetilde\g\) to such elements. Hence \(\Phi_{\widetilde\g}\) preserves brackets, anchors, and differentials. This proves the assertion.
\end{proof}

\begin{lemma}\label{lem:reconstruction-basic-extension}
Let \(\e\) be a \(B\)-dg-Lie algebroid fitting into an exact sequence
\[
 0\longrightarrow \g\longrightarrow\e\longrightarrow T_{B/\kk}\longrightarrow0
\]
and suppose that it is equipped with a morphism
$
 \Theta\from \e\longrightarrow D^1_{\bas}(\g/B)
$
whose restriction to \(\g\) is the inner map \(\iota\) and whose projection to \(T_{B/\kk}\) is the quotient map \(\e\to T_{B/\kk}\). We moreover require that, for every homogeneous \(u\in\e\) and \(x\in\g\), the derivation component \(\delta_u\) of \(\Theta(u)\) satisfies
\[
 \delta_u(x)=[u,x]_{\e}.
\]
Put
\[
 \sigma_\e:=\sigma\circ\Theta\from \e\longrightarrow T_{A/\kk}.
\]
Then
\[
 \widetilde\g_\e:=(A\otimes_B\e)/\cR,
 \qquad
 \cR=\langle a\otimes x-1\otimes ax\mid a\in A,\quad x\in\g\rangle,
\]
has a structure of dg-Lie algebroid over \(A\), with anchor
$
 \widetilde\rho_\e(a\otimes u)=a\sigma_\e(u),
$
and bracket
\[
 [a\otimes u,b\otimes v]
 =ab\otimes[u,v]
 +a\sigma_\e(u)(b)\otimes v
 -(-1)^{|u||v|}b\sigma_\e(v)(a)\otimes u .
\]
Moreover there is an exact sequence
\[
 0\longrightarrow\g\longrightarrow\widetilde\g_\e
 \longrightarrow A\otimes_B T_{B/\kk}\longrightarrow0,
\]
so \(\widetilde\g_\e\) is a strict transversal unfolding.
\end{lemma}

\begin{proof}
The map $\sigma_\e:\e\to T_{A/\kk}$ is a morphism of dg-Lie
algebras and satisfies
$\sigma_\e(bu)=b\sigma_\e(u)$ and
$\sigma_\e(u)(b)=a_\e(u)(b)$ for $b\in B$.
These identities make the Lie--Rinehart bracket in the statement
balanced over $B$. Its Jacobi identity follows by expanding: terms
with no differentiated coefficient give Jacobi in $\e$; terms with
one differentiated coefficient cancel by the derivation rule; and
terms with two differentiated coefficients cancel by
$\sigma_\e([u,v])=[\sigma_\e(u),\sigma_\e(v)]$.
The differential $d(a\otimes u)=a\otimes d_\e u$ is a bracket
derivation because $\sigma_\e$ is a chain map.

The anchor kills every defining relation, since
$a\rho(x)=\rho(ax)$. To check stability of the relations under
the bracket, it suffices to bracket them with $1\otimes u$.
For $r(a,x)=a\otimes x-1\otimes ax$, adjoint compatibility gives
\[
 [1\otimes u,r(a,x)]
 =r(a,[u,x])+r(\sigma_\e(u)(a),x).
\]
Multiplying the first argument by a coefficient introduces no extra
term modulo $\cR$, since the anchor of $r(a,x)$ is zero.
Also $d\,r(a,x)=r(a,d_\g x)$. Thus $\cR$ is a differential
Lie ideal of zero anchor and the formulas descend.

Flatness of $A$ over $B$ gives an exact sequence
\[
 0\longrightarrow A\otimes_B\g\longrightarrow A\otimes_B\e
 \longrightarrow A\otimes_BT_{B/\kk}\longrightarrow0.
\]
Pushing out its first map by multiplication
$A\otimes_B\g\to\g$ gives
$0\to\g\to\widetilde\g_\e\to A\otimes_BT_{B/\kk}\to0$.
In particular the copy of $\g$ is embedded, not merely an image.
The reconstructed anchor has the required transverse projection,
so this is a strict unfolding.

Finally $u\mapsto(1\otimes u,a_\e(u))$ maps $\e$ to the
basic part of $\widetilde\g_\e$. It induces the identity on the
kernel $\g$ and on the quotient $T_{B/\kk}$; the two exact
sequences therefore show that it is an isomorphism. This also
proves reconstruction of the basic extension when $B\to A$
is not faithfully flat.
\end{proof}

\begin{theorem}\label{thm:strict-affine}
Let \(B\to A\) be a smooth morphism between smooth commutative \(\kk\)-algebras, and let
\[
 \rho\from\g\longrightarrow T_{A/B}
\]
be an effective relative dg-Lie algebroid. Then the effective unfolding groupoid has no non-trivial stabilisers, and the adjoint-symbol construction induces a bijection
\[
 \left|\Unf^{\tr,\eff}(\g/B)\right|
 \simeq
 \Flat(T_{B/\kk},\uuder(\g/B)).
\]
\end{theorem}

\begin{proof}
Lemma~\ref{lem:effective-discrete} reduces the assertion to isomorphism classes. The construction is the derived analogue of the adjoint-symbol construction for unfoldings of singular Lie algebroids \cite{CMQ}.

\smallskip
\noindent\emph{From an unfolding to a splitting.}
Let
\[
 0\longrightarrow\g\longrightarrow\widetilde\g
 \longrightarrow A\otimes_B T_{B/\kk}\longrightarrow 0
\]
be a strict effective transversal unfolding. Denote by \(\widetilde\g^{\bas}\) the fibre-product basic part of Definition~\ref{def:basic-part}. Thus there is an exact sequence of \(B\)-dg-Lie algebroids
\[
 0\longrightarrow\g\longrightarrow\widetilde\g^{\bas}
 \longrightarrow T_{B/\kk}\longrightarrow0 .
\]
For \(\widehat\xi=(\widetilde\xi,\xi)\in\widetilde\g^{\bas}\), define
\[
 \Theta_{\widehat\xi}:=\bigl(\widetilde\rho(\widetilde\xi),
 \operatorname{ad}_{\widetilde\xi}|_\g,\xi\bigr).
\]
We verify that this is an element of \(D^1_{\bas}(\g/B)\). By Lemma~\ref{lem:basic-ideal}, \(\g\) is preserved by brackets with basic lifts; hence \(\operatorname{ad}_{\widetilde\xi}|_\g\) is defined on \(\g\). For homogeneous \(a\in A\) and \(x\in\g\), the Lie--Rinehart identity in \(\widetilde\g\) gives
\[
 [\widetilde\xi,ax]
 =\widetilde\rho(\widetilde\xi)(a)x
 +(-1)^{|\widetilde\xi||a|}a[\widetilde\xi,x].
\]
Hence \(\operatorname{ad}_{\widetilde\xi}|_\g\) is a first-order differential operator on \(\g\) with symbol \(\widetilde\rho(\widetilde\xi)\). Moreover, the compatibility of the anchor with the bracket gives
\[
 \rho([\widetilde\xi,x])
 =[\widetilde\rho(\widetilde\xi),\rho(x)]
\]
for every \(x\in\g\); therefore the operator is a Lie derivation in the sense defining \(D^1_{A/\kk}(\g)\). If \(\widetilde\xi\) maps to \(\xi\in T_{B/\kk}\), then
\[
 q(\widetilde\rho(\widetilde\xi))=1\otimes\xi.
\]
Thus the symbol is basic and \(\Theta_{\widehat\xi}\in D^1_{\bas}(\g/B)\).
If \(\widehat\xi=(\widetilde\xi,\xi)\) is replaced by \(\widehat\xi+x=(\widetilde\xi+x,\xi)\) with \(x\in\g\), then
\[
 \Theta_{\widehat\xi+x}-\Theta_{\widehat\xi}
 =\bigl(\rho(x),\operatorname{ad}_x,0\bigr)=\iota(x).
\]
Consequently the class of \(\Theta_{\widehat\xi}\) in \(\uuder(\g/B)\) depends only on the image \(\xi\) of \(\widehat\xi\) in \(T_{B/\kk}\). We obtain a well-defined morphism
\[
 s_{\widetilde\g}\from T_{B/\kk}\longrightarrow\uuder(\g/B),
 \qquad
 s_{\widetilde\g}(\xi)=[\Theta_{\widehat\xi}].
\]
By construction its anchor is \(\xi\), so \(a_{\uuder}\circ s_{\widetilde\g}=\id\). The map is \(B\)-linear: if \(b\in B\), then \(\rho(x)(b)=0\) for every relative element \(x\in\g\), and the Lie--Rinehart identity gives
\[
 [b\widetilde\xi,x]=b[\widetilde\xi,x].
\]
Hence \(\Theta_{b\widehat\xi}=b\Theta_{\widehat\xi}\). It is also compatible with the cohomological differentials. Since the quotient \(A\otimes_B T_{B/\kk}\) is concentrated in degree zero, \(d\widetilde\xi\in\g\), and the differential of the first-order derivation \(\Theta_{\widehat\xi}\) is \(\iota(d\widetilde\xi)\). Therefore the class of \(\Theta_{\widehat\xi}\) is closed in the quotient. Thus \(s_{\widetilde\g}\) is a morphism of dg \(B\)-modules over \(T_{B/\kk}\).

\smallskip
\noindent\emph{Flatness.}
Choose basic lifts \(\widehat\xi=(\widetilde\xi,\xi)\) and \(\widehat\eta=(\widetilde\eta,\eta)\) of \(\xi,\eta\in T_{B/\kk}\). Since the projection \(\widetilde\g^{\bas}\to T_{B/\kk}\) is a morphism of \(B\)-dg-Lie algebroids by Lemma~\ref{lem:basic-ideal}, the element
\[
 [\widetilde\xi,\widetilde\eta]-\widetilde{[\xi,\eta]}
\]
belongs to \(\g\), where \(\widehat{[\xi,\eta]}=(\widetilde{[\xi,\eta]},[\xi,\eta])\) is any basic lift of \([\xi,\eta]\). Applying the adjoint-symbol construction yields
\[
 [\Theta_{\widehat\xi},\Theta_{\widehat\eta}]
 -\Theta_{\widehat{[\xi,\eta]}}
 \in\iota(\g).
\]
After passing to \(\uuder(\g/B)\) this becomes
\[
 [s_{\widetilde\g}(\xi),s_{\widetilde\g}(\eta)]
 =s_{\widetilde\g}([\xi,\eta]).
\]
Thus \(s_{\widetilde\g}\) is a flat splitting.

\smallskip
\noindent\emph{Reconstruction.}
Conversely, let
\[
 s\from T_{B/\kk}\longrightarrow\uuder(\g/B)
\]
be a flat \(B\)-linear splitting. Form the pullback dg-Lie algebra
\[
\begin{tikzcd}[column sep=large,row sep=large]
\e_s \arrow[r] \arrow[d] & D^1_{\bas}(\g/B) \arrow[d] \\
T_{B/\kk} \arrow[r,"s"'] & \uuder(\g/B).
\end{tikzcd}
\]
Since \(\iota\) is injective in the effective case, this pullback sits in an exact sequence
\[
 0\longrightarrow\g\longrightarrow\e_s
 \longrightarrow T_{B/\kk}\longrightarrow0.
\]
For \(u\in\e_s\), the identity
\([\Theta(u),\iota(x)]=\iota(\delta_u x)\), together with
injectivity of \(\iota\), shows that
\([u,x]_{\e_s}=\delta_u x\).  Thus the map
\(\e_s\to D^1_{\bas}(\g/B)\) satisfies the
adjoint-action compatibility required in
Lemma~\ref{lem:reconstruction-basic-extension}.
The composite
\[
 \e_s\longrightarrow D^1_{\bas}(\g/B)
 \xto{\sigma} T_{A/\kk}
\]
will be denoted by \(\sigma_s\). It restricts to \(\rho\) on \(\g\), and its projection to \(A\otimes_BT_{B/\kk}\) is the tautological basic vector field.

Lemma~\ref{lem:reconstruction-basic-extension} now applies to
$(\e_s,\Theta)$. It gives the absolute algebroid
\[
 \widetilde\g_s=(A\otimes_B\e_s)/
 \langle a\otimes x-1\otimes ax:a\in A,\ x\in\g\rangle,
\]
with anchor $a\otimes u\mapsto a\sigma_s(u)$ and the
Lie--Rinehart bracket of that lemma. Its basic extension is
$\e_s$, and its transverse quotient is $A\otimes_BT_{B/\kk}$.
Thus it is the required strict unfolding.

\smallskip
\noindent\emph{The inverse correspondence.}
Starting from an unfolding \(\widetilde\g\), define
\[
 \Psi\from\widetilde\g^{\bas}\longrightarrow
 T_{B/\kk}\times_{\uuder(\g/B)}D^1_{\bas}(\g/B),
 \qquad
 (\widetilde\xi,\xi)\longmapsto
 (\xi,\Theta_{(\widetilde\xi,\xi)}).
\]
The map is injective: if two basic lifts have the same image under \(\Psi\), their difference is an element \(x\in\g\) with \(\iota(x)=0\), and effectivity gives \(x=0\). It is surjective: given \((\xi,D)\) in the pullback, choose any basic lift \(\widehat\xi=(\widetilde\xi,\xi)\) of \(\xi\). Since \([D]=s_{\widetilde\g}(\xi)\), there is a unique \(x\in\g\) with
\[
 D-\Theta_{\widehat\xi}=\iota(x),
\]
and \(\widehat\xi+x=(\widetilde\xi+x,\xi)\) maps to \((\xi,D)\). Thus \(\Psi\) is an isomorphism. Extending this isomorphism by the \(A\)-linearisation construction recovers \(\widetilde\g\), by Lemma~\ref{lem:automatic-basic-reconstruction}. Conversely, if one starts from a splitting \(s\), the basic part of the reconstructed unfolding is \(\e_s\), and its adjoint-symbol splitting is \(s\). The two constructions are therefore inverse on isomorphism classes. The constructions commute with isomorphisms of the fixed data and with restrictions for which the operator comparison is established in Section~\ref{sec:base-change}.
\end{proof}

\begin{example}\label{ex:affine-product-test}
Take \(B=\kk[t]\), \(A=\kk[t,x]\) and \(\g=T_{A/B}=A\partial_x\) with its usual anchor. The natural absolute unfolding is \(\widetilde\g=T_{A/\kk}\), and
\[
 \widetilde\g^{\bas}=A\partial_x\oplus B\partial_t
 \quad\text{as a }B\text{-module}.
\]
The relative subalgebroid is an ideal in this basic extension, whereas
\([\partial_x,x\partial_t]=\partial_t\) shows that it is not an ideal of \(T_{A/\kk}\). The basic derivation represented by \(\partial_t\) acts on \(\g\) by differentiating its coefficients in \(t\), and its class defines a flat splitting of the quotient Lie algebroid. The adjoint-symbol construction and the reconstruction in Theorem~\ref{thm:strict-affine} recover this extension. The distinction between the basic and absolute brackets is therefore already necessary in smooth transverse geometry.
\end{example}

\subsection{Unfoldings with central isotropy}

If $\iota$ has a kernel, its central elements act trivially
by derivations but can give non-trivial automorphisms of an
unfolding. To retain them, consider the crossed module
\[
 \Uder(\g/B)
 :=\bigl[\g\xto{\iota}D^1_{\bas}(\g/B)\bigr],
\]
where \(D^1_{\bas}(\g/B)\) acts on \(\g\) through its derivation component. The crossed-module identities are
\[
 \iota(D\cdot x)=[D,\iota(x)],
 \qquad
 \iota(x)\cdot y=[x,y].
\]
The kernel
\[
 \z(\g)=\ker(\iota)
 =\{x\in\g:\rho(x)=0,\ \operatorname{ad}_x=0\}
\]
is the central isotropy complex of this crossed module. The ordinary crossed-module construction is recalled in \cite{BrownSpencer76,Noohi07}; here both entries retain their internal cohomological degrees.

\begin{definition}\label{def:strict-lift-groupoid}
We denote by
\(\operatorname{Flat}_{\mathrm{str}}(T_{B/\kk},\Uder(\g/B))\)
the groupoid whose objects are the basic extensions
\[
 0\longrightarrow\g\longrightarrow\e\longrightarrow T_{B/\kk}\longrightarrow0
\]
together with a morphism
\(\Theta\from\e\to D^1_{\bas}(\g/B)\) restricting to \(\iota\), inducing the identity on the base, and preserving the differential and bracket. If \(\delta_u\) is the derivation component of \(\Theta(u)\), we require
\begin{equation}\label{eq:basic-extension-adjoint-compatibility}
 \delta_u(x)=[u,x]_{\e}
 \qquad (u\in\e,\ x\in\g).
\end{equation}
Morphisms are isomorphisms of such extensions which are the identity on \(\g\) and on \(T_{B/\kk}\) and commute with \(\Theta\).
\end{definition}

\begin{remark}
The compatibility in \eqref{eq:basic-extension-adjoint-compatibility}
is not implied by the fact that $\Theta$ is a Lie-algebroid
morphism when the inner action is not injective.  To see this,
take $A=B=\kk[t]$ and $\g=Bz$ with zero anchor and bracket.
Give $\e=T_{B/\kk}\ltimes Bz$ the action
$[\partial_t,z]_{\e}=z$.  Define $\Theta$ by sending
$\partial_t$ to the first-order derivation which differentiates
coefficients of $Bz$ and sends $z$ to zero, and by sending
$z$ to zero.  This is a morphism over $T_{B/\kk}$ and restricts
to the inner map, but its derivation component sends $z$ to zero,
whereas $[\partial_t,z]_{\e}=z$.  Thus the extension together
with this $\Theta$ is not recovered by taking its adjoint-symbol
map.  Condition \eqref{eq:basic-extension-adjoint-compatibility}
excludes precisely this discrepancy.
\end{remark}

\begin{proposition}\label{prop:central-stabiliser-complex}
Every basic extension $(\e,\Theta)$ induces a canonical flat
$T_{B/\kk}$-connection on the complex $\z(\g)$.
Its automorphisms are the degree-zero $B$-linear maps
$c:T_{B/\kk}\to\z(\g)$ satisfying
\[
 d_{\z(\g)}c=0,\qquad
 c([\xi,\eta])=\nabla_\xi c(\eta)-\nabla_\eta c(\xi).
\]
For an ordinary algebroid this gives
$\Aut(\e,\Theta)=Z^1_{\mathrm{CE}}(T_{B/\kk};\z(\g))$.
No flat section of $\e\to T_{B/\kk}$ is required.
\end{proposition}

\begin{proof}
Choose a graded $B$-linear section $\sigma:T_{B/\kk}\to\e$
and set $\nabla_\xi z=[\sigma(\xi),z]$.
Every derivation preserves $\z(\g)$, by the anchor and bracket
identities. Changing $\sigma$ by a $\g$-valued map leaves this
action unchanged because $[\g,\z(\g)]=0$.
Both $d_\e\sigma$ and
$[\sigma(\xi),\sigma(\eta)]-\sigma([\xi,\eta])$
take values in $\g$. They act trivially on $\z(\g)$, so the
action commutes with its differential and has zero curvature.

An automorphism fixing the kernel and quotient is uniquely
\[
 \Phi_c(x+\sigma(\xi))=x+\sigma(\xi)+c(\xi).
\]
The identity $\Theta\Phi_c=\Theta$ forces $\iota c=0$,
hence $c$ takes values in $\z(\g)$. Conversely, for such a $c$,
the chain-map condition is $d_{\z(\g)}c=0$.
Preservation of brackets with $\g$ is automatic, and preservation
of brackets of lifted base fields is exactly the displayed
cocycle equation; the term $[c(\xi),c(\eta)]$ vanishes.
The inverse is $\Phi_{-c}$ and composition is addition.
\end{proof}

\begin{theorem}\label{thm:higher}
Let \(B\to A\) be a smooth morphism between smooth \(\kk\)-algebras
and let \(\g\to T_{A/B}\) be a relative dg-Lie algebroid. The adjoint-symbol construction and scalar-extension reconstruction induce an equivalence of groupoids
\[
 \operatorname{Unf}^{\tr}_{\mathrm{str}}(\g/B)
 \simeq
 \operatorname{Flat}_{\mathrm{str}}
 \bigl(T_{B/\kk},\Uder(\g/B)\bigr).
\]
If $\iota$ is injective, the strict equivalence recovers
Theorem~\ref{thm:strict-affine} on isomorphism classes.
\end{theorem}

\begin{proof}
For an unfolding $\widetilde\g$, its basic part
$\widetilde\g^{\bas}$ is an extension of $T_{B/\kk}$ by $\g$.
A basic section $u=(\widetilde u,\xi)$ acts on $\g$ by the
first-order derivation $\operatorname{ad}_{\widetilde u}|_{\g}$.
Together with its symbol this gives
\[
 \Theta(u)=\bigl(\widetilde\rho(\widetilde u),
       \operatorname{ad}_{\widetilde u}|_{\g},\xi\bigr)
       \in D^1_{\bas}(\g/B).
\]
The Lie--Rinehart identity shows that the derivation has the
required symbol; Jacobi and the anchor identity show that $\Theta$
preserves the bracket.  It also commutes with the differentials and
restricts to $\iota$ on $\g$.  In particular,
$\delta_u(x)=[u,x]_{\widetilde\g^{\bas}}$, so the adjoint
compatibility in \eqref{eq:basic-extension-adjoint-compatibility}
holds.

Conversely, let $(\e,\Theta)$ be a basic extension satisfying
\eqref{eq:basic-extension-adjoint-compatibility}.
Lemma~\ref{lem:reconstruction-basic-extension} constructs the
absolute algebroid
\[
 \widetilde\g_{\e}
 =\bigl(A\otimes_B\e\bigr)/
 \langle a\otimes x-1\otimes ax : a\in A,\ x\in\g\rangle.
\]
Its anchor and Lie--Rinehart bracket extend those of $\e$.
The isomorphism of Lemma~\ref{lem:automatic-basic-reconstruction}
recovers an unfolding from its basic part.  In the opposite
direction the basic part of $\widetilde\g_{\e}$ is identified
with $\e$; the adjoint derivation of $u\in\e$ is
$\delta_u$ by \eqref{eq:basic-extension-adjoint-compatibility}.
Thus the two assignments are inverse on objects.  Both constructions
are natural for isomorphisms fixing $\g$ and $T_{B/\kk}$, and are
therefore inverse on morphisms as well.  The description of
stabilisers follows from
Proposition~\ref{prop:central-stabiliser-complex}.

\end{proof}

Central isotropy need not be an \(A\)-submodule, since the
Lie--Rinehart identity gives
\([x,az]=\rho(x)(a)z+a[x,z]\). The following ordinary
example makes this distinction explicit.

Let \(B=\kk[t]\), \(A=\kk[t,x]\) and
\[
 \g=A\partial_x\ltimes A e,
 \qquad \rho(f\partial_x+g e)=f\partial_x.
\]
The module \(Ae\) is abelian and the action of \(A\partial_x\) is the standard flat connection, so that
\[
 [f\partial_x,g e]=f\,\partial_x(g)e,
 \qquad [f e,g e]=0.
\]
These formulas, together with the bracket of vertical vector fields, define the semidirect Lie--Rinehart algebra. The kernel of the inner derivation map is
\[
 \ker\iota
 =\{g e\mid\partial_x(g)=0\}=B e.
\]
In particular, it is a non-zero \(B\)-submodule of \(\g\), but it is not generally an \(A\)-submodule: \(x e\) is not central. The basic derivation \(\partial_t\), acting coefficientwise and fixing \(e\), provides flat transport and preserves \(B e\). For a fixed strict flat lift the corresponding automorphisms are Chevalley--Eilenberg one-cocycles with values in the transported \(B\)-module \(B e\), as in Proposition~\ref{prop:central-stabiliser-complex}. The crossed module retains this isotropy, whereas the ordinary quotient of basic derivations by the inner action does not.

The classical holomorphic setting is covered by the work of
Corr\^ea--Molinuevo--Quallbrunn \cite[Definition~3.13 and
Theorem~3.16]{CMQ}. For a relative reflexive Lie algebroid
flat over the base, their sheaf of transverse symmetries is
obtained from projectable first-order Lie derivations modulo
inner derivations, under the rank assumptions of that theorem.
The adjoint-symbol map and the Lie--Rinehart reconstruction
above recover this description on the charts where those
hypotheses apply. We use the classical calculation below to
make the transverse term explicit.

\subsection{The first-order Suwa term}

Consider a codimension-one foliation locally defined by
an integrable one-form \(\omega\), with parameter $t$.
To first order an unfolded generator has the form
$
 \widetilde\omega=\omega+t\eta+h\,dt.
$
If \(\widetilde Y=\partial_t+Y\) is a lifted parameter direction tangent to the unfolded foliation, then at \(t=0\);
$
 h=-\omega(Y).
$
Thus \(h\) is the local coefficient of the adjoint-symbol class of the lift.

\begin{proposition}\label{prop:suwa-class}
The coefficient \(h\) is the local transverse term corresponding to the chosen lift of \(\partial_t\). If the conormal generator is changed to
\(\omega'=u\omega\), then the unfolded generator may be chosen so that
$
 h'=u h.
$
If the lifted vector field is changed by a vector field \(Z\) tangent to the original foliation, then
\(\omega(Z)=0\) and the coefficient \(-\omega(Y)\) is unchanged. Hence the collection of local coefficients transforms with the conormal line and represents the same transverse infinitesimal symmetry. The coefficient of \(dt\) in
\(\widetilde\omega\wedge d\widetilde\omega\), evaluated at \(t=0\), gives its first-order integrability equation.
\end{proposition}

\begin{proof}
Evaluation of $\widetilde\omega$ on $\partial_t+Y$ at $t=0$
gives $h=-\omega(Y)$. Multiplication by a unit whose value at
$t=0$ is $u$ gives $h'=uh$. Replacing $Y$ by $Y+Z$ with
$\omega(Z)=0$ leaves $h$ unchanged. 
Write $d_X$ for differentiation in the original variables.
At $t=0$ the coefficient of $dt$ in
$\widetilde\omega\wedge d\widetilde\omega$ is
\[
 h\,d_X\omega-\omega\wedge\eta+\omega\wedge d_Xh.
\]
Thus the transverse Frobenius equation is
$h\,d_X\omega=\omega\wedge(\eta-d_Xh)$.
Together with $\omega\wedge d_X\omega=0$, it is the explicit
first-order integrability condition. In a chart with the chosen
tangent lift it says that the corresponding derivation preserves
the Pfaff ideal, which is the linearised flatness equation.
All evaluations are made on the smooth parameter space before
restricting to $t=0$.
\end{proof}

The first-order thickening \(\Spec\kk[t]/(t^2)\) is used only after taking the jet of the family on the smooth base: \(\partial_t\) does not descend to a derivation of \(\kk[t]/(t^2)\) in characteristic zero. Thus we do not apply the smooth-base theorem directly to a dual-number base.

Let $(X,D)\to S$ be a smooth family with a relative simple
normal crossing divisor, regarded as log-smooth over the base
with trivial log structure \cite{Kato89}.
There is an exact logarithmic tangent sequence
\[
 0\longrightarrow T_{X/S}(-\log D)
 \longrightarrow T_X(-\log D)
 \longrightarrow\pi^*T_{S/\kk}\longrightarrow0.
\]
For a non-trivial log structure on the base, replace the base
tangent sheaf by its logarithmic tangent sheaf.

Let $\g\subset T_{X/S}(-\log D)$ be a locally free
involutive subalgebroid. Restrict strict unfoldings to anchors
in $T_X(-\log D)$. The same adjoint and scalar-extension
proofs apply to this tangent sequence. Their transverse
symmetry sheaf is
\[
 \frac{\{(Y,\xi)\in T_X(-\log D)\times\pi^{-1}T_{S/\kk}:
          d\pi(Y)=1\otimes\xi,\ [Y,\g]\subset\g\}}{\g}.
\]
It classifies logarithmic transverse actions modulo $\g$.

For $\g=T_{X/S}(-\log D)$, the unique unfolding is
the absolute logarithmic tangent algebroid, by the anchor
argument for the relative de Rham foliation. A logarithmic
Ehresmann connection is an additional splitting, and its
flatness is not automatic. 

\subsection{Flat splittings and Maurer--Cartan equations}

Fix a flat reference splitting $s_0:T_{B/\kk}\to\uuder(\g/B)$
and put $\kfr=\ker(a_{\uuder})$. A strict splitting
$s_0+\alpha$, where
$\alpha\in\Omega^1_{B/\kk}\otimes_B\kfr^0$, is flat
precisely when
\[
 d_\kfr\alpha=0,\qquad
 d_{\mathrm{CE},\nabla^{s_0}}\alpha
                      +\tfrac12[\alpha,\alpha]=0.
\]
The first equality expresses compatibility with the
differential, and the second expresses flatness. Both belong
to the total Maurer--Cartan equation. Derived sections are treated using the reduced positive-arity
Chevalley--Eilenberg complex of the following section.

\section{Derived sections and the classification space}\label{sec:derived-flat-splittings}

A section in the localisation of dg-Lie algebroids need not
be represented by a strict morphism.
We therefore work in Nuiten's localisation
$\mathrm{LieAlgd}^{\infty}_B$ \cite{Nuiten}, using the crossed
realisation of Proposition~\ref{prop:crossed-dg-realisation}.
The comparison with strict unfoldings is established in
Theorem~\ref{thm:strict-derived-comparison}.

\begin{definition}\label{def:derived-flat-space}
Let \(\mathbb U\to T_{B/\kk}\) be a dg- or \(L_\infty\)-algebroid over the base tangent algebroid. Its derived space of flat splittings is
\[
 \operatorname{Flat}^{\mathrm{der}}(T_{B/\kk},\mathbb U)
 :=
 \Map_{(\mathrm{LieAlgd}^{\infty}_{B})_{/T_{B/\kk}}}
 \bigl(T_{B/\kk},\mathbb U\bigr).
\]
Thus a point is a morphism over \(T_{B/\kk}\), and therefore already includes compatibility with the internal differential, the anchor and all brackets. Higher simplices record homotopies between such morphisms.
\end{definition}

A strict splitting determines a derived point and hence a
map $\Flat(T_{B/\kk},\mathfrak u)\to
\pi_0\operatorname{Flat}^{\mathrm{der}}(T_{B/\kk},\mathfrak u)$.
It need not be bijective: strict splittings can become homotopic,
and derived splittings need not be strictly representable.
Effectivity alone does not imply rectification.

\subsection{The reduced Chevalley--Eilenberg complex and formal splittings}
\label{subsec:formal-splitting-deformations}

Let $s:T_{B/\kk}\to\mathfrak u$ be a strict flat splitting of
a surjective dg-Lie algebroid, with vertical dg-Lie algebra
$\mathfrak k$. The adjoint action of $s$ gives a flat
connection on $\mathfrak k$ by Lie derivations. To study
sections fixing the source and the quotient, we use the
\emph{reduced} Chevalley--Eilenberg complex
\[
 \overline C^\bullet_{\mathrm{CE}}(T;\mathfrak k)
 :=\ker\!\left(C^\bullet_{\mathrm{CE}}(T;\mathfrak k)
 \longrightarrow\mathfrak k\right)
 =C^{\geq1}_{\mathrm{CE}}(T;\mathfrak k),
 \qquad T=T_{B/\kk}.
\]
The map is evaluation on the empty list of arguments.
Only the arity-zero term is removed; the remaining terms carry
the same total differential.  Write
\[
 L_s^{\mathrm{red}}
 :=\Tot\overline C^\bullet_{\mathrm{CE}}(T;\mathfrak k),
 \qquad
 D_s\alpha=d_{\mathrm{CE},\nabla^s}\alpha+(-1)^p d_{\mathfrak k}\alpha
 \quad\text{for }\alpha\text{ of CE degree }p.
\]
This is a dg-Lie \emph{subalgebra} of the full total complex:
$D_s$ preserves positive CE degree, and the bracket of cochains
of CE degrees $p,q\geq1$ has CE degree $p+q\geq2$.
Our total bracket convention is
\[
 [\alpha\otimes x,\beta\otimes y]
   =(-1)^{|x|q}(\alpha\wedge\beta)\otimes[x,y],
 \qquad \beta\in\Omega^q_{B/\kk}.
\]
Together with the displayed differential, this fixes the signs
in the curvature and change-of-section formulas below.

Zero-cochains in the full complex generate infinitesimal gauge
transformations.  They do not, in general, represent homotopies in
the derived space of \emph{sections}.  In the split abelian case, the reduced complex is the one
used in \cite[Corollary~6.16]{Nuiten}.

Let $(R,\mathfrak m)$ be a local Artinian $\kk$-algebra with residue
field $\kk$, and let $\Omega^\bullet(\Delta^n)$ denote the cdga
of polynomial differential forms on the standard $n$-simplex.  Put
\[
 \mathcal S_s(R)_n
 :=\operatorname{MC}\!\left(
 (L_s^{\mathrm{red}}\otimes_\kk\mathfrak m)
 \otimes_\kk\Omega^\bullet(\Delta^n)\right).
\]
The powers of $\mathfrak m$ give a finite nilpotent filtration,
so the associated gauge series are finite. For a dg-Lie algebroid
of bounded cohomological amplitude, this gives a nilpotent
simplicial Maurer--Cartan space. Its comparison with derived
section spaces will be established in
Theorem~\ref{thm:derived-sections-smooth-base}.

\begin{theorem}
\label{thm:formal-splitting-MC}
Assume that $B$ is smooth over a field $\kk$ of characteristic zero,
that the vertical dg-Lie algebra is cohomologically bounded, and that
$s:T\to\mathfrak u$ is a strict flat splitting.
Then:
\begin{enumerate}[label=(\roman*)]
\item $\mathcal S_s(R)$ is a Kan complex, with vertices the solutions
\[
 D_s\alpha+\tfrac12[\alpha,\alpha]=0,
 \qquad \alpha\in (L_s^{\mathrm{red}}\otimes\mathfrak m)^1.
\]
Its strict vertices are the elements
$\alpha\in C^1_{\mathrm{CE}}(T;\mathfrak k^0)\otimes\mathfrak m$
for which both the internal differential defect and the bracket
curvature vanish.  Additional components of total degree one may
occur when $\mathfrak k$ has negative cohomological degrees.
\item For $R=\kk\oplus I$, where $I^2=0$, the simplicial set is
abelian, and (after reindexing in non-negative homological degrees)
\[
 \mathcal S_s(R)\simeq
 \mathrm{DK}\!\left(
 \tau_{\leq0}((L_s^{\mathrm{red}}\otimes_\kk I)[1])\right).
\]
Hence
$
 \pi_0\mathcal S_s(R)=H^1(L_s^{\mathrm{red}})\otimes_\kk I, \ 
 \pi_j\mathcal S_s(R)=H^{1-j}(L_s^{\mathrm{red}})\otimes_\kk I
 \quad (j\geq1).
$
\item A quasi-isomorphism of reduced dg-Lie algebras induces,
after tensoring with $\mathfrak m$, a weak equivalence of their
simplicial Maurer--Cartan spaces.  The analogous statement holds
for filtered $L_\infty$ quasi-isomorphisms under the standard
hypotheses of Dolgushev--Rogers.
\end{enumerate}
\end{theorem}

\begin{proof}
The Chevalley--Eilenberg differential and the internal differential
anticommute in the totalisation, since the flat action defined by
$s$ commutes with $d_{\mathfrak k}$.  Flatness of $s$ also makes
the Chevalley--Eilenberg differential square to zero.  The usual
wedge bracket is compatible with both differentials, and the
reduced complex is stable under these operations.

The powers of $\mathfrak m$ define a finite filtration, with
$[L\otimes\mathfrak m^p,L\otimes\mathfrak m^q]
\subseteq L\otimes\mathfrak m^{p+q}$.  Thus
$L_s^{\mathrm{red}}\otimes\mathfrak m$ is nilpotent.
Getzler's construction gives the Kan property and the
Maurer--Cartan equation on vertices \cite{Getzler09}.  For a
strict correction $\alpha$ of bidegree $(1,0)$, the internal and
Chevalley--Eilenberg components have bidegrees $(1,1)$ and $(2,0)$;
these are precisely the two strict flatness equations above.
When $I^2=0$, the bracket on $L_s^{\mathrm{red}}\otimes I$ is
zero.  The abelian case of Getzler's construction is Dold--Kan,
which gives~\emph{(ii)}.  Finally, a quasi-isomorphism remains a
quasi-isomorphism on each quotient
$\mathfrak m^p/\mathfrak m^{p+1}$, because $\kk$ is a field.
The filtered Goldman--Millson theorem applies to these finite
filtrations \cite{DolgushevRogers15}, proving~\emph{(iii)}.
\end{proof}

\begin{example}
\label{ex:derived-degree-one-direction}
Let $B=\kk[t]$, $T=B\partial_t$, and
$\mathfrak u=T\ltimes\mathfrak k$, where
$\mathfrak k=Be$ is abelian, has zero anchor and differential,
and is concentrated in cohomological degree one.  The action is
$[f\partial_t,ge]=fg'e$.  The projection has a unique strict
graded splitting, since $\mathfrak k^0=0$.
The \emph{full} total Chevalley--Eilenberg complex has
$C^0(T;\mathfrak k^1)=Be$ in total degree one, and
$D_s(fe)=f'\,dt\otimes e$; consequently its first cohomology
contains the non-zero class $\kk e$.  The reduced complex removes
that zero-cochain term, and has no elements in total degree one.
Therefore
\[
 H^1(L_s^{\mathrm{red}})=0,
 \qquad H^1(\Tot C^\bullet_{\mathrm{CE}}(T;\mathfrak k))=\kk e.
\]
The full complex adds a first-order class absent from the
reduced complex. These two Maurer--Cartan problems therefore
cannot be interchanged when modelling derived \emph{sections}.
\end{example}

\subsection{Derived sections of Lie-algebroid extensions}
\label{subsec:derived-sections}

Let
\[
 0\longrightarrow\mathfrak k\longrightarrow\mathfrak u
 \xrightarrow{q}T\longrightarrow0,\qquad T=T_{B/\kk},
\]
be an extension of dg-Lie algebroids over a smooth affine base.
A graded $B$-linear splitting need not preserve either the
differential or the bracket. Both failures enter the same
Maurer--Cartan equation. We shall not assume that a flat
splitting already exists.

\begin{theorem}
\label{thm:derived-sections-smooth-base}
Suppose that $B$ is smooth over a field $\kk$ of characteristic
zero and that $q:\mathfrak u\to T$ is degreewise surjective, with
vertical dg-Lie algebra $\mathfrak k$. No boundedness, projectivity,
or nilpotence assumption is imposed on $\mathfrak k$.
Choose a graded $B$-linear section $\sigma:T\to\mathfrak u$;
such a section exists since $T$ is projective over $B$.
Put
\[
 \tau_\sigma(\xi)=d_{\mathfrak u}\sigma(\xi),\qquad
 F_\sigma(\xi,\eta)
  =[\sigma(\xi),\sigma(\eta)]-\sigma([\xi,\eta]),
\]
where $\tau_\sigma\in C^1_{\mathrm{CE}}(T;\mathfrak k^1)$ and
$F_\sigma\in C^2_{\mathrm{CE}}(T;\mathfrak k^0)$.
The adjoint action of $\sigma$ defines a degree-one derivation
$D_\sigma$ of the reduced total Chevalley--Eilenberg graded Lie
algebra
\[
 L_\sigma=\Tot\overline C^\bullet_{\mathrm{CE}}
 (T;\mathfrak k),\qquad
 \mathcal F_\sigma=F_\sigma-\tau_\sigma\in L_\sigma^2,
\]
with the totalisation convention of
Proposition~\ref{prop:AKS-bianchi}.  It satisfies
$
 D_\sigma^2=[\mathcal F_\sigma,-],\ 
 D_\sigma\mathcal F_\sigma=0.
$
There is a natural weak equivalence
\begin{equation*}
 \Map^{\mathbb R}_{(\mathrm{LieAlgd}^{\infty}_{B})_{/T}}
 (T,\mathfrak u)
 \simeq
 \bigl\{\alpha\in
  (L_\sigma\otimes\Omega^\bullet(\Delta^\bullet))^1:
  \mathcal F_\sigma+D_\sigma\alpha
       +\tfrac12[\alpha,\alpha]=0\bigr\}.
\end{equation*}
On each simplex the differential includes the de Rham
differential of $\Omega^\bullet(\Delta^n)$, with the usual
Koszul sign.  The right-hand side is the simplicial
Maurer--Cartan set of this \emph{curved} differential graded
Lie algebra.  Changing the
graded section produces a canonically isomorphic simplicial set.
If $\sigma$ is a flat dg-Lie-algebroid splitting, the curvature
vanishes and the equation reduces to the ordinary reduced
Maurer--Cartan equation.
\end{theorem}

\begin{proof}
Put $C=\operatorname{Sym}^{\geq1}_B(T[1])$, with its reduced
graded cocommutative coproduct. The Lie--Rinehart bracket enters through the differential
of the free anchored resolution; it is not represented by a
$B$-linear coderivation on $C$ alone.  The cobar resolution $Q(T)$ is freely generated, as a
graded $L_\infty$-algebroid, by $C[-1]$; its differential is the
one described in \cite[Lemma~5.13]{Nuiten}.  Since $T$ is finite
projective over $B$, it is $B$-cofibrant.  By
\cite[Propositions~5.17 and~5.22(c)]{Nuiten}, $Q(T)\to T$ is a
cofibrant replacement. Nuiten uses homological grading; here the
resolution and its Taylor coefficients are reindexed cohomologically.  The map $q$ is a fibration, so this resolution
may be used to calculate morphisms in the slice over $T$.

A graded map from the free generators of $Q(T)$ to $\mathfrak u$
is a sequence of $B$-linear Taylor coefficients
\[
 f_p:\operatorname{Sym}^p_B(T[1])\longrightarrow
                       \mathfrak u[1],\qquad p\geq1.
\]
The condition that it lie over $T$ prescribes the projection
of $f_1$ to be the identity and makes the projection of $f_p$
zero for $p>1$.  The graded section $\sigma$ determines a
reference map $f_\sigma$, supported in arity one.  Every such
graded map is uniquely $f_\sigma+\alpha$, where
\[
 \alpha\in\operatorname{Hom}_B(C,\mathfrak k[1])^0
                      =L_\sigma^1.
\]
The reference map is not assumed to commute with the differential.
The condition that $f_\sigma+\alpha$ do so is precisely the
convolution twisting-cochain equation of
\cite[(5.1) and Lemma~5.13]{Nuiten}.  Its horizontal component
vanishes because the projection to $T$ is prescribed.  The
vertical component is
\[
 \mathcal F_\sigma+D_\sigma\alpha
                  +\tfrac12[\alpha,\alpha]=0.
\]
Indeed, the arity-one term independent of $\alpha$ is
$-d_{\mathfrak u}\sigma=-\tau_\sigma$, and the arity-two term is
$[\sigma(-),\sigma(-)]-\sigma([-,-])=F_\sigma$.
The terms linear in $\alpha$ define the covariant total
Chevalley--Eilenberg differential $D_\sigma$, incorporating
$d_{\mathfrak k}$ and
$\nabla^\sigma_\xi x=[\sigma(\xi),x]$; the quadratic term is
the convolution bracket from the reduced coproduct.  The
Lie--Rinehart rule makes this identity $B$-multilinear: all
symbol terms cancel against those of the fixed horizontal map.
Conversely, a solution of the displayed equation determines a
morphism $Q(T)\to\mathfrak u$ over $T$, by the universal
property of the free resolution.  Thus the calculation works
in both directions and does not rely on the abelian extension
result of \cite[Corollary~6.16]{Nuiten}.

For completeness, the convolution differential and Jacobi identity
give
\[
 D_\sigma^2=[\mathcal F_\sigma,-],\qquad
 D_\sigma\mathcal F_\sigma=0.
\]
On sections these follow from
\[
 [\nabla^\sigma_\xi,d_{\mathfrak k}]x
       =-[\tau_\sigma(\xi),x],\quad
 ([\nabla^\sigma_\xi,\nabla^\sigma_\eta]
       -\nabla^\sigma_{[\xi,\eta]})x
       =[F_\sigma(\xi,\eta),x],
\]
$d_{\mathfrak k}F_\sigma=d_{\mathrm{CE},\nabla^\sigma}
\tau_\sigma$ and the Bianchi identity for $F_\sigma$.

It remains to verify the simplicial statement.  Nuiten's
framing is the pullback
$\mathfrak u^{\Delta^n}
 =\mathfrak u\boxtimes\Omega^\bullet(\Delta^n)$
along the constant anchor, rather than the unrestricted tensor
product; see \cite[Construction~5.23]{Nuiten}.  Its vertical
ideal is $\mathfrak k\otimes\Omega^\bullet(\Delta^n)$.
A map $Q(T)\to\mathfrak u^{\Delta^n}$ over $T$ therefore has
exactly the same fixed horizontal coefficients and vertical
twisting cochain, now including the simplicial de Rham
differential.  The construction is natural for all cofaces and
codegeneracies.  Hence the twisted Maurer--Cartan equations
identify the simplicial mapping sets degree by degree, not only
their vertices.  Corollary~5.25 of \cite{Nuiten} identifies
this simplicial mapping set with the derived mapping space.

No nilpotence of $\mathfrak k$ is required.  If $d$ bounds the
rank of $T$, then $\operatorname{Sym}^p_B(T[1])=0$ for $p>d$.
For the filtration by positive Chevalley--Eilenberg arity,
\[
 D_\sigma F^pL_\sigma\subseteq F^pL_\sigma,\quad
 [F^pL_\sigma,F^qL_\sigma]\subseteq F^{p+q}L_\sigma,\quad
 F^{d+1}L_\sigma=0.
\]
The curvature has positive arity, so all twisting equations
are finite in this filtration.  The argument does not require
boundedness or projectivity of the vertical complex.

Finally, if $\sigma'=\sigma+\beta$ for
$\beta\in C^1_{\mathrm{CE}}(T;\mathfrak k^0)$, direct expansion
gives
\[
 D_{\sigma'}=D_\sigma+[\beta,-],\qquad
 \mathcal F_{\sigma'}=\mathcal F_\sigma+D_\sigma\beta
                         +\tfrac12[\beta,\beta].
\]
Translation $\alpha\mapsto\alpha-\beta$ is an isomorphism of
curved Maurer--Cartan simplicial sets, compatible with successive
changes of section.
\end{proof}

\begin{corollary}
\label{cor:ordinary-rectification}
In Theorem~\ref{thm:derived-sections-smooth-base}, if
$\mathfrak k$ is concentrated in cohomological degree zero,
then the derived space of sections is homotopy-discrete and
\[
 \Flat(T,\mathfrak u)\ \xrightarrow{\ \sim\ }
 \pi_0\Map^{\mathbb R}_{(\mathrm{LieAlgd}^{\infty}_B)_{/T}}
       (T,\mathfrak u).
\]
Both sides are the set of solutions
\(
 F_\sigma+d_{\mathrm{CE},\nabla^\sigma}a
 +[a,a]/2=0
\)
for $a\in C^1_{\mathrm{CE}}(T;\mathfrak k)$.
No flat splitting of $\mathfrak u$ is assumed in advance.
Consequently, if the effective quotient Lie algebroid of
Theorem~\ref{thm:strict-affine} has such an ordinary kernel
and its anchor to $T$ is surjective,
its strict unfolding groupoid and its derived splitting space
are weakly equivalent.
\end{corollary}

\begin{proof}
For $\mathfrak k$ in degree zero, a cochain of total degree one
and positive Chevalley--Eilenberg degree can only have arity
one and no positive simplicial form degree.  Thus every
simplicial Maurer--Cartan solution is constant on the simplex:
the component of simplicial de Rham degree one forces its
coefficients to be constant. Hence the space is discrete.
Its equation has no internal
differential component and is exactly the strict bracket
flatness condition for $\sigma+a$.  The final statement
follows from Theorem~\ref{thm:strict-affine}, which identifies
the effective strict unfolding groupoid with this set of
splittings and has no non-trivial stabilisers.
\end{proof}

\begin{example}
\label{ex:derived-sections-higher-homotopies}
Let $B=\kk[t]$, $T=B\partial_t$, and let $E=B[m]$,
$m\geq2$, have zero internal differential and the natural
$T$-action by differentiation.  The projection
$T\ltimes E\to T$ has only one strict degree-zero splitting.
Nevertheless, its derived section space has non-trivial higher
homotopy.  Indeed, because $\wedge^2_BT=0$, the reduced
Chevalley--Eilenberg complex has only the arity-one term
$B\,dt$ with coefficients in $E$.  By the split abelian
calculation of \cite[Corollary~6.16]{Nuiten},
\[
 \pi_m\Map^{\mathbb R}_{/T}(T,T\ltimes E)
                         \simeq B\,dt\ne0.
\]
Thus the derived section space need not be a one-type, even
on a smooth affine curve.  This higher homotopy is invisible to strict degree-zero splittings.
\end{example}

\subsection{The crossed module and its derived sections}

\begin{proposition}
\label{prop:crossed-dg-realisation}
Let $B$ be a commutative $\kk$-algebra, write
$T=T_{B/\kk}$, and let $D\to T$ be a dg-Lie algebroid
over $B$.  Let
$\iota:\g\to D$ be a crossed module of dg-Lie algebras over $B$.
Assume that $\iota$ is $B$-linear, that $a_D\iota=0$,
and that $\g$ has zero anchor over $B$. Suppose $D$ acts
on $\g$ by derivations with the symbol prescribed by its
anchor, satisfying
\[
 \iota(u\cdot x)=[u,\iota(x)],\qquad
 \iota(x)\cdot y=[x,y]
\]
for homogeneous $u\in D$ and $x,y\in\g$.
There is a dg-Lie algebroid over $B$, with anchor to $T$, whose
underlying complex is the mapping cone
$
 \mathbb U_{\dg}(\g\to D)=D\oplus\g[1].
$
Writing $sx$ for $x$ regarded as an element of $\g[1]$ of degree
$|x|-1$, its differential and bracket are determined by
\begin{align*}
 d(u)&=d_Du,&d(sx)&=\iota(x)-s(d_\g x),\\
 [u,v]&=[u,v]_D,&[u,sx]&=(-1)^{|u|}s(u\cdot x),
 & [sx,sy]&=0.
\end{align*}
The anchor is the given anchor on $D$ and vanishes on $\g[1]$.
A morphism of crossed modules induces a morphism of these
dg-Lie algebroids.  If its maps on $\g$ and $D$ are
quasi-isomorphisms, the induced map on $\mathbb U_{\dg}$ is a
quasi-isomorphism.  If $\iota$ is degreewise injective, the projection
\[
 \mathbb U_{\dg}(\g\to D)\longrightarrow D/\iota(\g)
\]
is a quasi-isomorphism of dg-Lie algebroids, without a
cofibration hypothesis on $\iota$.
When $\g$ and $D$ are concentrated in degree
zero, the only non-zero cohomology objects of $\mathbb U_{\dg}$
are $\ker\iota$ in degree $-1$ and $\operatorname{coker}\iota$
in degree zero.
\end{proposition}

\begin{proof}
The differential squares to zero: on $sx$ its $D$-component is
$d_D\iota(x)-\iota(d_\g x)=0$, and its shifted component is
$s(d_\g^2x)=0$.  The formula for $[u,sx]$ is the action of $D$ on
the shifted complex $\g[1]$.  It satisfies the graded Jacobi
identity with two elements of $D$ by the representation identity;
Jacobi with two shifted elements is immediate because their
bracket vanishes.  The Leibniz identity over $B$ follows from the
symbol rule for the action of $D$ on $\g$.

For the bracket of two shifted elements, compatibility with the
differential reduces to the Peiffer identity.  For homogeneous $x,y\in\g$ the required
right-hand side is
\[
 [d(sx),sy]+(-1)^{|x|-1}[sx,d(sy)]
 =(-1)^{|x|}s[x,y]+(-1)^{|x|-1}s[x,y]=0,
\]
where the second crossed-module identity has been used.  For
$u\in D$ and $x\in\g$, the derivation rule for the differential
is the dg-module identity
\[
 d_\g(u\cdot x)=(d_Du)\cdot x
                  +(-1)^{|u|}u\cdot d_\g x
\]
together with the first crossed-module identity.  The remaining
cases are those in $D$.  Thus the displayed formulas define a
dg-Lie algebroid. 
They are natural for morphisms preserving the action and the inner
map.  Moreover, there is a short exact sequence of underlying
complexes
\[
 0\longrightarrow D\longrightarrow\mathbb U_{\dg}
 \longrightarrow\g[1]\longrightarrow0.
\]
The long exact sequence in cohomology shows that an objectwise
quasi-isomorphism of crossed modules induces a quasi-isomorphism
of $\mathbb U_{\dg}$.  If $\iota$ is injective, the projection kills the shifted
summand and sends $D$ to its quotient. It preserves the differential,
bracket and anchor; its kernel is the contractible cone of
$\id_\g$, after identifying $\iota(\g)$ with $\g$.
This proves the additional assertion as a statement about Lie
algebroids. For complexes concentrated in degree zero, the
cohomology is that of $\g\xrightarrow{\iota}D$.
\end{proof}

\begin{corollary}
\label{cor:crossed-dg-sections}
Let $B$ be smooth over $\kk$.
If $D\to T_{B/\kk}$ is surjective, the derived sections of
$\mathbb U_{\dg}(\g\to D)$ are calculated by the curved
reduced CE equations of
Theorem~\ref{thm:derived-sections-smooth-base}.
If it is not surjective, the derived section space is empty.
No nilpotence of $\g$ or $D$ is required.
The space is invariant under zigzags of objectwise
quasi-isomorphisms of crossed modules over $T_{B/\kk}$.
\end{corollary}

\begin{proof}
In the surjective case apply
Proposition~\ref{prop:crossed-dg-realisation} and
Theorem~\ref{thm:derived-sections-smooth-base}.
If a derived section exists, its underlying morphism of
complexes splits the anchor on $H^0$.
A closed representative in
$D^0\oplus\g^1$ lifts each base vector field, so $D^0\to T$
must be surjective. This proves the empty case.
An objectwise quasi-isomorphism of crossed modules gives a
quasi-isomorphism of realisations by
Proposition~\ref{prop:crossed-dg-realisation}; derived mapping
spaces preserve these equivalences.
\end{proof}

\begin{theorem}
\label{thm:strict-derived-comparison}
Let $B\to A$ be smooth between smooth $\kk$-algebras and let
$\g\to T_{A/B}$ be a relative dg-Lie algebroid.
Put $T=T_{B/\kk}$,
$D=D^1_{\bas}(\g/B)$ and
$\mathbb U=D\oplus\g[1]$, with the structure of
Proposition~\ref{prop:crossed-dg-realisation}.
There is a comparison, natural up to coherent homotopy,
\[
 N\Unf^{\tr}_{\str}(\g/B)\longrightarrow
       \Map^{\mathbb R}_{/T}(T,\mathbb U).
\]
If $D\to T$ is not surjective, both sides are empty.
Otherwise choose a graded $B$-linear section $\theta_0:T\to D$.
For a basic extension $(\e,\Theta)$ and a graded section
$\sigma:T\to\e$, put
\[
 \theta=\Theta\sigma,\qquad \tau=d_\e\sigma,\qquad
 F(\xi,\eta)=[\sigma(\xi),\sigma(\eta)]-\sigma([\xi,\eta]).
\]
Writing $\mathrm{s} x$ for the image of $x$ in $\g[1]$, its
derived section is the curved reduced Maurer--Cartan solution
\begin{equation}\label{eq:strict-extension-derived-MC}
 a_{\e,\sigma}=(\theta-\theta_0)-\mathrm{s}\tau-\mathrm{s}F.
\end{equation}
The groupoid of strict unfoldings is equivalent to the groupoid
of solutions of this form, with gauge parameters
$\mathrm{s} b\in C^1_{\mathrm{CE}}(T;\g[1]^{-1})$,
$b\in C^1_{\mathrm{CE}}(T;\g^0)$.
The comparison need not identify this groupoid with the
one-truncation of the full derived section space, as shown in
Corollary~\ref{cor:strict-derived-loops-counterexample}.
\end{theorem}

\begin{proof}
If an unfolding exists, its basic extension maps to $D$ over
$T$, so $D\to T$ is surjective.
Conversely, a point of the derived section space gives, under
the forgetful functor to the derived category of $B$-modules,
a section of the underlying anchor.
On $H^0$ this is a section of $H^0(\mathbb U)\to T$.
Every class in $H^0(\mathbb U)$ is represented by a closed
pair in $D^0\oplus\g^1$; its anchor is the anchor of the
$D^0$ component. Thus $D^0\to T$ is surjective.
This proves the empty case. Otherwise projectivity of $T$
gives $\theta_0$.

Use the curved reduced model of
Theorem~\ref{thm:derived-sections-smooth-base}.
The extension and adjoint identities give
\begin{equation}\label{eq:strict-extension-defect-identities}
 d_D\theta=\iota\tau,\quad d_\g\tau=0,\quad
 F_\theta=\iota F,\quad
 d_\g F=d_{\nabla^\theta}\tau,\quad
 d_{\nabla^\theta}F=0.
\end{equation}
The first identity is compatibility of $\Theta$ with the
differential. The second is $d_\e^2\sigma=0$.
The third follows by applying $\Theta$ to the bracket
curvature. Applying the differential to that curvature gives
the fourth; Jacobi on three lifted base fields gives the fifth.

First translate the reference section from $\theta_0$ to
$\theta$. The remaining correction is
$-\mathrm{s}\tau-\mathrm{s}F$, whose two terms have bidegrees
$(1,0)$ and $(2,-1)$.
The reference curvature is
$\iota F-\iota\tau$.
The internal part of the total differential on the first term
is $\iota\tau-\mathrm{s} d_\g\tau$; on the second it is
$-\iota F+\mathrm{s} d_\g F$.
The horizontal terms are
$-\mathrm{s} d_{\nabla^\theta}\tau$ and
$-\mathrm{s} d_{\nabla^\theta}F$.
The shifted summand has zero bracket with itself.
All components of the curved Maurer--Cartan equation
therefore vanish by
\eqref{eq:strict-extension-defect-identities}.
This proves \eqref{eq:strict-extension-derived-MC}, including
its internal differential component.
Changing $\sigma$ to $\sigma+b$ gives
\[
 \theta'=\theta+\iota b,\qquad
 \tau'=\tau+d_\g b,\qquad
 F'=F+d_{\nabla^\theta}b+\tfrac12[b,b]_{\mathrm{CE}}.
\]
These are precisely the gauge formulas with parameter
$\gamma=\mathrm{s} b$:
\[
 a^\gamma=e^{\operatorname{ad}_\gamma}a
 -\sum_{n\ge0}\frac{(\operatorname{ad}_\gamma)^n}{(n+1)!}
                                      D_{\theta_0}\gamma .
\]
The series is finite by positive CE arity.
For these parameters $[\gamma,\gamma']=0$, because the shifted
summand is abelian. A composable string of such arrows, with
cumulative parameters $\gamma_0=0,\gamma_1,\ldots,\gamma_n$,
gives a simplex by the same gauge formula over
$\Omega^\bullet(\Delta^n)$, using
$\gamma(t)=\sum_i t_i\gamma_i$ and the full total differential.
Commutativity makes this compatible with every face and
degeneracy, including the face omitting the initial vertex.
It defines the map to the simplicial Maurer--Cartan space,
including its coherences. The unary and binary components
are $\theta-\mathrm{s}\tau$ and $-\mathrm{s}F$;
$\theta_0$ is used only to express them in the affine
Maurer--Cartan coordinates.

Conversely, the Maurer--Cartan components of a solution of
the displayed form are exactly
\eqref{eq:strict-extension-defect-identities}.
On $\g\oplus T$ define the differential by
$d(x,\xi)=(d_\g x+\tau(\xi),0)$, the action of a lifted
$\xi$ on $\g$ by the derivation component of $\theta(\xi)$,
and its lifted bracket by the curvature $F$.
The identities just verified give respectively the square-zero
differential, its compatibility with brackets, Jacobi and
compatibility of $\Theta(x,\xi)=\iota(x)+\theta(\xi)$
with both operations. The derivation symbols give the
Lie--Rinehart rule. This is a basic extension, hence an
unfolding by Theorem~\ref{thm:higher}.
The gauge formulas are exactly its isomorphisms fixing $\g$
and $T$. Choosing all graded sections gives an equivalent
groupoid to the original one, since the groupoid of choices
over a fixed extension is contractible.
The construction consequently descends independently of the
choices and commutes with restriction of the fixed data.
\end{proof}

\begin{theorem}
\label{thm:ordinary-crossed-derived}
In the setting of Theorem~\ref{thm:higher}, suppose that $B$ is
smooth over $\kk$ and that $\g$ is concentrated in cohomological
degree zero. Regard the crossed
module $\g\xrightarrow{\iota}D^1_{\bas}(\g/B)$ as the dg-Lie
algebroid $\mathbb U$ over $B$ with $\g$ in degree $-1$,
$D^1_{\bas}(\g/B)$ in degree zero, differential $\iota$, and
bracket induced by the action of basic derivations on $\g$.
Then the dg-Lie algebroid $\mathbb U$ is the realisation of
Proposition~\ref{prop:crossed-dg-realisation}, and there is an
equivalence of spaces
\[
 N\operatorname{Unf}^{\tr}_{\mathrm{str}}(\g/B)
 \simeq
 \Map^{\mathbb R}_{(\mathrm{LieAlgd}^{\infty}_B)_{/T_{B/\kk}}}
 \bigl(T_{B/\kk},\mathbb U\bigr).
\]
In particular, the mapping space is a one-type.  This statement
includes non-effective ordinary Lie algebroids; neither a
separate surjectivity hypothesis nor a nilpotence condition
on the basic derivations is needed.
\end{theorem}

\begin{proof}
Put $T=T_{B/\kk}$ and $D=D^1_{\bas}(\g/B)$.  Proposition~\ref{prop:crossed-dg-realisation} supplies the
dg-Lie-algebroid structure on $D\oplus\g[1]$, including its
bracket and differential.  Its anchor is the projection $D\to T$.
If $D\to T$ is not surjective, both sides of the asserted
equivalence are empty by
Theorem~\ref{thm:strict-derived-comparison}.
Otherwise the anchor is a fibration, and $T$ is projective
as a $B$-module. 
Choose a $B$-linear section $s_0:T\to D$ of the anchor.
An $L_\infty$-section of $\mathbb U\to T$ is represented, by
Theorem~\ref{thm:derived-sections-smooth-base}, by a solution
of the reduced Maurer--Cartan equation.  Because the vertical
complex is concentrated in degrees $-1$ and zero, a total
degree-one cochain has only two components: a one-cochain
$a\in C^1_{\mathrm{CE}}(T;\ker(D\to T))$ and a two-cochain
$h\in C^2_{\mathrm{CE}}(T;\g)$.  Thus it amounts to
$s=s_0+a:T\to D$ together with $h$, satisfying
\begin{equation}\label{eq:ordinary-crossed-flatness}
 [s(\xi),s(\eta)]-s([\xi,\eta])+\iota(h(\xi,\eta))=0,
 \qquad d_{\nabla^s}h=0.
\end{equation}
Here $\nabla^s_\xi x=s(\xi)\cdot x$ is the induced action on
$\g$.  The first equation gives its curvature as an inner
derivation; the second is the Jacobi identity for the extension constructed
below. The expression $d_{\nabla^s}h$ is defined by the
indicated action, which need not be flat.

The equations reconstruct a basic extension $\e=\g\oplus T$
with $B$-module structure given by the direct sum, anchor
$\e\to T$, and bracket
\[
 [(x,\xi),(y,\eta)]_{\e}
 =\bigl([x,y]+\nabla^s_\xi y-\nabla^s_\eta x
                 -h(\xi,\eta),\,[\xi,\eta]\bigr).
\]
The two equations in~\eqref{eq:ordinary-crossed-flatness} and
the crossed-module identities verify Jacobi; the symbol
identities give the Lie--Rinehart rule.  The map
$\Theta(x,\xi)=\iota(x)+s(\xi)$ is a morphism to $D$ and
its derivation component agrees with the adjoint action on
$\g$.  Conversely, choosing a $B$-linear section of a basic
extension $(\e,\Theta)$ gives $s=\Theta|_T$ and
$h=-([\widehat\xi,\widehat\eta]-\widehat{[\xi,\eta]})$.
The two constructions are inverse after a section is chosen.

Changing the chosen section by
$\beta\in C^1_{\mathrm{CE}}(T;\g)$ gives, with the
sign convention in \eqref{eq:ordinary-crossed-flatness},
\[
 s'=s+\iota\beta,\qquad
 h'=h-d_{\nabla^s}\beta-\tfrac12[\beta,\beta]_{\mathrm{CE}}.
\]
Indeed $[\widehat\xi+\beta(\xi),
\widehat\eta+\beta(\eta)]-
(\widehat{[\xi,\eta]}+\beta([\xi,\eta]))$
expands to $-h+d_{\nabla^s}\beta+
[\beta(\xi),\beta(\eta)]$.
These formulas are also the degree-zero gauge transformations of
the reduced Maurer--Cartan complex.  They coincide with the
isomorphisms of basic extensions fixing $\g$ and $T$.
There are no cochains of negative total degree: the vertical
complex starts in degree $-1$ and the reduced complex starts
in Chevalley--Eilenberg degree one. Nilpotence follows from
the finite positive CE arity, as in
Theorem~\ref{thm:derived-sections-smooth-base}.
To apply the Deligne groupoid description to the curved
equations, first note that both spaces are empty if there
is no solution. Otherwise, let $\alpha$ denote a full
Maurer--Cartan solution, including both components.
Twisting replaces $D_{s_0}$ by $D_{s_0}+[\alpha,-]$ and
makes the curvature zero. The resulting dg-Lie algebra is still nilpotent
and concentrated in non-negative total degrees. The
Deligne--Getzler comparison therefore identifies its
Maurer--Cartan space with the nerve of its gauge
groupoid \cite{Getzler09}. Translation by $\alpha$ identifies
this with the original curved problem. The gauge formulas
above give precisely the isomorphisms of extensions, so
this is an equivalence on spaces, including their loop
groups. No flat reference section or nilpotence of
$\g$ or $D$ has been assumed.
Finally, Theorem~\ref{thm:higher} identifies the basic-extension
groupoid with the groupoid of strict unfoldings.  This proves
the asserted equivalence.
\end{proof}

\section{Unfoldings of graded mixed de Rham algebras}\label{sec:intrinsic}

The classification by basic derivations refers to a fixed
Lie-algebroid presentation. To compare it with the intrinsic
space of unfoldings, we begin with a strictly perfect
Chevalley--Eilenberg presentation $R=\CE^*(\g)$ and examine
its graded mixed derivations.

\begin{definition}\label{def:basic-gm-der}
Write $\Der^{\gm}(R)$ for the complex of weight-zero derivations
commuting with the mixed differential. Its differential is the
commutator with the internal differential and its bracket is
the commutator. It is a dg-Lie algebra over $\kk$, with symbol
$\sigma:\Der^{\gm}(R)\to T_{A/\kk}$.
In general it is not an $A$-module, since multiplication by a
function need not preserve the mixed-derivation condition.

Put $M=\Der_\kk(B,A)=A\otimes_BT_{B/\kk}$. The strict complex
of basic derivations is the ordinary fibre product
\[
 \Der^{\gm}_{\bas}(R)
 =\Der^{\gm}(R)\times_M T_{B/\kk},
\]
where the maps record restriction of the symbol and the action
of a base vector field. Its sections are pairs $(\Theta,\xi)$
with $\sigma(\Theta)(b)=\xi(b)$ for $b\in B$.
Their bracket is $([\Theta,\Theta'],[\xi,\eta])$.
With the induced $B$-module structure this is a dg-Lie algebroid
over $B$.

The homotopy fibre product of the underlying complexes will
instead be denoted
\[
 \mathfrak P(R)=\Der^{\gm}(R)\times^h_M T_{B/\kk}.
\]
It need not agree with the strict fibre product. Since $M$
has no prescribed Lie bracket, this formula alone does not
endow $\mathfrak P(R)$ with a Lie-algebroid structure.
\end{definition}

\begin{construction}\label{cons:inner-action}
In a strict CE chart contraction with a homogeneous
$x\in\g$ is a derivation $\imath_x$ of weight $-1$ and
cohomological degree $|x|+1$.
The Cartan Lie derivative $L_x=[\epsilon_{CE},\imath_x]$
has weight zero and symbol $\rho(x)$. Thus its base symbol is
zero and it defines the inner map
$\iota:\g\to\Der^{\gm}_{\bas}(R)$.
The action on $\g$ is the one obtained by dualising the
action on the weight-one generators.
\end{construction}

\begin{definition}\label{def:intrinsic-basic-derivations}
For a relative foliation with this specified CE chart, write
\[
 \Uder(\cF/S;\g)=
 [\,\bbT_{\cF/S}\xrightarrow{\iota}
                  \Der^{\gm}_{\bas}(\DR(\cF/S))\,].
\]
When the chart is fixed we suppress $\g$ from the notation.
Proposition~\ref{prop:crossed-dg-realisation} gives its strict
dg-Lie-algebroid realisation. Its underlying complex is the
mapping cone of the inner map, denoted
$\uuder^{h,\mathrm{mod}}(\cF/S;\g)$.
If the inner map is injective, the projection to its ordinary
cokernel is a quasi-isomorphism of dg-Lie algebroids by that
proposition. Dependence on a strict CE model is demonstrated in
Proposition~\ref{prop:strict-CE-noninvariance}.
\end{definition}

\begin{theorem}
\label{prop:intrinsic-comparison}
Let $B\to A$ be smooth between smooth $\kk$-algebras.
Relative restriction \eqref{eq:relative-restriction-foliation}
has cotangent complex
\[
 \bbL_{r_\pi C}\simeq
 \operatorname{cofib}\bigl(A\otimes_B\Omega^1_{B/\kk}
                              \longrightarrow\bbL_C\bigr).
\]
Consequently an unfolding has the cotangent triangle
\eqref{eq:intrinsic-unfolding-cotangent}, and an equivalence
of relative foliations induces an equivalence of their
spaces $\mathcal U_\pi$ with all homotopies. 
Suppose $R=\CE^*(\g)$ is a strictly perfect relative CE
presentation. Taking the CE algebra of an extension gives
a natural comparison
\begin{equation}\label{eq:strict-to-intrinsic-unfoldings}
 N\Unf^{\tr}_{\str}(\g/B)
                 \longrightarrow\mathcal U_\pi(\cF).
\end{equation}
If $\bbL_{\cF/S}$ is a vector bundle in degree zero, the
foliation determines the ordinary relative Lie algebroid
$\h=(\bbL_{\cF/S})^\vee$.
The comparison \eqref{eq:strict-to-intrinsic-unfoldings}
applied to this ordinary reconstruction $\h$ is an equivalence.
Its anchor may have singularities and isotropy.
These assertions hold for sheaves on smooth schemes and
commute with restriction.
For an arbitrary strictly perfect CE presentation, the
strict basic derivations satisfy
\[
 \Der^{\gm}_{\bas}(\CE^*(\g))
                         \cong D^1_{\bas}(\g/B).
\]
This identifies the differential, bracket, symbol, inner map
and action, hence the crossed module
$[\g\to D^1_{\bas}(\g/B)]$. In the effective case its strict
quotient is $\uuder(\g/B)$.
\end{theorem}

\begin{proof}
The relative de Rham algebra is the homotopy pushout
$\DR(A/\kk)\otimes^{\mathbb L}_{\DR(B/\kk)}B$.
Associativity of homotopy pushouts gives
\eqref{eq:relative-restriction-foliation}. The forgetful
functor on graded mixed complexes is conservative,
symmetric monoidal and preserves colimits. A homotopy
pushout of commutative algebras is computed by the
two-sided simplicial bar construction, whose terms use
tensor products. Forgetting the mixed structure therefore
preserves this pushout as well. The underlying
graded algebras are freely generated in weight one, and the
map on those generators is
$A\otimes_B\Omega^1_{B/\kk}[1]\to\bbL_C[1]$.
The free commutative algebra functor preserves homotopy
pushouts. It follows that
\[
 (r_\pi C)^{\mathrm{gr}}\simeq
 \Sym_A\!\left(
  \operatorname{cofib}(A\otimes_B\Omega^1_{B/\kk}
                                \to\bbL_C)[1]\right).
\]
The cofibre is perfect. If $\bbL_C$ is connective, then so
is this cofibre, since the base term has degree zero.
This proves that relative restriction stays in the
foliation category and gives the asserted cotangent
complex. The construction is made in the localised category
of graded mixed algebras. Composing the identification
$r_\pi C\simeq R$ with an equivalence $R\simeq R'$ gives
an equivalence of the homotopy fibres defining
$\mathcal U_\pi(R)$ and $\mathcal U_\pi(R')$.
Composition supplies its coherences; no strict inverse
of a chosen CE morphism is needed.

For a strict unfolding $\widetilde\g$, dualising its
degreewise split module extension gives
\[
 0\longrightarrow A\otimes_B\Omega^1_{B/\kk}
 \longrightarrow\widetilde\g^\vee
 \longrightarrow\g^\vee\longrightarrow0.
\]
Here $\widetilde\g$ is strictly perfect because the
relative term and the transverse quotient are.
The cofibre of its first arrow is $\g^\vee$.
By the first part of the proof, the canonical map
$r_\pi\CE^*(\widetilde\g)\to\CE^*(\g)$ is therefore an
equivalence on the underlying graded algebras: on
weight-one cofibres it is the identity of $\g^\vee[1]$.
This map also respects both differentials:
the base one-forms generate a mixed-stable ideal, and
the quotient anchor and bracket are those of $\g$.
This proves \eqref{eq:strict-to-intrinsic-unfoldings}.
An isomorphism of extensions induces the compatible CE
isomorphism, so the construction includes arrows.

Suppose now that $\bbL_{\cF/S}$ is a vector bundle in
degree zero. In \eqref{eq:intrinsic-unfolding-cotangent},
both the first and the third terms are vector bundles
in degree zero. The cohomology sequence forces
$\bbL_C$ to have no other cohomology and gives
\[
 0\longrightarrow A\otimes_B\Omega^1_{B/\kk}
 \longrightarrow\bbL_C\longrightarrow
 \bbL_{\cF/S}\longrightarrow0.
\]
It is locally split, hence $\bbL_C$ is also a vector
bundle. The equivalence between smooth derived
foliations and ordinary vector-bundle Lie algebroids
\cite[Section~1.3.1]{TVRH} now applies to both terms.
Concretely the mixed differential on functions is
dual to the anchor; on weight-one generators it is
dual to the bracket. Its square gives the anchor,
Jacobi and Lie--Rinehart identities.
The relative mixed differential annihilates $B$, so
the reconstructed relative anchor takes values in
$T_{A/B}$. Put $\h=(\bbL_{\cF/S})^\vee$ and
$\widetilde\h=\bbL_C^\vee$. The dual exact
sequence is
\[
 0\longrightarrow\h\longrightarrow\widetilde\h
 \longrightarrow A\otimes_BT_{B/\kk}\longrightarrow0.
\]
The structural map from $\DR(A/\kk)$ identifies the
quotient anchor with the identity, so this is a strict
transversal unfolding. Its CE algebra recovers $C$.

The smooth CE correspondence is fully faithful as an
$\infty$-functor: its mapping spaces over the fixed
scheme are discrete and are the ordinary Lie algebroid
morphisms \cite[Section~1.3.1]{TVRH}.
The rigidification to $R$ therefore identifies precisely
the isomorphisms fixing $\h$ and the transverse
quotient, including their automorphisms. Thus the two
constructions are inverse on the full homotopy fibre,
not only on its objects or connected components.
Dualisation, the anchor and bracket formulas, and
the pushout all commute with restriction.
Their local inverses glue on schemes by descent for
vector bundles and graded mixed algebras.

Finally, set $R=\CE^*(\g)=\Sym_A(\g^\vee[1])$.
A homogeneous weight-zero derivation $\Theta$ of degree $n$
is determined by
$\Theta_0=\Theta|_A:A\to A[n]$ and
$\Theta_1=\Theta|_{\g^\vee[1]}:\g^\vee[1]\to\g^\vee[1][n]$.
The first component is a derivation $\theta\in T_{A/\kk}^n$.
Dualising and desuspending the second, using strict perfection,
gives a $\kk$-linear operator $\delta:\g\to\g[n]$ with
\[
 \delta(ax)=\theta(a)x+(-1)^{n|a|}a\delta(x).
\]
Conversely, this Leibniz rule extends such a pair
$(\theta,\delta)$ uniquely to a weight-zero derivation of $R$.

The mixed condition $[\Theta,\epsilon_{CE}]=0$ can be
checked on functions and weight-one generators. On a function
$a\in A$, evaluation at $x\in\g$ gives
\[
 \theta(\rho(x)(a))-\rho(\delta x)(a)
       -(-1)^{|\Theta||x|}\rho(x)(\theta(a))=0,
\]
or equivalently $\rho(\delta x)=[\theta,\rho(x)]$.
On a linear generator $\alpha\in\g^\vee$, evaluation at
homogeneous $x,y\in\g$ gives
\[
 \delta([x,y])=[\delta x,y]+(-1)^{|\delta||x|}[x,\delta y].
\]
These are precisely the anchor and bracket conditions defining
$D^1_{A/\kk}(\g)$. The internal differential acts on the
pairs by $([d_A,\theta],[d_\g,\delta])$, as in that complex;
commutation with the mixed differential does not require
$\Theta$ to be closed for the internal differential.
Commutators also agree under this correspondence.

A basic derivation includes a prescribed base vector field
$\xi$, with $q(\theta)=1\otimes\xi$. Thus the correspondence
identifies $(\Theta,\xi)$ with $(\theta,\delta,\xi)$ and gives
\[
 \Der^{\gm}_{\bas}(\CE^*(\g))\cong D^1_{\bas}(\g/B)
\]
as dg-Lie algebroids over $B$, preserving symbols.
Contraction with $x\in\g$ gives the Cartan Lie derivative
$L_x=[\epsilon_{CE},\imath_x]$. Its symbol is $\rho(x)$,
and dualisation gives the operator $\operatorname{ad}_x$ on
$\g$. It therefore corresponds to
\[
 \iota(x)=(\rho(x),\operatorname{ad}_x,0).
\]
The action on $\g$ is likewise the derivation component
$\delta$, so the inner maps and their actions are identified.
This gives the crossed module $[\g\to D^1_{\bas}(\g/B)]$.
If $\iota$ is injective,
Proposition~\ref{prop:crossed-dg-realisation} identifies its
cone realisation with $\uuder(\g/B)$ by a quasi-isomorphism
of dg-Lie algebroids.
\end{proof}

\begin{example}\label{ex:strict-derived-symbol}
Take $B=\kk[t]$, $A=B[x]$ and $\g=T_{A/B}=A\partial_x$.
Anchor compatibility identifies a first-order Lie derivation of
$\g$ with a vector field normalising $A\partial_x$.
Writing it as $f\partial_x+h\partial_t$, the normaliser condition
is $\partial_xh=0$. Hence $
 \Der^{\gm}(R)=A\partial_x\oplus B\partial_t,\qquad
 \Der^{\gm}_{\bas}(R)=A\partial_x\oplus B\partial_t.
$ 
Both maps defining the homotopy fibre product have image in
$B\partial_t\subset A\partial_t$. The cone model
$\mathfrak P(R)=\operatorname{Cone}
(\Der^{\gm}(R)\oplus B\partial_t\to A\partial_t)[-1]$
therefore gives
\[
 H^0\mathfrak P(R)=A\partial_x\oplus B\partial_t,\qquad
 H^1\mathfrak P(R)=(A/B)\partial_t\ne0.
\]
The strict complex is concentrated in degree zero.
Thus its identification with a homotopy pullback cannot be
inferred from strict perfection of the CE chart.
\end{example}

\subsection{An intrinsic derived Lie-algebroid comparison}
\label{subsec:derived-lie-fibre}

The strict crossed module need not be invariant under a
change of CE presentation. The Chevalley--Eilenberg
correspondence nevertheless gives a description of the
intrinsic unfolding space for perfect foliations, without
choosing such a crossed module.

Let $\mathsf{Lie}^{\mathrm{perf}}_{A/\kk}$ and
$\mathsf{Lie}^{\mathrm{perf}}_{A/B}$ denote the
$\infty$-categories of perfect derived Lie algebroids over
$A/\kk$ and $A/B$, respectively. We use the
Chevalley--Eilenberg equivalences of
\cite[Theorem~2.2.0.7]{TV}, with the opposite variance
of the foliation category. Define the derived restriction
functor $\mathsf r^{\mathrm{Lie}}_\pi$, up to its canonical
coherent equivalence, by requiring the diagram
\[
\begin{tikzcd}[column sep=large]
 (\mathsf{Lie}^{\mathrm{perf}}_{A/\kk})^{\mathrm{op}}
 \arrow[r,"\CE_\kk","\sim"']
 \arrow[d,"(\mathsf r^{\mathrm{Lie}}_\pi)^{\mathrm{op}}"']
 & \Fol^{\mathrm{perf}}(A/\kk)
 \arrow[d,"r_\pi"]\\
 (\mathsf{Lie}^{\mathrm{perf}}_{A/B})^{\mathrm{op}}
 \arrow[r,"\CE_B","\sim"']
 & \Fol^{\mathrm{perf}}(A/B)
\end{tikzcd}
\]
to commute in the $\infty$-category of $\infty$-categories.
Here $r_\pi$ is the derived base change of
\eqref{eq:relative-restriction-foliation}; it preserves
perfectness by the cotangent cofibre calculation of
Theorem~\ref{prop:intrinsic-comparison}. The notation
$\mathsf r^{\mathrm{Lie}}_\pi$ does not mean the ordinary
kernel of the projection of anchors. In particular it does
not invoke a strict dg realisation of
$[\g\to D^1_{\bas}(\g/B)]$.

\begin{theorem}\label{thm:derived-lie-fibre}
Let $B\to A$ be a smooth morphism between smooth
$\kk$-algebras, where $\kk$ has characteristic zero, and
let $R=\DR(\cF/S)$ be a perfect relative derived foliation.
Write $\g_R$ for its corresponding object of
$\mathsf{Lie}^{\mathrm{perf}}_{A/B}$. There is an equivalence
of spaces, natural in the relative foliation,
\begin{equation}\label{eq:derived-lie-fibre}
 \mathcal U_\pi(\cF)\simeq
 \operatorname{fib}_{\g_R}\!\left(
 (\mathsf{Lie}^{\mathrm{perf}}_{A/\kk})^\simeq
 \xrightarrow{\mathsf r^{\mathrm{Lie}}_\pi}
 (\mathsf{Lie}^{\mathrm{perf}}_{A/B})^\simeq
 \right).
\end{equation}
The relative cotangent complex need not be concentrated in
degree zero. The equivalence retains all automorphisms and
higher homotopies. 
For every point $(C,r_\pi C\simeq R)$ of this space,
write $\bbL_{\widetilde\cF}$ for the cotangent complex
of $C$. Then there is a natural equivalence of perfect
$A$-modules
\begin{equation}\label{eq:normal-derived-comparison}
 \operatorname{fib}\!\left(
   \Omega^1_{A/\kk}\longrightarrow\bbL_{\widetilde\cF}
 \right)
 \simeq
 \operatorname{fib}\!\left(
   \Omega^1_{A/B}\longrightarrow\bbL_{\cF/S}
 \right).
\end{equation}
Consequently the normal complex of To\"en--Vezzosi is
unchanged as a perfect complex by a transversal unfolding,
even when the relative cotangent complex has non-trivial
negative cohomology.
\end{theorem}

\begin{proof}
By \cite[Theorem~2.2.0.7]{TV}, the two horizontal
Chevalley--Eilenberg functors above are equivalences.
Transport of $r_\pi$ through these equivalences defines
$\mathsf r^{\mathrm{Lie}}_\pi$ and gives an equivalence
between their diagrams of maximal subgroupoids. Homotopy
fibres of equivalent diagrams of spaces are equivalent.
Taking the fibre at $R$, together with its specified
identification, proves \eqref{eq:derived-lie-fibre}.
The construction takes place in localised
$\infty$-categories: replacing a CE presentation by an
equivalent one changes neither the fibre nor its higher
homotopies.

For the second statement consider the canonical cotangent
triangle of the smooth morphism and that of the unfolding:
\[
\begin{tikzcd}[column sep=large,row sep=large]
 A\otimes_B\Omega^1_{B/\kk}\arrow[r]\arrow[d,equal]
 &\Omega^1_{A/\kk}\arrow[r]\arrow[d]
 &\Omega^1_{A/B}\arrow[d]\\
 A\otimes_B\Omega^1_{B/\kk}\arrow[r]
 &\bbL_{\widetilde\cF}\arrow[r]
 &\bbL_{\cF/S}.
\end{tikzcd}
\]
Both rows are cofibre sequences, the lower one by
Theorem~\ref{prop:intrinsic-comparison}. Their
commutativity follows from the structural morphism
$\DR(A/\kk)\to C$ and its derived relative restriction.
Taking the fibres of the middle and right vertical arrows,
the identity on the left terms gives an equivalence of those
fibres by the $3\times3$ lemma in the stable category of
$A$-modules. This is
\eqref{eq:normal-derived-comparison}, and is compatible
with morphisms of pointed unfoldings.
\end{proof}

\subsection{Formal deformations of a fixed unfolding}

Fixing an intrinsic unfolding allows us to study its formal
deformations through derived derivations of graded mixed
algebras, without choosing a strict CE presentation.

Let $R=\DR(\cF/S)$ be perfect, and fix an unfolding
$(C,\alpha)$, where $\alpha:r_\pi C\simeq R$. All deformation
problems below keep $A$ and $B$ fixed. Write
\[
 \mathfrak d_{\kk}(C)
  :=\RDer^{\gm}_{\DR(A/\kk)}(C,C),\qquad
 \mathfrak d_B(R)
  :=\RDer^{\gm}_{\DR(A/B)}(R,R).
\]
These are the derived Lie algebras of weight-zero graded
mixed derivations relative to the displayed structural
algebras. They are calculated using cofibrant replacements in
the corresponding categories of graded mixed algebras. Base
change and $\alpha$ induce a morphism of derived Lie algebras
\[
  \beta_{C,\alpha}:\mathfrak d_{\kk}(C)
                 \longrightarrow\mathfrak d_B(R).
\]
We take its homotopy fibre in the $\infty$-category of dg-Lie
algebras; an ordinary kernel of a strict representative need
not compute this fibre.

\begin{theorem}\label{thm:formal-derived-unfolding}
Let $B\to A$ be smooth between smooth $\kk$-algebras, with
$\operatorname{char}\kk=0$, and let $R$ be a perfect relative
derived foliation. At any specified intrinsic unfolding
$(C,\alpha)$, put
\begin{equation}\label{eq:derived-unfolding-controlling-lie}
 \mathfrak d_{(C,\alpha)}
 :=\operatorname{fib}^{\mathrm h}_{\mathrm{dgLie}}
 \bigl(\mathfrak d_{\kk}(C)
           \xrightarrow{\beta_{C,\alpha}}
                         \mathfrak d_B(R)\bigr).
\end{equation}
Then $\mathfrak d_{(C,\alpha)}$ controls the formal
neighbourhood of $(C,\alpha)$ in $\mathcal U_\pi(\cF)$.
More precisely, for a connective local Artinian augmented
$\kk$-cdga $K$ with nilpotent augmentation ideal
$\mathfrak m_K$, there is a natural equivalence
\begin{equation}\label{eq:formal-MC-unfolding}
 \Def^{\tr}_{(C,\alpha)}(K)
 \simeq
 \MC_\bullet\bigl(\mathfrak d_{(C,\alpha)}
                     \widehat\otimes\mathfrak m_K\bigr).
\end{equation}
The left side consists of perfect absolute derived
foliations over $A\otimes_\kk K$, together with an
equivalence of their relative restrictions with the trivial
deformation of $R$, and an identification of the special
fibre with $(C,\alpha)$. The right side is formed using any
complete $L_\infty$ model of the homotopy fibre. 
In particular, the Zariski tangent complex is
$\mathfrak d_{(C,\alpha)}[1]$. For ordinary dual numbers,
first-order deformation classes, their infinitesimal
automorphisms, and higher automorphisms are respectively
\[
 H^1(\mathfrak d_{(C,\alpha)}),\qquad
 H^0(\mathfrak d_{(C,\alpha)}),\qquad
 H^{1-i}(\mathfrak d_{(C,\alpha)})\quad(i\geq2).
\]
The primary obstruction to extending an infinitesimal
deformation across a square-zero extension lies in
$H^2(\mathfrak d_{(C,\alpha)})$ with the appropriate
coefficient module and twisted differential. These groups
are invariant under equivalences of the pointed unfolding
in the graded mixed category, without any restriction on the
Tor-amplitude of $\bbL_{\cF/S}$.
\end{theorem}

\begin{proof}
A deformation of $(C,\alpha)$ consists of a deformation of
$C$ as a graded mixed $\DR(A/\kk)$-algebra, together with a
chosen equivalence between its derived relative restriction
and the constant deformation of $R$. Consequently its
formal deformation functor is the homotopy fibre
\begin{equation}\label{eq:formal-fibre-of-def}
 \Def^{\tr}_{(C,\alpha)}
 \simeq\operatorname{fib}\bigl(
 \Def^{\mathrm{abs}}_C\longrightarrow
 \Def^{\mathrm{rel}}_R\bigr),
\end{equation}
where the fibre is based at the constant relative
deformation. The equivalence is on spaces, not merely on
isomorphism classes: a homotopy in the relative
identification is part of a deformation on the left and is
retained by the homotopy fibre on the right.

We use derived deformations: all reductions and pushouts
are derived, and the special-fibre identification is part
of the object. A local Artinian test cdga is connective,
has residue field $\kk$ and finite-dimensional total
cohomology, and is represented with nilpotent augmentation
ideal. Its finite filtration by small extensions makes
reduction to the residue field conservative on modules.
Explicitly, if $M\otimes_K^{\mathbb L}\kk=0$, then
$M\otimes_K^{\mathbb L}(\mathfrak m_K^p/
\mathfrak m_K^{p+1})=0$, since each quotient is a
$\kk$-complex. The filtration then gives $M=0$.
This argument requires no boundedness of $M$.

For a deformation $C_K$ of the algebra, the map from its
weight-zero base $A\otimes_\kk K$ is therefore an
equivalence in weight zero. The canonical map
$\Sym_{A\otimes K}(C_K^{(1)})\to C_K$ becomes an
equivalence after reduction, since symmetric powers
commute with derived extension of scalars in
characteristic zero. Its cofibre vanishes weight by
weight by the preceding argument.

Perfectness is detected under this reduction as well.
Put $A_K=A\otimes_\kk K$, let $M$ be its weight-one
module, and write $N=M\otimes_K^{\mathbb L}\kk$.
A finite tower of small extensions builds $K$ from
finite shifts of $\kk$
\cite[Proposition~1.1.11]{LurieModuli}. Tensoring that
tower with $M$ shows that its underlying $A$-complex
has bounded, finitely generated cohomology, because $N$
is perfect. The connective algebra $A_K$ is noetherian
with bounded finite cohomology over $A$; successively
choosing finite generators therefore gives a
bounded-above finite free resolution of $M$, that is,
$M$ is pseudo-coherent. If $N$ has Tor-amplitude $[a,b]$,
every discrete $H^0(A_K)$-module $V$ has a finite
filtration by powers of the nilpotent augmentation
ideal, with quotients $V_p$ that are $A$-modules.
Associativity gives
$M\otimes_{A_K}^{\mathbb L}V_p
 \simeq N\otimes_A^{\mathbb L}V_p$.
The cohomology of each term lies in $[a,b]$; the
filtration gives the same bound for
$M\otimes_{A_K}^{\mathbb L}V$. Thus $M$ is
pseudo-coherent of finite Tor-amplitude and hence
perfect, by the finite-cell criterion for modules over
a connective cdga. In the discrete case this is the
criterion and nilpotent reduction of
\cite[Tags~0658 and~07LU]{Stacks}; the argument above
also checks the derived Artinian case.
Consequently the conditions defining perfect foliations
do not cut out a smaller formal deformation problem
at $C$. The same argument applies to $R$.

Choose a cofibrant graded mixed algebra resolution
under the fixed structural algebra. Its infinitesimal
deformations are perturbations of its differential by
relative, weight-zero mixed derivations with coefficients
in the augmentation ideal. The square-zero condition for
the perturbed differential is the Maurer--Cartan equation
in the derivation dg-Lie algebra.
Changes of trivialisation reducing to the identity are
given by its nilpotent gauge transformations; the same
construction with polynomial forms on simplices retains
their higher homotopies. Cofibrancy gives lifting along
trivial fibrations and makes this description independent
of the resolution. It identifies the full formal
deformation functors, rather than only their tangent
spaces. In particular these functors take small-extension
pullbacks to homotopy pullbacks and are formal moduli
problems \cite{Hinich01,Pridham10,LurieModuli}.
Their governing dg-Lie algebras are
$\RDer^{\gm}_{\DR(A/\kk)}(C,C)$ and
$\RDer^{\gm}_{\DR(A/B)}(R,R)$, respectively.

To compute the map between them, resolve the base-change
span $C\leftarrow\DR(B/\kk)\to B$ compatibly with
the structural maps. A relative derivation vanishes
on $\DR(A/\kk)$ and hence on $\DR(B/\kk)$; together
with the zero derivation on $B$, it induces a derivation
on the resolved pushout. It vanishes on
$\DR(A/B)$ and preserves commutators. The construction
with simplicial resolutions gives the morphism
$\beta_{C,\alpha}$ in the localised dg-Lie category,
including its coherences. Consequently the relative
restriction of deformations is the corresponding
morphism of formal moduli problems.

By \cite[Theorem~2.0.2]{LurieModuli} and \cite{Pridham10},
the characteristic-zero equivalence between formal moduli
problems and dg-Lie algebras preserves homotopy limits.
Applied to \eqref{eq:formal-fibre-of-def}, it identifies
its controlling dg-Lie algebra with the homotopy fibre
\eqref{eq:derived-unfolding-controlling-lie}. This proves
\eqref{eq:formal-MC-unfolding}. For square-zero
coefficients the bracket of Maurer--Cartan elements takes
values in the square of the augmentation ideal and hence
vanishes. The standard tangent calculation then gives the
stated cohomology groups. For a non-trivial small extension,
the degree-two component of the twisted Maurer--Cartan
curvature is the obstruction to lifting, with the
usual coefficient and gauge indeterminacies.

Finally, an equivalence of pointed unfoldings induces an
equivalence of their formal deformation functors by
conjugation in the underlying $\infty$-categories. It
therefore induces an equivalence of the controlling dg-Lie
algebras. The argument does not use the strict derivations
of an arbitrary Chevalley--Eilenberg presentation and is
unaffected by the counterexample of
Proposition~\ref{prop:strict-CE-noninvariance}.
\end{proof}

\begin{corollary}\label{cor:derived-symmetry-lifting}
Let $(C,\alpha)$ be as in
Theorem~\ref{thm:formal-derived-unfolding}. A class
$v\in H^0(\mathfrak d_B(R))$ of infinitesimal automorphisms
of the relative foliation lifts, up to homotopy, to an
infinitesimal automorphism of $C$ if and only if its
connecting class
$
 \partial(v)\in H^1(\mathfrak d_{(C,\alpha)})
$
vanishes. This obstruction depends only on the pointed
intrinsic unfolding, not on a chosen CE presentation.
\end{corollary}

\begin{proof}
The long exact sequence associated with the homotopy fibre
\eqref{eq:derived-unfolding-controlling-lie} identifies the
kernel of the connecting map with the image of
$H^0(\mathfrak d_{\kk}(C))$. These are precisely the classes of
absolute infinitesimal automorphisms restricting to $v$.
Invariance follows from Theorem~\ref{thm:formal-derived-unfolding}.
\end{proof}

The formal complex concerns a fixed unfolding; it does not
classify other connected components of $\mathcal U_\pi(\cF)$.

\subsection{Cotangent extensions and two-term unfoldings}
\label{subsec:two-term-abelian}

Write $D(A)$ for the derived $\infty$-category of $A$-modules;
$\Hom_{D(A)}$ denotes morphisms in its homotopy category.
For a relative cotangent complex of Tor-amplitude $[-1,0]$,
transverse extensions can be represented by two-term complexes.
Integrability imposes further conditions.

\begin{proposition}\label{prop:two-term-cotangent-extensions}
Let $B\to A$ be smooth between smooth $\kk$-algebras, put
$Q=A\otimes_B\Omega^1_{B/\kk}$, and let $R$ be a perfect
relative derived foliation with cotangent complex of
Tor-amplitude $[-1,0]$. Choose a two-term projective model
$L=[P^{-1}\xrightarrow{d}P^0]$ for $\bbL_{\cF/S}$.
The cotangent triangle defines a natural map
$\kappa:\mathcal U_\pi(\cF)\to\Map_{D(A)}(L,Q[1])$.
The target is equivalent to the nerve of the action groupoid
$[\Hom_A(P^{-1},Q)/\Hom_A(P^0,Q)]$, with action
$\theta\mapsto\theta+u\circ d$. Its components are
$\operatorname{Ext}^1_A(L,Q)$, its fundamental group at
any point is $\Hom_{D(A)}(L,Q)$, and all higher homotopy
groups vanish. The class of an
unfolding is represented by the two-term complex
\[
 E_\theta=
 [\,P^{-1}\xrightarrow{(d,\theta)}P^0\oplus Q\,].
\]
Its cotangent triangle splits in $D(A)$ if and only if
$[\theta]=0$ in $\operatorname{Ext}^1_A(L,Q)$.
\end{proposition}

\begin{proof}
Theorem~\ref{prop:intrinsic-comparison} assigns to every
pointed unfolding a cofibre sequence
$Q\to\bbL_{\widetilde\cF}\to L$. This assignment is
functorial in equivalences and their homotopies. In the
stable $\infty$-category of perfect $A$-complexes, the
space of such cofibre sequences, with its end terms
identified with $Q$ and $L$, is $\Map_{D(A)}(L,Q[1])$.
Since $L$ is a two-term complex of projectives,
$\RHom_A(L,Q[1])$ is represented in degrees $-1$ and zero by
$\Hom_A(P^0,Q)\xrightarrow{d^*}\Hom_A(P^{-1},Q)$,
where $d^*(u)=u\circ d$. Its degree-zero
cochains are the maps $\theta:P^{-1}\to Q$;
a degree-minus-one cochain $u$ changes $\theta$ by $u\circ d$.
The two-term Hom complex gives the stated action groupoid.
Every extension is represented by $(d,\theta)$, with
boundary $[\theta]$.
The boundary of a cofibre sequence vanishes exactly when
the sequence splits in the derived category.
\end{proof}

A cotangent extension alone does not determine an absolute
foliation. Besides the anchor, one must recover the mixed
structure, including its higher coherences.

Let $B\to A$, $L$, $Q$ and $\kappa$ be as in
Proposition~\ref{prop:two-term-cotangent-extensions},
and fix a point $e\in\Map_{D(A)}(L,Q[1])$, represented
by $E_\theta$. Write $a:\Omega^1_{A/B}\to L$ for the
relative anchor and $\omega$ for the extension class of
the cotangent sequence of $B\to A$. 
Put $D_B=\DR(B/\kk)^{\mathrm{gr}}$ and
$D_A=\DR(A/\kk)^{\mathrm{gr}}$. Let $\mathscr G_e$
be the $\infty$-groupoid of homotopy-coherent diagrams of
graded cdgas
\begin{equation}\label{eq:graded-unfolding-diagram}
 D_B\longrightarrow D_A\longrightarrow G,
 \qquad D_B\longrightarrow B\longrightarrow R^{\mathrm{gr}},
\end{equation}
with the canonical comparison
$G\otimes^{\mathbb L}_{D_B}B\simeq R^{\mathrm{gr}}$
under $\DR(A/B)^{\mathrm{gr}}$. The weight-zero term of
$G$ is $A$, its underlying graded algebra is free on its
weight-one part up to equivalence, and the induced
cotangent extension is identified with $e$. All these
identifications, and not merely their isomorphism classes,
form part of the diagram. More precisely, use the finite
poset $I$ generated by $0<1<2<4$ and $0<3<4$, with
vertices $(D_B,D_A,G,B,R^{\mathrm{gr}})$ in that order.
The arrow $G\to R^{\mathrm{gr}}$ exhibits the displayed
homotopy pushout. Its restriction to
$J=\{0,1,3,4\}\subset I$ is the prescribed boundary;
in particular $D_A\to R^{\mathrm{gr}}$ is the relative
structural map after killing base forms. Write
$\partial\mathscr D=\mathscr D|_J$.

For $\mathscr D\in\mathscr G_e$, choose a cofibrant
replacement of the \emph{whole} diagram in the model
category of diagrams of graded cdgas. Its algebra of compatible
graded derivations, with commutator bracket, represents a
derived graded dg-Lie algebra
$\RDer^{\mathrm{gr}}(\mathscr D)$. Denote by
$\mathscr M(\mathscr D)$ the homotopy fibre, at the
prescribed mixed structure on the boundary, of
\begin{equation}\label{eq:diagram-action-fibre}
 \Map_{\mathrm{dgLie}^{\mathrm{gr}}}
 (\kk(1)[1],\RDer^{\mathrm{gr}}(\mathscr D))
 \longrightarrow
 \Map_{\mathrm{dgLie}^{\mathrm{gr}}}
 (\kk(1)[1],\RDer^{\mathrm{gr}}(\partial\mathscr D)).
\end{equation}
\begin{theorem}\label{thm:homotopy-fibre-mixed-lifts}
An anchor
$j:\Omega^1_{A/\kk}\to E_\theta$ inducing $a$ and the
identity on $Q$ exists if and only if
\begin{equation}\label{eq:fixed-extension-anchor}
 a^*[\theta]=\omega\quad\text{in }
 \operatorname{Ext}^1_A(\Omega^1_{A/B},Q).
\end{equation}
On a smooth affine chart this condition is automatic. 
There is a natural equivalence
\begin{equation}\label{eq:full-fibre-mixed}
 \operatorname{fib}_{e}(\kappa)
 \simeq\hocolim_{\mathscr D\in\mathscr G_e}
                   \mathscr M(\mathscr D).
\end{equation}
It preserves the identifications of the relative foliation,
automorphisms and higher homotopies, and does not depend
on a CE presentation. 
On a cofibrant graded diagram the derived Lie mapping spaces
in \eqref{eq:diagram-action-fibre} can be computed with
the free resolution of $\kk(1)[1]$ from
\cite[Corollary~1.6.0.4]{TV}. Thus their points admit
representatives by compatible derivations $D_n$ of weight
$n$ and degree $1-2n$, subject to the equations
\begin{equation}\label{eq:higher-mixed-equations}
 dD_n+\sum_{i+j=n}[D_i,D_j]=0\qquad(n\geq1),
\end{equation}
with the prescribed boundary actions. These equations
compute the strict representatives; the mapping spaces
retain their homotopies.
\end{theorem}

\begin{proof}
Pulling back a cotangent extension $Q\to E_\theta\to L$
along $a$ must recover the cotangent sequence of $B\to A$.
Classification of extensions in $D(A)$ gives
\eqref{eq:fixed-extension-anchor}. On an affine smooth
chart $\Omega^1_{A/B}$ is projective, so both extension
classes vanish. The assertion about anchors follows from
the corresponding morphism of cofibre sequences.

To compare mixed structures on diagrams, work in the
projective model category
$\mathcal C_I=\operatorname{Fun}(I,\mathrm{cdga}^{\mathrm{gr}}_\kk)$,
with weightwise quasi-isomorphisms and fibrations at every
vertex. Since $I$ is direct, this is also its direct Reedy
model structure. The characteristic-zero model structure
on graded cdgas gives functorial cofibrant replacements
of whole diagrams; no boundedness or finite-generation
condition is imposed on a replacement. These are the
diagram resolutions used in \cite{LepriManetti}, now with
the additional grading. The grading causes no change to
the construction: free extensions and their lifting
conditions are formed weight by weight.

Let $D^{\mathrm c}\to D$ be such a replacement. A graded
derivation of $D^{\mathrm c}$ is a family
$(\delta_i)_{i\in I}$ such that
\[
 \delta_j\circ f=f\circ\delta_i
 \quad\text{for each arrow }f:D^{\mathrm c}_i\to
 D^{\mathrm c}_j.
\]
These families form the graded dg-Lie algebra
$\Der_I^{\mathrm{gr}}(D^{\mathrm c})$ under the commutator.
The module of K\"ahler differentials
$\Omega^1_{D^{\mathrm c}/\kk}$ is cofibrant in the category
of $D^{\mathrm c}$-module diagrams and represents compatible
derivations. Indeed, for a free diagram extension its
relative differentials are the free module on its
generators; passing through a cellular presentation and
retracts proves the assertion for a cofibrant diagram.
Thus compatible derivations are calculated in the
\emph{module category of the diagram}, rather than by
an underived limit of independently resolved vertices.

For a graded dg-Lie algebra $\mathfrak l$, let
$\mathscr A_{\mathfrak l}(D)$ be the homotopy fibre over
$D$ of the forgetful functor from diagrams of graded
$\mathfrak l$-cdgas to $\mathcal C_I$. Equivalently,
these are commutative algebra objects in
$\operatorname{Fun}(I,\mathfrak l\text{-}\mathrm{Mod}^{\mathrm{gr}})$,
with pointwise tensor product. An action on a fixed strict
diagram gives a morphism
$\mathfrak l\to\Der_I^{\mathrm{gr}}(D^{\mathrm c})$.
Applying a cofibrant simplicial resolution of $\mathfrak l$
gives a natural comparison
\[
 \Map_{\mathrm{dgLie}^{\mathrm{gr}}}
   (\mathfrak l,\Der_I^{\mathrm{gr}}(D^{\mathrm c}))
       \longrightarrow\mathscr A_{\mathfrak l}(D).
\]
We check it on spaces before imposing the boundary.

First let $\mathfrak l$ be the free graded dg-Lie algebra
on a cofibrant graded complex $V$. An action is then just
a $V$-family of compatible chain derivations; there are no
additional Lie relations to impose. If
$D^{\mathrm c}\to D'$ is a trivial fibration of diagrams,
an action on $D'$ restricts to a derivation along this map.
Its lifts are computed by
\[
 \Hom_{D^{\mathrm c}\text{-}\mathrm{Mod}_I}
   (V\otimes_\kk\Omega^1_{D^{\mathrm c}/\kk},D^{\mathrm c})
 \longrightarrow
 \Hom_{D^{\mathrm c}\text{-}\mathrm{Mod}_I}
   (V\otimes_\kk\Omega^1_{D^{\mathrm c}/\kk},D').
\]
In simplicial mapping spaces this is a trivial fibration:
the source module is cofibrant and
$D^{\mathrm c}\to D'$ is a trivial fibration of module
diagrams; compare \cite[Lemma~6.1]{LepriManetti}.
Its fibres are therefore contractible. The same
lifting statement holds over every simplex and relative
to its boundary. Two trivial-fibration resolutions of a
diagram admit a common resolution, obtained by resolving
their fibre product; its projections are again trivial
fibrations. Applying the lifting statement to these common
resolutions and their simplicial frames identifies the
choices and their higher compatibilities. Cofibrant replacement consequently
identifies the full space of free actions on diagrams
equivalent to $D$ with the space of such derivations on
$D^{\mathrm c}$. This proves the comparison for free
$\mathfrak l$, including its homotopies.

Both sides take homotopy colimits in $\mathfrak l$ to
homotopy limits. For the action side, the derived enveloping
algebra functor preserves colimits, and modules over a
colimit of associative algebras form the limit of their
module categories \cite[Lemma~1.6.0.3]{TV}.
This is an equivalence over the underlying graded
complexes and is compatible with tensor products.
Taking $I$-diagrams, commutative algebras and the fibre
at $D$ preserves that limit. For the Lie mapping space
the assertion follows from its universal property.
Every graded dg-Lie algebra is the geometric realisation
of its simplicial free resolution. The comparison for
free Lie algebras therefore proves the displayed
equivalence for every $\mathfrak l$.
This supplies the diagram version of
\cite[Proposition~1.6.0.2]{TV}; no diagram-action
classification is being inferred from the deformation
theorem of \cite{LepriManetti}.

The boundary can be computed with the same replacement.
To see this, let $F_i$ be the free-diagram functor left
adjoint to evaluation at $i$. For the posets $I$ and $J$
specified above, $F_i|_J=F_i$ if $i\in J$, whereas
$F_2|_J=F_4$. Restriction therefore takes projective
generating cofibrations and trivial cofibrations to the
corresponding maps for $J$, and preserves cofibrant
diagrams. Hence $(D^{\mathrm c})|_J$ is a cofibrant
replacement of the boundary. Restriction of compatible
families gives an actual dg-Lie morphism
\[
 \Der_I^{\mathrm{gr}}(D^{\mathrm c})
       \longrightarrow
 \Der_J^{\mathrm{gr}}((D^{\mathrm c})|_J),
\]
and the action comparisons just proved commute with it.
Taking their homotopy fibres at the specified boundary
action, with $\mathfrak l=\kk(1)[1]$, identifies
$\mathscr M(\mathscr D)$ with the space of mixed structures
extending the prescribed boundary. The fibre fixes the
boundary up to its specified coherent identification;
an ordinary kernel would not suffice.

These comparisons are independent of the replacements.
For each $\mathfrak l$ they identify the mapping space
with the intrinsic action space, naturally in
$\mathfrak l$ and under equivalences of diagrams.
The derived Yoneda property then identifies the
representing Lie objects, with their boundary maps.
It also supplies the coherent transport along equivalences
needed to regard $\mathscr M$ as a functor on
$\mathscr G_e$.

The diagram in \eqref{eq:graded-unfolding-diagram}
records both the absolute de Rham map and the chosen
identification of the derived relative reduction with $R$.
A mixed lift of this diagram is precisely an absolute
graded mixed algebra $C$, together with
$C\otimes^{\mathbb L}_{\DR(B/\kk)}B\simeq R$
and the prescribed identifications. Conversely such an
unfolding gives the diagram by forgetting its mixed
structure. The comparison remains an equivalence on
higher simplices, since both constructions use homotopy
pushouts and equivalences in the relevant $\infty$-categories.
The derived pushout condition is detected after forgetting
the mixed structure: the forgetful functor is conservative
and preserves these pushouts, as follows from the
symmetric monoidal module description of graded mixed
complexes. Its weight-one part recovers $e$.

Consequently, the space of all mixed lifts of admissible
graded diagrams is $\operatorname{fib}_e(\kappa)$. The
forgetful map from this space to $\mathscr G_e$ has
homotopy fibre $\mathscr M(\mathscr D)$ at $\mathscr D$.
Straightening this map of spaces, or equivalently taking
the Grothendieck construction of its fibres, proves
\eqref{eq:full-fibre-mixed}. This step takes place in the
category of spaces, so it retains the automorphisms and
higher homotopies of the diagrams and of their lifts.
Finally the free Lie resolution of $\kk(1)[1]$ in
\cite[Corollary~1.6.0.4]{TV} yields
\eqref{eq:higher-mixed-equations} on a fixed cofibrant
diagram. Its simplicial mapping space, rather than the
set of its strict Lie morphisms, supplies the homotopies.
The homotopy colimit in \eqref{eq:full-fibre-mixed}
organises these compatible actions over the possible
cotangent and anchor diagrams. The equations describe
their integrability for a fixed diagram; the cotangent
extension by itself supplies neither those equations nor
their solutions.
\end{proof}

\begin{corollary}\label{cor:intrinsic-etale-two-term}
Let $\pi:X\to S$ be a smooth morphism of smooth schemes
and let $\cF/S$ have perfect cotangent complex of
amplitude $[-1,0]$. For an \`etale affine hypercover
$U_\bullet\to X$ and a fixed global cotangent extension
$e$, there are natural equivalences
\[
 \mathcal U_\pi(\cF)\simeq
 \holim_{[n]\in\Delta}\mathcal U_{\pi|U_n}(\cF|_{U_n}),
 \qquad
 \operatorname{fib}_{e}(\kappa_X)\simeq
 \holim_{[n]\in\Delta}
 \operatorname{fib}_{e|U_n}(\kappa_{U_n}).
\]
These are intrinsic descent statements; they do not
identify the right-hand sides with derived flat sections
of a presentation-dependent crossed module.
\end{corollary}
\begin{proof}
All sheaves in this argument are on the small \`etale site
of the fixed smooth scheme. Perfect complexes and their derived mapping spaces satisfy
\`etale hyperdescent. On a fixed weight-one complex, the
graded algebra is its derived symmetric algebra. Each
weight is perfect, and the spaces of compatible mixed
structures and structural maps descend as spaces of
algebraic structures on these sheaves. Equivalently,
sheaves of graded mixed commutative algebras form the
limit of their categories on a hypercover: the algebra
operations and all their coherent compatibilities are
maps of sheaves, and their defining limits commute with
descent. The canonical quasi-freeness condition and the amplitude
bounds are detected locally. Taking the full
subcategories of perfect foliations and their maximal
subgroupoids therefore preserves this limit
\cite[Section~2.1.3]{TV}.

Relative restriction is formed by a homotopy pushout of
sheaves and is compatible with \`etale pullback.
The map $\kappa$ is likewise natural, since it is obtained
from the weight-one cofibre sequence. Its target is the
derived mapping space of perfect complexes, including
its automorphisms. Hence the diagrams defining both
homotopy fibres are recovered as limits of the
corresponding diagrams on $U_\bullet$. Commutation of
limits in spaces gives the asserted equivalences. This
uses descent for mapping spaces and their coherences,
not only effective descent of objects.
\end{proof}

On a fixed strict free model $C_\theta=\Sym_A(E_\theta[1])$
with anchor $j$, suppose $\epsilon_0$ is a compatible
weight-one chain derivation. Its strict corrections are
$\epsilon_0+\eta$, where $\eta$ vanishes on the absolute
and relative end data and satisfies
\begin{equation*}
 \epsilon_0^2+[\epsilon_0,\eta]
                  +\tfrac12[\eta,\eta]=0.
\end{equation*}
This identity may be checked on $A\oplus E_\theta[1]$.
Strict conjugation does not recover all homotopies in
\eqref{eq:full-fibre-mixed}.

The anchor condition remains necessary on non-affine
families. For the following class of two-term algebroids,
the integrability equations have a more explicit form.

Let $P$ be finite projective over $A$, with integrable
relative connection $\nabla^{\mathrm{rel}}$ along $T_{A/B}$.
Place $P[-1]$ in degree one with zero differential and form
\[
 \g_P=T_{A/B}\ltimes_{\nabla^{\mathrm{rel}}}P[-1].
\]
Its anchor is the identity on $T_{A/B}$ and zero on $P[-1]$.
The bracket of two elements of $P[-1]$ vanishes. The associated
relative foliation has cotangent complex
\[
 \bbL_{\cF_P/S}\simeq
 \Omega^1_{A/B}\oplus P^\vee[1].
\]
The summand $H^{-1}(\bbL_{\cF_P/S})=P^\vee$ contributes
an additional datum to the absolute classification: a transverse
one-form, as well as a connection.

\begin{theorem}\label{thm:two-term-transitive}
Let $B\to A$ be smooth between smooth $\kk$-algebras and let
$(P,\nabla^{\mathrm{rel}})$ be as above. The groupoid of strict
transversal unfoldings of $\g_P$ is equivalent to the discrete
set of pairs $(\nabla,\theta)$, where $\nabla$ is a flat
$T_{A/\kk}$-connection on $P$ extending
$\nabla^{\mathrm{rel}}$, and
\[
 \theta\in\Hom_A(A\otimes_B T_{B/\kk},P),
 \qquad d_\nabla\theta=0.
\]
Here $\theta$ is viewed as a $P$-valued absolute one-form
annihilating $T_{A/B}$. The corresponding absolute dg-Lie
algebroid is, as a graded $A$-module,
$
 \widetilde\g_{\nabla,\theta}
       =T_{A/\kk}\oplus P[-1],
$
with the usual tangent bracket, $[u,p]=\nabla_u p$,
$[p,q]=0$, and differential
$
 d_{\widetilde\g}(u)=\theta(u),\
 d_{\widetilde\g}(p)=0.
$
Its cotangent complex is the two-term complex
$
 \bbL_{\widetilde\cF}\simeq
 [\,P^\vee\xrightarrow{\theta^\vee}\Omega^1_{A/\kk}\,],
$
with terms in degrees $-1$ and $0$. The map $\theta$ is, up
to the sign convention for connecting morphisms, the component
in $P$ of the boundary
\[
 A\otimes_B T_{B/\kk}\longrightarrow \g_P[1]
\]
in the derived tangent triangle of the unfolding. Thus its
class is determined by the corresponding point of
$\mathcal U_\pi(\cF_P)$ and its specified relative
identification. The pair $(P,\nabla^{\mathrm{rel}})$ is itself
recovered from the cohomology and the bracket of the perfect
derived Lie algebroid corresponding to $\cF_P/S$.
\end{theorem}

\begin{proof}
Let $\widetilde\g$ be a strict unfolding. Its degree-zero
part sits in a morphism of short exact sequences
\[
\begin{tikzcd}[column sep=small]
0\arrow[r]&T_{A/B}\arrow[r]\arrow[d,equal]
 &\widetilde\g^{0}\arrow[r]\arrow[d,"\widetilde\rho^{0}"]
 &A\otimes_BT_{B/\kk}\arrow[r]\arrow[d,equal]&0\\
0\arrow[r]&T_{A/B}\arrow[r]
 &T_{A/\kk}\arrow[r]
 &A\otimes_BT_{B/\kk}\arrow[r]&0.
\end{tikzcd}
\]
The middle arrow is therefore an isomorphism. In degree one
we have $\widetilde\g^{1}=P$ because the transverse quotient
is concentrated in degree zero; there are no components in
other degrees. We use the anchor to identify
$\widetilde\g^0$ with $T_{A/\kk}$. The degree-zero bracket
must then be the usual bracket of vector fields, since the
anchor preserves brackets. The bracket of a vector field $u$
with $p\in P$ defines an operator $\nabla_u p$. The
Lie--Rinehart identity makes $\nabla$ a connection and the
Jacobi identity for $(u,v,p)$ gives
\[
 [\nabla_u,\nabla_v]p=\nabla_{[u,v]}p.
\]
It is thus flat, and its restriction to $T_{A/B}$ is the
given relative connection. The bracket of two elements of
$P[-1]$ is zero for degree reasons.

The differential of $\widetilde\g$ is an $A$-linear map
$T_{A/\kk}\to P$, since $A$ has zero differential. It
vanishes on $T_{A/B}$, because the inclusion of $\g_P$ is a
chain map, and consequently factors through
$A\otimes_BT_{B/\kk}$. Denote this map by $\theta$.
The differential is a derivation of the bracket precisely
when, for vector fields $u,v$,
\[
 \theta([u,v])
 =\nabla_u\theta(v)-\nabla_v\theta(u),
\]
which is $d_\nabla\theta=0$. The remaining differential
identities follow because $P[-1]$ is abelian and there is no
degree-two term. Conversely, these formulas define a dg-Lie
algebroid: the Lie--Rinehart identity is the connection
Leibniz rule, Jacobi follows from flatness, and the displayed
closedness condition makes the differential a derivation.
The relative summand is a dg-Lie subalgebroid and the anchor
induces the prescribed identity on the transverse quotient.
Hence every pair defines a strict unfolding.

A morphism of strict unfoldings fixes $P$ and $T_{A/B}$,
and preserves the absolute anchor and the transverse quotient.
Because the anchor is an isomorphism in degree zero, such a
morphism is the identity on $T_{A/\kk}$; it is already the
identity on $P$ in degree one. It can exist between two of
the algebroids above only when their differentials and
brackets agree. This proves the classification of the strict
groupoid, including its stabilisers.

For the boundary, apply
Proposition~\ref{prop:two-term-cotangent-extensions} to
$L=\Omega^1_{A/B}\oplus P^\vee[1]$. Its extension map
$P^\vee\to Q$ is $\theta^\vee$; dualising gives precisely
$\partial:A\otimes_BT_{B/\kk}\to P$ represented by
$\theta$, up to sign. The identification is functorial
for pointed intrinsic unfoldings. Finally,
$H^0(\g_P)=T_{A/B}$ and $H^1(\g_P)=P$; their bracket
recovers $\nabla^{\mathrm{rel}}$. These data are invariant
under the CE equivalence \cite[Theorem~2.2.0.7]{TV}.
\end{proof}

\begin{theorem}\label{thm:two-term-curve-obstruction}
Let $(P,\nabla^{\mathrm{rel}})$ be as in
Theorem~\ref{thm:two-term-transitive}. Put $Q^p=A\otimes_B\Omega^p_{B/\kk}$,
$K^\bullet(E)=\Omega^\bullet_{A/B}\otimes_AE$ and
$W=\Omega^1_{B/\kk}$.
The mixed curvature of an absolute connection extending
$\nabla^{\mathrm{rel}}$ defines a canonical class
\[
 \mathfrak o_1(P,\nabla^{\mathrm{rel}})
 \in H^1(K^\bullet(\End_A P)\otimes_B W).
\]
Its vanishing is necessary for a strict unfolding and is
equivalent to the existence of an absolute connection
$\nabla^0$ with zero mixed curvature. For
$E=P,\End_A P$ denote by $C_E^p\subset E\otimes_AQ^p$
the sections horizontal for the relative connection. The absolute covariant
differential of $\nabla^0$ restricts to
$D:C^p_{\End P}\to C^{p+1}_{\End P}$, and its curvature
$F\in C^2_{\End P}$ satisfies
\[
 D^2=[F,-],\qquad DF=0.
\]
The strict unfoldings correspond exactly to pairs
$\alpha\in C^1_{\End P}$ and $\theta\in C^1_P$ with
\begin{equation}\label{eq:two-term-curved-flatness}
 F+D\alpha+\tfrac12[\alpha,\alpha]=0,
 \qquad d_{\nabla^0+\alpha}\theta=0.
\end{equation}
If $P$ is a line bundle, $D$ is independent of $\nabla^0$
and the secondary class
$\mathfrak o_2=[F]\in H^2(C^\bullet_{\End P},D)$
is independent of the extension with zero mixed curvature.
A strict unfolding then exists if and only if
$\mathfrak o_1=\mathfrak o_2=0$.

For a smooth affine curve base, $Q^2=0$; in any rank,
$\mathfrak o_1=0$ is necessary and sufficient. The strict
unfoldings then form a discrete torsor under
\[
 H^0(K^\bullet(\End_A P)\otimes_B W)
 \oplus H^0(K^\bullet(P)\otimes_B W).
\]
The second summand gives the cotangent connecting maps
$P^\vee\to Q^1$; each fibre is a torsor under the first.
\end{theorem}

\begin{proof}
The smooth affine tangent sequence splits $A$-linearly,
and projectivity of $P$ extends its relative connection
to an absolute connection $\nabla$. Its relative curvature
vanishes. Thus the mixed component
$F_{\mathrm{mix}}(v,\bar u)=F_\nabla(v,u)$, with
$v\in T_{A/B}$, does not depend on the lift $u$ of
$\bar u$. The Bianchi identity makes it closed in the $B$-linear complex
$K^\bullet(\End_A P)\otimes_BW$.
Replacing $\nabla$ by $\nabla+\beta$ for
$\beta\in\End(P)\otimes_AQ^1$ changes this component by
$d_{\nabla^{\mathrm{rel}},\End}\beta$. This proves the
independence of $\mathfrak o_1$ and its stated vanishing
criterion.

Fix an extension $\nabla^0$ with zero mixed curvature. Its
remaining curvature $F$ has only transverse arguments and is
horizontal for the relative connection, by the Bianchi identity.
For $\gamma\in C^p_{\End P}$, contraction of
$d_{\nabla^0}\gamma$ with a vertical vector field vanishes
because $\gamma$ is horizontal for the relative connection.
Applying the same observation to
$d_{\nabla^0}^2\gamma=[F,\gamma]$, whose right-hand side has
only transverse arguments, shows that $d_{\nabla^0}\gamma$
is horizontal for the relative connection. Thus $D$ preserves
$C^\bullet_{\End P}$, with $D^2=[F,-]$ and $DF=0$. The extensions with zero mixed
curvature are precisely $\nabla^0+\alpha$ with
$\alpha\in C^1_{\End P}$; their curvature is
$F+D\alpha+\frac12[\alpha,\alpha]$.
A transverse $P$-valued form is covariantly closed for
this connection precisely when it is relatively flat and
satisfies the second equation of
\eqref{eq:two-term-curved-flatness}. The classification
now follows from Theorem~\ref{thm:two-term-transitive}.

For a line bundle, the induced connection on $\End(P)=A$
is the ordinary differential. Hence $D$ is fixed,
$D^2=0$, and replacing $\nabla^0$ by another connection with
zero mixed curvature changes $F$ by $D\alpha$, since the
bracket vanishes. The class
$\mathfrak o_2$ is consequently well defined. Its
vanishing gives a flat extension; taking $\theta=0$
proves sufficiency. If $B$ is a curve, $Q^2=0$, so zero mixed curvature implies
flatness, and every
$\theta\in C^1_P$ is closed. The flat connections differ
by $C^1_{\End P}$. These spaces are the displayed
$H^0$ groups, and the connecting map is $\theta^\vee$
by Theorem~\ref{thm:two-term-transitive}.
\end{proof}

The connecting morphism also distinguishes absolute
foliations which have identical relative restrictions.

\begin{example}\label{ex:two-term-boundary}
Let $B=\kk[t]$, $A=B[x]$ and $P=Ae$. Give $P$ the
relative connection $\nabla^{\mathrm{rel}}_{\partial_x}e=0$.
For $\lambda,q\in B$, define
\[
 \nabla_{\partial_x}e=0,\qquad
 \nabla_{\partial_t}e=\lambda(t)e,
 \qquad \theta(\partial_t)=q(t)e,
 \quad \theta(\partial_x)=0.
\]
Both $\nabla$ and $\theta$ satisfy the theorem. Thus the
absolute algebroid has, in addition to its bracket, the
differential
\[
 d_{\widetilde\g}(\partial_t)=q(t)e,
 \qquad d_{\widetilde\g}(\partial_x)=0.
\]
Dually, the absolute cotangent differential sends $e^\vee$
to $q(t)\,dt$. Passing to relative forms kills $dt$ and
recovers the original cotangent complex. If $q\ne q'$,
the resulting pointed
intrinsic unfoldings are inequivalent, because their
connecting morphisms differ in
$\Hom_A(A\partial_t,P)$. In particular, the family with
$q\ne0$ cannot be recovered by considering only flat
connections on the degree-one isotropy.
\end{example}

\begin{example}\label{ex:two-term-no-extension}
Let $B=\kk[t]$ and $A=B[x,x^{-1}]$. On $P=Ae$ take the
flat relative connection
\[
 \nabla^{\mathrm{rel}}_{\partial_x}e=\frac{t}{x}e.
\]
A strict transversal unfolding would, by
Theorem~\ref{thm:two-term-transitive}, give a flat absolute
connection. Writing
$\nabla_{\partial_t}e=h(x,t)e$ with $h\in A$, its curvature
on $\partial_x,\partial_t$ would be
\[
 [\nabla_{\partial_x},\nabla_{\partial_t}]e
  =\bigl(\partial_xh-x^{-1}\bigr)e.
\]
No Laurent polynomial in $x$ has derivative $x^{-1}$.
Equivalently, with this order of curvature arguments, the class is
$-[dx/x]\otimes dt\ne0$ in
$H^1(K^\bullet(\End_A P)\otimes_B W)$. By
Theorem~\ref{thm:two-term-curve-obstruction}, the relative
foliation $\CE^*(\g_P)$ has no strict transversal unfolding,
although its cotangent complex is perfect of amplitude
$[-1,0]$.
\end{example}

\begin{example}\label{ex:two-term-secondary-obstruction}
The primary obstruction need not suffice over a surface.
Let $\kk=\mathbb C$ and
$S=(\mathbb P^1\times\mathbb P^1)\setminus\Delta$, with
$\Delta$ the diagonal. Since $\Delta$ is ample of type
$(1,1)$, $S$ is smooth and affine. Take
$L=\mathcal O(1,0)|_S$, $X=S\times\mathbb A^1$ and
$P=\operatorname{pr}_S^*L$ with its tautological relative
flat connection. The Atiyah extension of $L$ splits
on the affine scheme $S$, giving an algebraic connection;
its pullback has zero mixed curvature, so
$\mathfrak o_1=0$. The Gysin sequence, writing $h_i$ for
the hyperplane classes of the factors, gives
\[
 H^2_{\mathrm{dR}}(S)
 \simeq(\mathbb C h_1\oplus\mathbb C h_2)/
          \mathbb C(h_1+h_2).
\]
Since $C^\bullet_{\End P}=\Omega^\bullet_{S/\mathbb C}$,
the class $\mathfrak o_2$ is, up to the curvature sign
convention, $c_1^{\mathrm{dR}}(L)=h_1|_S\ne0$. No flat absolute
connection on $P$, and hence no strict unfolding of
$T_{X/S}\ltimes P[-1]$, exists.
\end{example}

\begin{remark}\label{rem:tv-transverse-distinction}
The equivalence \eqref{eq:normal-derived-comparison}
of normal complexes does not identify the transverse crystals
of \cite[Section~3.4]{TV}. In particular, the existence of an
unfolding alone does not provide a Gauss--Manin connection.
\end{remark}

\subsection{Singular pullbacks and Hamiltonian presentations}

A non-transverse pullback of a smooth foliation can acquire
a non-trivial derived cotangent complex
\cite[Section~2.1]{TV}. The existence of the requisite
derived structure is not automatic; in relevant classes it is
related to Godbillon--Vey sequences
\cite[Section~2.4]{TV}. For quasi-smooth rigid foliations,
\[
 N^*_{\cF}\longrightarrow\Omega^1_X\longrightarrow\bbL_\cF
\]
records the singular correction \cite{TVRH,TV}.

For a Hamiltonian presentation arising from shifted Poisson
geometry \cite{PTVV,CPTVV}, the strict classification applies
when a compatible relative dg-Lie-algebroid model is specified.
An $L_\infty$ presentation alone does not identify the strict
unfolding groupoid. The relative de Rham case is recovered
when the anchor is the relative tangent algebroid. Neither
application supplies presentation independence of the strict
construction or global derived descent.

\section{Comparison of Chevalley--Eilenberg presentations}\label{sec:homotopy-invariance}

A strict isomorphism of CE presentations identifies their
basic derivations by conjugation. The same assertion fails
for quasi-isomorphisms: the counterexample below shows that
the resulting strict crossed modules need not have the same
cohomology.

\begin{definition}
For a relative graded mixed algebra $R$, let
$\mathcal{A}ut^{\gm}_{\bas}(R)$ be the formal derived group of
automorphisms of the diagram consisting of $B\to A$ and
$\DR(A/B)\to R$, completed at the identity on nilpotent
augmented test cdgas. Compatible automorphisms of $A$ and $B$
are included. Whenever its relative infinitesimal Lie algebroid
is formed, denote it by $\mathbb R\Der^{\gm}_{\bas}(R)$.
This notation refers to derived automorphisms. It does not
identify their tangent complex with either strict derivations
or $\mathfrak P(R)$ without a separate calculation.
\end{definition}

\begin{proposition}\label{prop:homotopy-invariance-der}
An equivalence $f:R\simeq R'$ of projectable graded mixed
diagrams induces an equivalence of formal derived automorphism
groups over the automorphisms of $B$, coherent under composition.
Consequently their relative infinitesimal Lie algebroids,
when formed in the same category of formal moduli problems, are equivalent
over $T_{B/\kk}$, with their brackets and higher operations.
\end{proposition}

\begin{proof}
Let $\mathcal C$ be the $\infty$-category of the diagrams in
question and $\operatorname{Eq}(\mathcal C)$ its category of
equivalence arrows. Evaluation at either endpoint is an
equivalence $\operatorname{Eq}(\mathcal C)\simeq\mathcal C$:
identity arrows give an inverse, and the square from an
equivalence arrow to an identity arrow is itself an equivalence.
Thus both maps
\[
 \Aut_{\operatorname{Eq}(\mathcal C)}(f)\longrightarrow
 \Aut_{\mathcal C}(R),\qquad
 \Aut_{\operatorname{Eq}(\mathcal C)}(f)\longrightarrow
 \Aut_{\mathcal C}(R')
\]
are equivalences of group-like spaces. They preserve projection
to the common base automorphisms.

Apply this after extension to each nilpotent test cdga and
completion at the identity. It proves the assertion for formal
derived groups. Passing to relative infinitesimal Lie objects
preserves this equivalence and their operations in the
characteristic-zero formal-moduli setting
\cite{Pridham10,NuitenKoszul}.
For a composable pair of equivalences, the composition
$2$-simplex identifies the two transports; higher simplices
supply the coherences for longer strings. This proves the
claimed naturality without choosing strict inverses.
\end{proof}

\begin{proposition}\label{prop:cartan-action-equivalence}
Let $f:R\to R'$ be a morphism of strictly perfect CE algebras
over the fixed diagram $B\to A$.
Write $R=\Sym_A(W)$, $R'=\Sym_A(W')$, and let
$p:\g'\to\g$ be the dual of $f^{(1)}:W\to W'$, with the
suspensions understood. Then $p$ is a dg-Lie-algebroid morphism
and
$
 f\imath_{p(x)}=\imath_xf,\  fL_{p(x)}=L_xf.
$
If $\Theta'f=f\Theta$ for homogeneous projectable
weight-zero derivations of the same degree, then
\begin{equation*}
 f[\Theta,\imath_{p(x)}]=[\Theta',\imath_x]f,\qquad
 f[\Theta,L_{p(x)}]=[\Theta',L_x]f.
\end{equation*}
A weak equivalence $f$ makes $p$ a quasi-isomorphism.
These are strict identities along $f$; by themselves they do
not identify all derived automorphisms or their higher homotopies.
\end{proposition}

\begin{proof}
Compatibility of $f$ with the mixed differential on functions
and weight-one generators identifies respectively the anchors
and brackets under $p$. Its internal differential identifies
the chain maps. On functions both contractions vanish; on a
weight-one generator $w$ their common value is
$\langle f^{(1)}w,x\rangle=\langle w,p(x)\rangle$.
They are derivations along $f$, so equality on the generators
proves the first identity. Commuting it with the mixed
differentials proves the Lie-derivative identity.

Using $\Theta'f=f\Theta$ and the graded commutator gives
\[
 [\Theta',\imath_x]f
 =\Theta'\imath_xf-(-1)^{|\Theta||\imath_x|}
                      \imath_x\Theta'f
 =f[\Theta,\imath_{p(x)}].
\]
The argument for Lie derivatives is the same.
Strict perfection ensures that dualising the weight-one
quasi-isomorphism gives a quasi-isomorphism.
All identities commute with composition and extension of scalars,
including polynomial forms on a simplex.
This verifies their compatibility for strict intertwining
operators. It does not assert that every homotopy-coherent
automorphism of the equivalence arrow is represented by such
strict operators.
\end{proof}

\begin{corollary}\label{cor:cartan-crossed-invariant}
Suppose two strict crossed modules of basic derivations are
linked over $T_{B/\kk}$ by a zigzag of morphisms preserving
the inner maps and actions and inducing quasi-isomorphisms on
both entries. Their dg-Lie-algebroid realisations and derived
spaces of sections are equivalent.
The same conclusion holds for specified homotopy-coherent
crossed modules with an equivalence of the full crossed data
and a compatible realisation functor.
\end{corollary}

\begin{proof}
For strict crossed modules this is
Proposition~\ref{prop:crossed-dg-realisation}, applied to each
arrow, followed by invariance of derived mapping spaces.
For the second assertion, apply the specified realisation
functor to the equivalence of crossed data.
Equivalence of the underlying CE diagrams alone is not one
of these comparison hypotheses.
\end{proof}

\begin{proposition}\label{prop:strict-ce-isomorphism}
Let $f:R\to R'$ be an isomorphism of quasi-free graded mixed
Chevalley--Eilenberg algebras over the fixed diagram $B\to A$.
Then $f$ induces an isomorphism of crossed modules
\[
 [\g\xrightarrow{\iota}D^1_{\bas}(\g/B)]
 \;\simeq\;
 [\g'\xrightarrow{\iota'}D^1_{\bas}(\g'/B)]
\]
over $T_{B/\kk}$, and hence an isomorphism of their dg-Lie-algebroid
realisations and an equivalence of their derived splitting spaces.
The assertion concerns isomorphisms of presentations; a weak
equivalence need not admit an inverse in the category of strict
presentations.
\end{proposition}

\begin{proof}
Write $W=R^{(1)}$ and $W'=(R')^{(1)}$. The weight-one
restriction of $f$ is an isomorphism $u:W\to W'$.
Dualising it identifies the two relative dg-Lie algebroids, including
their anchors and brackets, because $f$ commutes with the mixed
differentials.  Conjugation $\Theta\mapsto f\Theta f^{-1}$
identifies their graded mixed derivation dg-Lie algebras and fixes the
transverse symbol. For $x\in(W')^\vee$ the identity
$
 f\,\imath_{u^\vee x}=\imath_x f
$
and commutation with the mixed differential imply
$fL_{u^\vee x}=L_x f$.  Thus conjugation intertwines the two
inner maps, not merely their underlying complexes. It also
intertwines the actions, since they are commutators of derivations.
The construction of Proposition~\ref{prop:crossed-dg-realisation}
therefore gives an isomorphism over $T_{B/\kk}$.  Applying the
derived mapping-space functor proves the last assertion.
\end{proof}

\begin{proposition}
\label{prop:strict-CE-noninvariance}
The strict realisation
$\g\mapsto\mathbb U_{\dg}(\g\to\Der_\kk(\g))$
does not preserve weak equivalences of strictly perfect
CE presentations, even over $A=B=\kk$.
More precisely, let $\h=\kk z$ be abelian in degree zero,
and let $\g$ have basis
\[
 |z|=0,\quad |e|=1,\quad |f|=2,\qquad
 dz=df=0,\quad de=f,\qquad [e,e]=f,
\]
with every other bracket zero.
The projection $\g\to\h$ and the inclusion $\h\to\g$
are quasi-isomorphisms and induce weak equivalences of CE
graded mixed algebras. Nevertheless
\[
 H^2\mathbb U_{\dg}(\g)\cong\kk,\qquad
 H^2\mathbb U_{\dg}(\h)=0.
\]
Moreover $H^{-1}\mathbb U_{\dg}(\g)=
\kk P\oplus\kk\mathrm{s} z$, where $P(e)=z$, whereas
$H^{-1}\mathbb U_{\dg}(\h)=\kk\mathrm{s} z$.
\end{proposition}

\begin{proof}
The bracket is graded skew-symmetric because $e$ is odd;
every iterated non-zero bracket would contain the central
element $f$, so Jacobi holds. The differential is a bracket
derivation: on $[e,e]$ its two terms are
$[f,e]-[e,f]=0$, and all other cases vanish.
The span of $e,f$ is an acyclic ideal.
This proves the two quasi-isomorphisms.
Finite-dimensional duality preserves the quasi-isomorphisms.
In each weight the CE map is a symmetric power of that dual
map; tensor powers preserve quasi-isomorphisms over the field
and averaging over the symmetric group is exact in
characteristic zero. Thus the CE maps are weak equivalences
weight by weight.

All graded Lie derivations of $\g$ are spanned by the following
operators, with unspecified values equal to zero:
\[
\begin{array}{c|l|c}
 \text{degree}&\text{operator}&[d,-]\\ \hline
 -1&P(e)=z&0\\
 0&A(z)=z&0\\
 0&E(e)=e,\quad E(f)=2f&-Q\\
 1&Q(e)=f&0\\
 2&R(z)=f&0 .
\end{array}
\]
Indeed the derivation identity on $[e,e]=f$ determines the
value on $f$ from the value on $e$.
The identities on $[z,e]=[z,f]=0$ force the other possible
components to vanish. For example, a degree-minus-two map
$f\mapsto z$ violates the derivation identity on $[e,e]=f$,
and a degree-one map $z\mapsto e$ violates it on $[z,e]=0$.
This gives the entire list and the displayed differential.
Consequently
\[
 H^{-1}\Der_\kk(\g)=\kk P,\qquad
 H^0\Der_\kk(\g)=\kk A,\qquad
 H^2\Der_\kk(\g)=\kk R,
\]
with no other cohomology.

The inner source has cohomology only $\kk z$ in degree zero,
and $\operatorname{ad}_z=0$.
The cone exact sequence therefore gives the asserted
degree-two cohomology and the additional class
$\mathrm{s} z$ in degree minus one.
For $\h$, the inner map is zero and
$\Der_\kk(\h)=\kk A$ is concentrated in degree zero.
Its cone is $\kk A\oplus\kk\mathrm{s} z$, proving the
comparison. Thus the two strict realisations cannot be
equivalent even as complexes.
\end{proof}

\begin{corollary}
\label{cor:strict-derived-loops-counterexample}
For $B=A=\kk[t]$, tensor the preceding $\g$ with $B$
and give it zero relative anchor. At its coefficientwise flat
unfolding, the comparison of
Theorem~\ref{thm:strict-derived-comparison} induces on
fundamental groups the proper inclusion
\[
 B\,dt\,\mathrm{s} z
 \longrightarrow
 B\,dt\,\mathrm{s} z\oplus B\,dt\,P.
\]
In particular this comparison is not an equivalence on
one-truncations, although the underlying relative complex
is quasi-isomorphic to the ordinary abelian algebroid $Bz$.
\end{corollary}

\begin{proof}
Here
$D^1_{\bas}(\g_B/B)=T_{B/\kk}\ltimes
                  (B\otimes_\kk\Der_\kk(\g))$:
the first summand differentiates coefficients and the second
is $B$-linear. The vertical complex of its cone realisation
is $B\otimes_\kk\mathbb U_{\dg}(\g)$.
Since the tangent module has rank one, the reduced CE complex
has only arity one, with zero wedge bracket.
Its based Maurer--Cartan space is abelian and has
\[
 \pi_1= B\,dt\otimes_\kk
             H^{-1}\mathbb U_{\dg}(\g).
\]
On the strict side the degree-zero central isotropy is $Bz$.
Proposition~\ref{prop:central-stabiliser-complex} identifies
its automorphisms with $B\,dt\,z$.
The explicit comparison sends this cocycle to its suspension
in the cone, so its image is the first summand and does not
contain the class of $dt\,P$.
\end{proof}

\begin{theorem}\label{thm:presentation-independence}
Equivalent projectable CE diagrams have equivalent formal
derived basic automorphism groups.
For their strict crossed modules, suppose in addition that
the comparison lifts to a zigzag of objectwise
quasi-isomorphisms of crossed modules, or to an equivalence
of full homotopy-coherent crossed data with compatible
realisations, as in
Corollary~\ref{cor:cartan-crossed-invariant}.
Then their realisations have equivalent derived spaces of
sections over $T=T_{B/\kk}$:
\[
 \Map^{\mathbb R}_{/T}(T,\mathbb U_{\dg}(\g/B))
 \simeq
 \Map^{\mathbb R}_{/T}(T,\mathbb U_{\dg}(\g'/B)).
\]
For an objectwise comparison this also identifies the underlying
mapping cones. If the inner maps are injective, the realisations
can be replaced by their ordinary quotient dg-Lie algebroids.
For CE isomorphisms these additional comparison conditions
hold by Proposition~\ref{prop:strict-ce-isomorphism}.
\end{theorem}

\begin{proof}
Proposition~\ref{prop:homotopy-invariance-der} proves the
derived-automorphism assertion.
Under the additional crossed comparison,
Corollary~\ref{cor:cartan-crossed-invariant} identifies the
realisations. Derived mapping spaces preserve this equivalence.
For strict comparisons, forgetting brackets gives the
corresponding equivalence of mapping cones.
If the inner maps are injective, the projections to the
quotients are quasi-isomorphisms of dg-Lie algebroids by
Proposition~\ref{prop:crossed-dg-realisation}.
The CE-isomorphism case follows by conjugation as proved above.
\end{proof}

The additional crossed comparisons are essential:
Proposition~\ref{prop:strict-CE-noninvariance} shows that
CE weak equivalence alone does not suffice.

\section{\'Etale base change and descent}\label{sec:base-change}

We distinguish \'etale base change on $S$ from \'etale
restriction on $X$. The former is proved using principal parts;
the latter is needed to glue the local classifications.

\begin{definition}
Let \(E\) be a perfect \(A\)-complex. Write
\(\operatorname{At}_{A/\kk}(E)\) for the derived Atiyah complex of first-order operators on \(E\), with symbol morphism
\[
 \sigma_E\from\operatorname{At}_{A/\kk}(E)\longrightarrow T_{A/\kk}.
\]
For \(\theta\in T_{A/\kk}\), the space of operators with symbol \(\theta\) is the homotopy fibre
\[
 \operatorname{Diff}^1_\theta(E,E)
 :=\operatorname{hofib}_\theta(\sigma_E).
\]
When non-empty, this is a torsor for the additive derived
space associated with \(\RHom_A(E,E)\). In general it
has no preferred origin, and it can be empty.
\end{definition}

\begin{lemma}\label{lem:first-order-etale-base-change}
Let \(B\to C\) be \'etale, put \(A_C=A\otimes_BC\), and set
\(E_C=A_C\otimes_A^{\mathbb L}E\). Then first principal parts and Atiyah complexes commute with base change:
\[
 A_C\otimes_A^{\mathbb L}\operatorname{At}_{A/\kk}(E)
 \simeq
 \operatorname{At}_{A_C/\kk}(E_C).
\]
The equivalence intertwines symbol maps. Consequently, for every symbol \(\theta\) and its base change \(\theta_C\), there is an equivalence of affine derived fibres
\[
 \operatorname{Diff}^1_\theta(E,E)\times_{\Spec A}\Spec A_C
 \simeq
 \operatorname{Diff}^1_{\theta_C}(E_C,E_C).
\]
\end{lemma}

\begin{proof}
First-order operators are represented by
$\RHom_A(P^1_{A/\kk}(E),E)$; the Atiyah complex is the
homotopy pullback imposing a scalar symbol
$T_{A/\kk}\to T_{A/\kk}\otimes_A\RHom_A(E,E)$.
For an \'etale algebra map, principal parts and cotangent complexes
base-change, since the relative cotangent complex vanishes.
Here $A$ is smooth and $E$ is perfect, so the principal-parts
complex is perfect. Derived Hom and the symbol pullback therefore
commute with this flat base change. This gives the Atiyah
comparison and its compatibility with symbols. Taking the affine
derived fibre at the base-changed symbol proves the last assertion.
\end{proof}

\begin{proposition}\label{prop:base-change-D}
Assume \(\g\) is strictly perfect over \(A\). For every \'etale morphism \(B\to C\), with
\(\g_C=A_C\otimes_A\g\), there is a natural equivalence of dg-Lie algebras
\[
 C\otimes_B^{\mathbb L}D^1_{\bas}(\g/B)
 \simeq
 D^1_{\bas}(\g_C/C),
\]
compatible with the symbol, the bracket and the inner map. Hence
\[
 C\otimes_B^{\mathbb L}\Uder(\g/B)
 \simeq
 \Uder(\g_C/C)
\]
as crossed modules.
\end{proposition}

\begin{proof}
Use the complex of all first-order operators with scalar symbol
and first impose strict projectability. Its defining fibre product
commutes with flat tensor product over $B$.
For a projectable pair $(\theta,\delta)$ the anchor defect
$x\mapsto\rho(\delta x)-[\theta,\rho(x)]$ is $A$-linear.
On its kernel the bracket defect
\[
 (x,y)\longmapsto
 \delta[x,y]-[\delta x,y]-(-1)^{|\delta||x|}[x,\delta y]
\]
is also $A$-bilinear: its failure of $A$-bilinearity is exactly
the anchor defect. Both defect maps are $B$-linear, since the
relative anchor annihilates $B$. The basic Lie derivations are
their successive ordinary kernels. They are not replaced here
by homotopy fibres.

Strict perfection identifies the relevant Hom complexes after
extension of scalars, and flatness of $C/B$ preserves these
kernels. This proves the comparison on complexes.
The Lie--Rinehart base change has bracket
\[
 [cu,c'v]=cc'[u,v]+c\,a(u)(c')v
               -(-1)^{|u||v|}c'\,a(v)(c)u,
\]
where the anchors extend uniquely through the \'etale map
$B\to C$. This formula agrees with the commutator of the
extended first-order operators. The inner maps and their actions
also agree, proving the crossed-module assertion.
\end{proof}

Consequently, for the prescribed pullback of the local model and
its anchor, the local derived splitting spaces and strict unfolding
groupoids commute with \'etale restriction of the base. This
concerns the restricted moduli problem, not commutation of global
sections with arbitrary base change.

\begin{lemma}
\label{lem:strict-etale-descent}
Let $\pi:X\to S$ be a smooth morphism of smooth schemes, and
let $\g\to T_{X/S}$ be a globally specified strict relative
dg-Lie algebroid whose underlying complex is perfect.  Fix the
projection of every transverse symbol to $\pi^{-1}T_S$.
For an \'etale affine hypercover $U_\bullet\to X$, the natural
restriction functors from the groupoid of strict unfoldings of
$\g$ and from the groupoid of basic extensions, respectively,
to their groupoids of descent data on $U_\bullet$ are equivalences.
The adjoint-symbol equivalence commutes with these restriction
functors.
\end{lemma}

\begin{proof}
An \'etale restriction $V\to X$ satisfies
$\Omega^1_{V/X}=0$.  The tangent and first-principal-parts
sequences therefore identify the restrictions of the absolute
and relative anchors and of the first-order operators on $\g$.
In particular a derivation with a prescribed transverse symbol
restricts uniquely as a first-order operator; its difference from
another such derivation has vertical symbol.  The Lie bracket,
its derivation action on $\g$, and the inner map commute with
restriction, since each is defined by local differential operators.

The underlying dg-modules of an unfolding, together with their
differentials and the maps in the relative-to-absolute extension,
satisfy effective \'etale descent.  Once the anchors are fixed,
the bracket is a first-order bidifferential operation: its
Leibniz identity determines the terms involving derivatives of
scalars.  Thus compatible brackets glue, and the Jacobi,
anchor and dg identities hold after gluing because they can be
checked \'etale locally.  The basic extensions are sheaves of
$\pi^{-1}\cO_S$-modules, rather than $\cO_X$-modules;
they descend in the \'etale topos, with their compatible
anchors and first-order operations.  The quotient identifying the
transverse tangent directions is part of the descent data, so
the base symbols agree on all intersections; no independent
choice of a symbol is made in each chart.  Descent for the
groupoids of objects and isomorphisms follows.  Since these are
groupoid-valued stacks, their \'etale descent also applies to
the hypercover.

Finally, the adjoint action and the scalar-extension
reconstruction of Theorem~\ref{thm:higher} use the restricted
anchors, brackets and module relations.  Both constructions
commute with \'etale restriction.  The local equivalences
therefore assemble into the claimed equivalence of descent
groupoids.
\end{proof}

\begin{theorem}
\label{thm:etale-descent-splittings}
Let $\pi:X\to S$ be a smooth morphism of smooth schemes.
For a globally specified strict relative dg-Lie algebroid
with perfect underlying complex, the groupoids of strict
unfoldings and basic extensions satisfy étale hyperdescent,
and their adjoint-symbol classifications descend.

Let $\cF/S$ be a relative derived foliation with cotangent
complex a vector bundle in degree zero, and put
$\g=(\bbL_{\cF/S})^\vee$, with the ordinary relative Lie
algebroid structure recovered from its mixed differential.
Fix $V=\Spec B\subset S$. On affine étale charts
$U=\Spec A\to X_V$, set
\[
 \mathscr S_{\cF,V}(U)=\mathscr S_{\g,V}(U)
  :=\Map^{\mathbb R}_{\mathrm{LieAlgd}^{\infty}_B/T_{B/\kk}}
       \bigl(T_{B/\kk},
         \mathbb U_{\dg}(\g(U)\to D^1_{\bas}(\g(U)/B))\bigr).
\]
This is a one-truncated étale hypersheaf and there are
natural equivalences
\begin{equation}\label{eq:ordinary-global-classification}
 \mathcal U_{\pi_V}(\cF|_{X_V})
 \simeq N\Unf^{\tr}_{\str}(\g|_{X_V}/V)
 \simeq\mathrm R\Gamma(X_V,\mathscr S_{\cF,V}).
\end{equation}
Thus any affine étale hypercover of $X_V$ computes the
intrinsic unfolding space from these local derived
section spaces. The comparisons commute with restriction
of $V$ and glue over $S$.
\end{theorem}

\begin{proof}
The first assertion is Lemma~\ref{lem:strict-etale-descent}.
For the second, Theorem~\ref{prop:intrinsic-comparison} recovers
$\g$ from $\cF/S$ and identifies its intrinsic unfoldings with
the strict unfolding groupoid, including all automorphisms.
Let $\mathscr G_{\g,V}(U)$ be that groupoid on $U$, with its
prescribed transverse quotient. Its nerve is a hypersheaf by the same lemma:
modules, anchors, differentials and brackets glue, as do
isomorphisms fixing the relative term and quotient.
A groupoid-valued stack is one-truncated: hypercover data
consist of objects in degree zero, isomorphisms in degree one
and their cocycle in degree two; higher compatibilities are
then forced. Effective descent for objects and isomorphisms
therefore computes hypercover descent as well.

Theorem~\ref{thm:ordinary-crossed-derived} gives an equivalence
\[
 N\mathscr G_{\g,V}(U)\simeq\mathscr S_{\g,V}(U)
\]
for every affine chart, including the empty case when the
basic anchor is not surjective.
The comparison is the one constructed in
Theorem~\ref{thm:strict-derived-comparison}, so restriction
preserves the adjoint operators, both defects, and their
gauge homotopies. First-order operators extend uniquely
through an \'etale coordinate restriction. Thus these local
comparisons commute with the restriction maps and their
compositions. The comparison therefore identifies the presheaves of spaces,
including their automorphisms.

It follows that $\mathscr S_{\g,V}$ is a one-truncated
hypersheaf. Taking its sections, or its homotopy
limit on a hypercover, and applying the intrinsic CE comparison
gives \eqref{eq:ordinary-global-classification}.
For an affine open restriction $V'\subset V$, the base
change is \'etale. Proposition~\ref{prop:base-change-D}
identifies the crossed modules after restriction;
the free resolution of $T_{B/\kk}$ and its simplicial frame
commute with this base change. The explicit comparison of
strict extensions also commutes with it.
The local equivalences therefore agree on overlaps of
the affine base charts and glue.
No pushforward of a presentation-independent Lie-algebroid
stack is required for this construction.
\end{proof}

\section{Existence of unfoldings and flat transverse lifts}
\label{sec:AKS}\label{sec:existence}

A flat splitting of the effective quotient classifies an
unfolding; a splitting of the basic extension of an existing
unfolding chooses horizontal lifts. We keep these two
obstruction problems distinct.

\subsection{Linear and curvature obstructions}

Before asking for a flat splitting, one must know whether
the anchor admits a section compatible with the differential.
The connecting morphism gives an Atiyah-type obstruction
\cite{Atiyah57}. After choosing a lift, its differential and
bracket defects test flatness.

\begin{definition}
Let $\mathfrak u_S\to\bbT_S$ be a specified transverse
Lie algebroid on $S$, and put
$\mathfrak K_{\cF/S}=\fib(\mathfrak u_S\to\bbT_S)$.
When a derived pushforward supplies $\mathfrak u_S$, its existence
and anchor are included in the hypotheses.
The exact triangle
\[
 \mathfrak K_{\cF/S}\longrightarrow\mathfrak u_S
 \longrightarrow\bbT_S
 \xrightarrow{\kappa_{\cF/S}}\mathfrak K_{\cF/S}[1]
\]
defines the Atiyah--Kodaira--Spencer morphism.
The same definition applies to a specified realisation of the
non-effective crossed module.
\end{definition}

\begin{proposition}\label{prop:AKS-splitting}
The class $\kappa_{\cF/S}$ vanishes if and only if the underlying
morphism of complexes admits a section in the derived category.
Suppose, on a cover computing the relevant hypercohomology, that
the anchor is represented by a degreewise surjection and admits
local graded sections $\sigma_i$. Then $\kappa_{\cF/S}$ is
represented in the total \v{C}ech--Hom complex by
\[
 \tau_i=d_{\mathfrak u}\sigma_i-\sigma_i d_{\bbT_S},
 \qquad c_{ij}=\sigma_j-\sigma_i.
\]
Only when the $\sigma_i$ are chain maps is it represented by the
\v{C}ech one-cocycle $c_{ij}$ alone.
\end{proposition}

\begin{proof}
Apply $\RHom(\bbT_S,-)$ to the triangle. The identity of
$\bbT_S$ lifts to a morphism into $\mathfrak u_S$ precisely when
its image $\kappa_{\cF/S}$ vanishes.

For the local description, let $\partial$ denote the internal
Hom differential. Then
$\partial\tau_i=0$,
$\tau_j-\tau_i=\partial c_{ij}$, and $\delta c=0$.
Thus $(\tau,c)$ is closed for the total differential
$\delta+(-1)^p\partial$ on \v{C}ech degree $p$.
It is the connecting cocycle obtained by lifting the identity
locally to the $\sigma_i$. Changing the graded sections changes
it by a total coboundary. If all $\tau_i$ vanish, this reduces to
the usual \v{C}ech cocycle of differences of local chain splittings.
\end{proof}

Let \(\sigma\) be a graded splitting of the anchor. Its failure to commute with the internal differential is
\[
 \tau_\sigma(\xi)
 :=d_{\mathfrak u}\sigma(\xi)-\sigma(d_{\bbT_S}\xi),
\]
and its bracket curvature is
\[
 F_\sigma(\xi,\eta)
 :=[\sigma(\xi),\sigma(\eta)]-\sigma([\xi,\eta]).
\]
These have bidegrees \((1,1)\) and \((2,0)\), respectively, and combine into the total curvature
\[
 \mathcal F_\sigma:=F_\sigma-\tau_\sigma
 \in
 \Tot^2 C^\bullet_{\mathrm{CE}}
 \bigl(\bbT_S;\mathfrak K_{\cF/S}\bigr).
\]
The minus sign is dictated by our totalisation convention
\(D\alpha=d_{\mathrm{CE}}\alpha+(-1)^p d_{\mathfrak K}\alpha\)
for a form of exterior degree \(p\). In particular, for
\(\alpha\in\Omega^1_S\otimes\mathfrak K^0\), the internal
component of \(D\alpha\) is \(-d_{\mathfrak K}\alpha\).

\begin{proposition}\label{prop:AKS-bianchi}
In a strict dg-Lie-algebroid extension over a smooth affine base, the covariant total differential \(D_\sigma\) determined by a graded splitting satisfies
$
 D_\sigma^2=[\mathcal F_\sigma,-],
 \ 
 D_\sigma\mathcal F_\sigma=0.
$
If the vertical kernel is abelian, \(D_\sigma\) is a differential and \(\mathcal F_\sigma\) is a total cocycle.
\end{proposition}

\begin{proof}
Write $\mathfrak K$ for the strict vertical kernel. The commutator of the internal differential with the horizontal connection is
\([d_{\mathfrak K},\nabla^\sigma_\xi]=\operatorname{ad}_{\tau_\sigma(\xi)}\),
whereas its horizontal curvature is
\(R_{\nabla^\sigma}(\xi,\eta)=\operatorname{ad}_{F_\sigma(\xi,\eta)}\).
For the total differential
\(D_\sigma=d_{\mathrm{CE},\nabla^\sigma}+(-1)^p d_{\mathfrak K}\)
on exterior degree \(p\), the mixed commutator therefore contributes
\(-\operatorname{ad}_{\tau_\sigma}\), while the horizontal square contributes
\(\operatorname{ad}_{F_\sigma}\). This gives
\(D_\sigma^2=\operatorname{ad}_{F_\sigma-\tau_\sigma}\).
Furthermore the ordinary Bianchi identity gives
\(d_{\mathrm{CE},\nabla^\sigma}F_\sigma=0\);
compatibility of the differential with the bracket gives
\(d_{\mathfrak K}F_\sigma=d_{\mathrm{CE},\nabla^\sigma}\tau_\sigma\);
and \(d_{\mathfrak K}\tau_\sigma=0\).
These three relations prove \(D_\sigma(F_\sigma-\tau_\sigma)=0\).
If the kernel is abelian, the adjoint action of the total curvature vanishes.
\end{proof}

The class $\kappa_{\cF/S}$ of
Proposition~\ref{prop:AKS-splitting} determines whether the anchor
has a section in the derived category. Bracket compatibility
is a separate condition: for a graded splitting $\sigma$,
flatness is measured by $\mathcal F_\sigma$. If
$\sigma'=\sigma+\alpha$, then
\[
 \mathcal F_{\sigma'}
 =\mathcal F_\sigma+D_\sigma\alpha
  +\tfrac12[\alpha,\alpha].
\]
When the vertical kernel is abelian, $D_\sigma$ is independent
of the splitting and the last term vanishes. Thus
\[
 [\mathcal F_\sigma]\in
 H^2\Tot\overline C^\bullet_{\mathrm{CE}}
 \bigl(\bbT_S;\mathfrak K_{\cF/S}\bigr)
\]
is a well-defined obstruction class, independent of $\sigma$.

\subsection{Abelian extensions and their homotopy two-type}\label{subsec:abelian-extensions-two-type}

For abelian kernels, Lie-algebroid extensions are classified
cohomologically \cite{ABR18}. The strict extension groupoid
retains one-cocycles as automorphisms; allowing two-morphisms
also identifies arrows that differ by a coboundary.

Let $B$ be a smooth $\kk$-algebra, put $T=T_{B/\kk}$, and let $Z$ be a
$B$-module equipped with a flat $T$-connection $\nabla$.  We regard $Z$
as an abelian Lie algebroid with zero anchor.  Write
\[
 C^p=C^p_{\mathrm{CE}}(T;Z)
   =\operatorname{Hom}_B(\wedge_B^pT,Z)
   \simeq\Omega^p_{B/\kk}\otimes_B Z,
 \qquad d=d_{\mathrm{CE},\nabla}.
\]
The complex $C^\bullet$ is concentrated in non-negative degrees.  We
consider extensions
\begin{equation}\label{eq:abelian-ext-sequence}
 0\longrightarrow Z\longrightarrow E\longrightarrow T
 \longrightarrow0
\end{equation}
of Lie algebroids over $B$, with kernel bracket zero, zero kernel anchor,
and prescribed induced $T$-action $\nabla$ on $Z$.  Morphisms are the
identity on both $Z$ and $T$.  The $B$-module underlying
\eqref{eq:abelian-ext-sequence} splits since $T$ is projective.

\begin{theorem}
\label{thm:abelian-extension-two-type}
In the setting above the following hold.
\begin{enumerate}[label=(\roman*)]
\item The groupoid of extensions \eqref{eq:abelian-ext-sequence} is
equivalent to the action groupoid whose objects are $Z^2(C^\bullet)$
and whose arrows $\omega\to\omega'$ are elements $a\in C^1$ with
$\omega'=\omega+da$.  Consequently its isomorphism classes are
$H^2(C^\bullet)$ and the strict automorphism group of an extension is
$Z^1(C^\bullet)$.
\item There is a canonical abelian two-groupoid enhancing this explicit
presentation: its objects and arrows are as in~\emph{(i)}, and a
$2$-morphism $a\Rightarrow a'$ is an element $b\in C^0$ satisfying
$a'-a=db$.  Its classifying space is the Picard two-type
\[
  \mathcal E_{(2)}(T,Z)
    \simeq \mathrm{DK}\bigl(\tau_{\leq0}(C^\bullet[2])\bigr),
\]
where $\tau_{\leq0}$ is the standard cohomological truncation and
$\mathrm{DK}$ denotes Dold--Kan after reindexing in non-negative
homological degrees.  At every object,
\[
 \pi_0\mathcal E_{(2)}(T,Z)\simeq H^2(C^\bullet),\qquad
 \pi_1\mathcal E_{(2)}(T,Z)\simeq H^1(C^\bullet),\qquad
 \pi_2\mathcal E_{(2)}(T,Z)\simeq H^0(C^\bullet).
\]
\item The map from the nerve of the strict extension groupoid to the
one-truncation of $\mathcal E_{(2)}(T,Z)$ induces the identity on
$H^2(C^\bullet)$ and the quotient homomorphism
\[
 Z^1(C^\bullet)\longrightarrow
 Z^1(C^\bullet)/dC^0=H^1(C^\bullet)
\]
on automorphism groups.  It is an equivalence of one-types if and only
if $d:C^0\to C^1$ is zero.
\end{enumerate}
\end{theorem}

\begin{proof}
Choose a $B$-linear section of the underlying module sequence.  The
anchor and prescribed kernel action force the bracket, after identifying
$E$ with $Z\oplus T$, to have the form
\begin{equation}\label{eq:abelian-extension-bracket}
 [(z,\xi),(w,\eta)]_\omega
   =\bigl(\nabla_\xi w-\nabla_\eta z
                    +\omega(\xi,\eta),[\xi,\eta]\bigr).
\end{equation}
The Lie--Rinehart identity follows from the connection Leibniz rule.
The Jacobi identity on three lifted elements of $T$ is exactly
$d\omega=0$; the other components are the flatness of $\nabla$ and
the abelianness of $Z$.  Conversely, each $\omega\in Z^2(C^\bullet)$
therefore defines an extension.  A morphism which fixes the kernel and
quotient must, for a unique $a\in C^1$, take the form
\[
 \Phi_a:E_\omega\longrightarrow E_{\omega+da},\qquad
       (z,\xi)\longmapsto(z-a(\xi),\xi).
\]
Substitution in \eqref{eq:abelian-extension-bracket} proves the bracket
identity.  Composition corresponds to addition in $C^1$.  This proves
\emph{(i)}, including the formula for strict automorphisms.

For \emph{(ii)}, add $2$-morphisms $b\in C^0$ with $a'-a=db$,
with vertical and horizontal composition given by addition.  The
identities $d^2=0$ show that these operations define a strict abelian
$2$-groupoid.  Its underlying three-term chain complex, in homological
degrees $2,1,0$, is
\[
 C^0\xrightarrow{\ d\ }C^1\xrightarrow{\ d\ }Z^2(C^\bullet).
\]
The Dold--Kan correspondence identifies its classifying simplicial
abelian group with the displayed truncation of $C^\bullet[2]$.
Taking homology in degrees $0,1,2$ gives respectively
$H^2(C^\bullet)$, $H^1(C^\bullet)$ and $H^0(C^\bullet)$.

Finally, one-truncation identifies two arrows precisely when they differ
by $db$.  Thus its isotropy is $Z^1/dC^0$, whereas strict isotropy is
$Z^1$.  The comparison induces an equivalence on components, and an
equivalence on all based loop groups exactly when $dC^0=0$.
This proves \emph{(iii)}.
\end{proof}

Changing a $B$-linear splitting of
\eqref{eq:abelian-ext-sequence} by $a\in C^1$ changes its
curvature by $da$. Thus the extension admits a flat
splitting exactly when $[\omega]=0$; when non-empty, its
set of strict splittings is a torsor under $Z^1(C^\bullet)$.
The two-homotopies of
Theorem~\ref{thm:abelian-extension-two-type} identify
differences in $dC^0$, yielding a torsor under $H^1(C^\bullet)$.

\subsection{Abelian vertical kernels over an affine base}

Consider the affine quotient Lie algebroid
\[
 a_{\uuder}\from\uuder(\g/B)\longrightarrow T_{B/\kk}
\]
with vertical kernel \(\kfr\), and assume that its anchor is
surjective. Smoothness of $B$ makes $T_{B/\kk}$ projective,
so the anchor admits a graded $B$-linear section.

For a graded section \(\sigma\), define \(\tau_\sigma\), \(F_\sigma\) and
\(\mathcal F_\sigma\) as above. Twisting by a total degree-one
cochain \(\alpha\) gives the curvature
\[
 \mathcal F_\sigma^\alpha
 :=\mathcal F_\sigma+D_\sigma\alpha
  +\frac12[\alpha,\alpha].
\]
For a strict correction, \(\alpha\) lies in
\(\Omega^1_{B/\kk}\otimes_B\kfr^0\), and this is
\(\mathcal F_{\sigma+\alpha}\). The derived problem permits
all total degree-one components of the reduced complex.

\begin{theorem}\label{thm:abelian-obstruction}
Suppose that the anchor of the effective quotient Lie algebroid is
surjective and that its vertical dg-Lie kernel \(\kfr\) is abelian.
For a graded section \(\sigma\), the total curvature determines a
class
\[
 \mathfrak o(\g/B):=[\mathcal F_\sigma]\in
 H^2\!\left(\Tot\overline C^\bullet_{\mathrm{CE}}
 (T_{B/\kk};\kfr)\right).
\]
It is independent of the graded section.  It vanishes if and only if
the \emph{reduced Chevalley--Eilenberg} flatness equation has a
solution of total degree one.  When solutions exist, they form a
torsor under the degree-one cocycles of the reduced complex; their
homotopy classes in this abelian Maurer--Cartan model form a torsor
under its first cohomology group.

If \(\kfr\) is concentrated in cohomological degree zero, these
solutions are precisely the strict flat splittings.  In that case
the set of strict splittings, without a gauge quotient, is a torsor
under \(Z^1_{\mathrm{CE}}(T_{B/\kk};\kfr)\).  For dg kernels satisfying the hypotheses of
Theorem~\ref{thm:derived-sections-smooth-base}, this reduced
Maurer--Cartan equation, including its curvature term, computes
the derived mapping space of sections.
\end{theorem}

\begin{proof}
For abelian \(\kfr\), the adjoint action of the curvature vanishes,
so \(D_\sigma^2=0\).  Both components of \(\mathcal F_\sigma\)
have positive Chevalley--Eilenberg degree, and the Bianchi identity
shows that they form a cocycle in the reduced total complex.  If
\(\sigma'=\sigma+\alpha\) is another graded section, then
\(\alpha\in C^1_{\mathrm{CE}}(T;\kfr^0)\) and
\[
 \mathcal F_{\sigma'}-\mathcal F_\sigma=D_\sigma\alpha.
\]
Thus the cohomology class is independent of the section.  It
vanishes precisely when a total degree-one element \(\alpha\) in
the reduced complex satisfies
\(\mathcal F_\sigma+D_\sigma\alpha=0\).  Two solutions differ
by a closed degree-one cochain.  In the associated abelian
simplicial Maurer--Cartan space, homotopies identify differences
which are boundaries in the reduced complex.  Finally, when
\(\kfr\) has degree zero only, every total degree-one element
of the reduced complex is a one-cochain of internal degree zero;
its flatness equation is the strict Lie-algebroid splitting
equation.  There are no degree-zero cochains in the reduced
complex, and the set of solutions is a torsor under closed
one-cochains.\end{proof}

\section{The de Rham algebra of an unfolding}\label{sec:normal-form}

Let $(\e,\Theta)$ be the basic extension of a strict
unfolding, and choose a graded $B$-linear section
$\sigma:T_{B/\kk}\to\e$.
Write
\[
 \nabla^\sigma_\xi x=[\sigma(\xi),x],\qquad
 \tau_\sigma(\xi)=d_\e\sigma(\xi),\qquad
 F_\sigma(\xi,\eta)=[\sigma(\xi),\sigma(\eta)]
                           -\sigma([\xi,\eta]).
\]
The operators are defined by a lift to the basic extension.
A section of the quotient by inner derivations does not, by
itself, act on the relative CE algebra.

\begin{proposition}\label{prop:normal-form}
Suppose $\g$ is strictly perfect. Under the graded splitting
$\widetilde\g=\g\oplus(A\otimes_BT_{B/\kk})$ there is an
isomorphism of graded algebras
\[
 \CE^*(\widetilde\g)^{\mathrm{gr}}
 \cong \CE^*(\g)^{\mathrm{gr}}
                  \otimes_B\DR(B/\kk)^{\mathrm{gr}}.
\]
Its internal and mixed differentials are the duals of the
differential, anchor and brackets
\begin{align*}
 d\,\sigma(\xi)&=\tau_\sigma(\xi),&
 \widetilde\rho(\sigma(\xi))&=\sigma_{\e}(\sigma(\xi)),\\
 [\sigma(\xi),x]&=\nabla^\sigma_\xi x,&
 [\sigma(\xi),\sigma(\eta)]
   &=\sigma([\xi,\eta])+F_\sigma(\xi,\eta),
\end{align*}
together with the original structure on $\g$ and the
Lie--Rinehart rule for coefficients in $A$.
Thus the normal form includes both defects, even when the
corresponding quotient splitting is flat. 
For an ordinary algebroid, after realisation this is the
usual CE differential
\[
 d_{\widetilde\g}
   =d_{\g}+d_{\nabla^\sigma}-\imath_{F_\sigma}.
\]
Here $\imath_{F_\sigma}$ vanishes on functions and base
one-forms, and on a relative one-form $\lambda$ is the base
two-form
$(\xi,\eta)\mapsto\lambda(F_\sigma(\xi,\eta))$.
It extends as a total-degree-one derivation.
A flat section of the actual extension makes $F_\sigma=0$
and gives the semidirect-product normal form.
\end{proposition}

\begin{proof}
Scalar extension reconstructs the graded $A$-module from
$\g\oplus T_{B/\kk}$, so dualisation and the symmetric algebra
give the displayed graded isomorphism.
The four formulas and the structure on $\g$ specify the
differential, anchor and bracket on generators. The extension
identities prove $d^2=0$, compatibility with the anchor and
Jacobi; dualising therefore gives the CE internal and mixed
identities, with all suspension signs fixed by the CE convention.

For the ordinary formula, the differential of a relative
one-form $\lambda$, evaluated on two lifted base fields, is
$-\lambda(F_\sigma(\xi,\eta))$ by the CE rule.
On one base field and one relative field it is the covariant
differential, and on two relative fields it is $d_\g\lambda$.
On a function it is the sum of the relative derivative and the
symbol derivative along $\sigma$. These evaluations determine
the displayed derivation on all generators.
Its square vanishes by the same identities.
If the curvature contraction is omitted, the square of
$d_\g+d_{\nabla^\sigma}$ is $L_{F_\sigma}$, by Cartan's
formula and
$(d_{\nabla^\sigma})^2=L_{F_\sigma}$.
The contraction term is therefore necessary unless this
curvature action vanishes. For a dg extension the additional
internal component is dual to $\tau_\sigma$, already included
in the generator description.
\end{proof}

In a complete convolution $L_\infty$-model, curvature
induces the defect of the mixed differential
\cite{Hinich01,Pridham10,Nuiten}. Central curvature can act
trivially, so vanishing of this defect need not force the
curvature to vanish.

\section{Leaf-space interpretation in the representable case}\label{sec:leaf-space}

When the relative leaves are the fibres of a smooth
morphism $q:X\to Y$, transverse vector fields that descend
along $q$ determine vector fields on $Y$. If $Y\to S$ is
smooth, the descended flat action is an integrable
Ehresmann connection on the leaf-space fibration.

\begin{hypothesis}
Let
\[
\begin{tikzcd}
X \arrow[r,"q"] \arrow[dr,"\pi"'] & Y \arrow[d,"p"] \\
& S
\end{tikzcd}
\]
be smooth morphisms of smooth schemes, with \(q\) faithfully flat and \(\pi=p\circ q\). Let \(\cF/S\) be the relative foliation on \(X/S\) whose tangent complex is \(T_{X/Y}\). Equivalently, \(\cF/S\) is induced by the smooth leaf-space map \(q\).
\end{hypothesis}

\begin{proposition}\label{prop:projectable-leaf-space-fields}
The normal sheaf of the foliation induced by \(q\) is the
\(\cO_X\)-module
$
 T_X/T_{X/Y}\simeq q^*T_{Y/\kk}.
$
The sheaf of vector fields projectable along \(q\) is the
fibre product
\[
 T_X^{\mathrm{proj}/Y}:=
 T_X\times_{\underline{\Der}_{\kk}(q^{-1}\cO_Y,\cO_X)}q^{-1}T_{Y/\kk}.
\]
Here the first arrow is restriction of derivations to \(q^{-1}\cO_Y\). The second arrow is obtained from a derivation of \(\cO_Y\) by applying \(q^{-1}\) and composing with \(q^{-1}\cO_Y\to\cO_X\). There is an exact sequence of sheaves of \(\kk\)-Lie algebras
\[
 0\longrightarrow T_{X/Y}\longrightarrow T_X^{\mathrm{proj}/Y}
 \longrightarrow q^{-1}T_{Y/\kk}\longrightarrow 0.
\]
Thus the projectable transverse symmetry sheaf is \(q^{-1}T_{Y/\kk}\), not \(q^*T_{Y/\kk}\). The latter is obtained from the former only after extension of scalars from \(q^{-1}\cO_Y\) to \(\cO_X\), and this extension forgets the projectability condition.

The transverse anchor for projectable symmetries is the inverse image of
\[
 dp\from T_{Y/\kk}\longrightarrow p^*T_{S/\kk}.
\]
Consequently, we consider transverse splittings of the normal sequence whose classes descend from \(Y\), namely those obtained by pulling back splittings of the tangent sequence of \(p\from Y\to S\) along \(q\).
\end{proposition}

\begin{proof}
The first assertion is the usual exact sequence of tangent sheaves for the smooth morphism \(q\):
\[
 0\longrightarrow T_{X/Y}\longrightarrow T_X
 \longrightarrow q^*T_{Y/\kk}\longrightarrow 0.
\]
It identifies the \(\cO_X\)-linear normal sheaf of the foliation with \(q^*T_{Y/\kk}\). Projectability imposes the additional condition that the normal
class be induced by a vector field on \(Y\).

By definition, a section of \(T_X^{\mathrm{proj}/Y}\) is a pair \((V,W)\), where \(V\) is a vector field on \(X\), \(W\) is a local vector field on \(Y\), and the derivation induced by \(V\) on functions pulled back from \(Y\) is the pullback of \(W\). Projection to the second component gives a morphism of sheaves of \(\kk\)-Lie algebras
\[
 T_X^{\mathrm{proj}/Y}\longrightarrow q^{-1}T_{Y/\kk}.
\]
Its kernel consists of those pairs \((V,0)\) for which \(V\) annihilates \(q^{-1}\cO_Y\); this is \(T_{X/Y}\). Since \(q\) is smooth, the morphism is locally surjective: locally there is an \'etale coordinate map from $X$ to an open subset of $Y\times\mathbb A^r$. A vector field $(W,0)$ on that product lifts uniquely through the \'etale map. Hence
\[
 0\to T_{X/Y}\to T_X^{\mathrm{proj}/Y}\to q^{-1}T_{Y/\kk}\to0
\]
is exact as a sequence of sheaves of \(\kk\)-Lie algebras.
The fibre product cannot be replaced by the normaliser condition: \([V,T_{X/Y}]\subset T_{X/Y}\) is satisfied by every vector field when \(q\) is \'{e}tale, and therefore cannot by itself characterise vector fields related to vector fields on \(Y\). For example, for \(q\from\mathbb G_m\to\mathbb G_m\), \(y=x^2\), the relative tangent sheaf is zero, but \(x^2\partial_x\) is not related to a vector field on the target, since its action on \(y\) is \(2x^3\), not the pullback of a regular function of \(y\).
The distinction between \(q^{-1}T_{Y/\kk}\) and \(q^*T_{Y/\kk}\) is also visible in the elementary projection \(\mathbb A^2_{x,y}\to\mathbb A^1_y\). The normal module contains \(x\partial_y\), but no vector field representing this normal class is projectable, since
\[
 [x\partial_y+g(x,y)\partial_x,\partial_x]
 =-\partial_y-\frac{\partial g}{\partial x}\partial_x
\]
is not vertical. Thus projectability is not \(\cO_X\)-linear.
The map on projectable transverse symmetries is the inverse image of \(dp\). Extending scalars gives the normal morphism
\[
 q^*T_{Y/\kk}\longrightarrow q^*p^*T_{S/\kk}=\pi^*T_{S/\kk}.
\]
We restrict attention to transverse normal splittings obtained by pulling back a splitting of
$
 T_{Y/\kk}\longrightarrow p^*T_{S/\kk}.
$

Faithfully flat descent for quasi-coherent sheaves along \(q\) identifies such descended splittings with splittings on \(Y\), as claimed.
\end{proof}

\begin{theorem}\label{thm:leaf-space-connection}
In the representable leaf-space situation, transversal unfoldings of \(\cF/S\) obtained by pullback from \(Y\) are equivalent to flat splittings
$
 p^*T_{S/\kk}\longrightarrow T_{Y/\kk}
$
of the tangent sequence of \(p\from Y\to S\). Equivalently, they are integrable Ehresmann connections on the leaf-space fibration \(Y\to S\), pulled back to \(X\).
\end{theorem}

\begin{proof}
By Proposition~\ref{prop:projectable-leaf-space-fields}, a
projectable transverse splitting induced from $Y$ is the
pullback of a splitting of $T_{Y/\kk}\to p^*T_{S/\kk}$;
faithfully flat descent along $q$ makes the latter unique.
The bracket curvature of the pullback is the pullback of
the curvature on $Y$, and faithful flatness detects its
vanishing. Thus the descended splitting is flat exactly
when its pullback is flat. The adjoint-symbol classification
then identifies these flat splittings with the unfoldings
obtained by pullback, as asserted.
\end{proof}

For formal and analytic leaf spaces, representability and
the comparison of transverse symmetries require additional
hypotheses; see \cite[Corollary~2.3.3.3 and
Theorem~6.2.2.1]{TV} and \cite{TVInt}. The theorem above
concerns the specified smooth, representable situation.

\subsection*{Acknowledgements.}
  The author acknowledges the partial support of the Universit\`a degli Studi di Bari. He is grateful to Ariel Molinuevo, Federico Quallbrunn and Simone Noja for many stimulating  conversations. He is a member of GNSAGA, a national research group of INdAM.

\subsection*{AI disclosure}
The mathematical ideas, constructions, proofs, and results presented in this paper are the author’s own. AI-assisted tools were used for proofreading, editorial suggestions, and the preliminary exploration of possible proof strategies. The author takes full responsibility for the mathematical correctness and content of the manuscript.

\end{document}